\documentclass[review]{elsarticle}
\usepackage{srcltx}
\usepackage{eurosym}
\usepackage{mathtools}
\usepackage{amsmath}
\usepackage{amsfonts}
\usepackage{amssymb}
\usepackage{amsthm}
\usepackage{graphicx}
\usepackage{mathrsfs}
\usepackage{xcolor}
\usepackage{exscale}
\usepackage{latexsym}
\usepackage[colorlinks,plainpages=true,pdfpagelabels,hypertexnames=true,colorlinks=true,pdfstartview=FitV,linkcolor=blue,citecolor=red,urlcolor=black]{hyperref}
\PassOptionsToPackage{unicode}{hyperref}
\PassOptionsToPackage{naturalnames}{hyperref}
\usepackage{enumerate}
\usepackage[shortlabels]{enumitem}
\usepackage{bookmark}
\usepackage{wasysym}
\usepackage{esint}
\usepackage[ddmmyyyy]{datetime}
\usepackage[margin=2.2cm]{geometry}
\makeatletter
\g@addto@macro\normalsize{%
  \setlength\abovedisplayskip{4pt}
  \setlength\belowdisplayskip{4pt}
  \setlength\abovedisplayshortskip{4pt}
  \setlength\belowdisplayshortskip{4pt}
}
\numberwithin{equation}{section}
\everymath{\displaystyle}
\newtheorem{theorem}{Theorem}[section]
\newtheorem{lemma}[theorem]{Lemma}
\newtheorem{corollary}[theorem]{Corollary}
\newtheorem{proposition}[theorem]{Proposition}

\newtheorem{definition}[theorem]{Definition}
\newtheorem{remark}[theorem]{Remark}        
  
\numberwithin{equation}{section}
\def\aa{\mathcal{A}}
\newcommand{\bb}{\mathcal{B}}
\newcommand{\bbb}{\overline{\mathcal{B}}}
\newcommand{\lamot}{\La_0,\La_1}

\newcommand{\pa}{\partial}
\newcommand{\ddt}[1]{\frac{d#1}{dt}}
\newcommand{\dds}[1]{\frac{d#1}{ds}}

\newcommand{\vlh}{\lsbt{v}{\la,h}}
\newcommand{\vl}{\lsbt{v}{\la}}
\newcommand{\elam}{\lsbo{E}{\lambda}}
\newcommand{\bff}{{\bf f}}

\newcommand{\mm}{\mathcal{M}}
\newcommand{\pp}{p(\cdot)}
\newcommand{\qq}{q(\cdot)}

\newcommand{\sss}{s(\cdot)}

\newcommand{\modp}{\om_{\pp}}
\newcommand{\modq}{\om_{\qq}}

\newcommand{\plog}{p^{\pm}_{\log}}
\newcommand{\qlog}{q^{\pm}_{\log}}
\newcommand{\rlog}{r^{\pm}_{\log}}
\newcommand{\slog}{s^{\pm}_{\log}}
\newcommand{\logh}{\log^{\pm}}
\newcommand{\tQ}{\tilde{Q}}
\newcommand{\hQ}{\hat{Q}}

\newcommand{\tom}{\tilde{\Om}}

\newcommand{\tx}{\tilde{x}}

\newcommand{\tlt}{\tilde{t}}

\newcommand{\tz}{\tilde{z}}

\makeatletter
\newcommand{\vo}{\vec{o}\@ifnextchar{^}{\,}{}}
\makeatother
\def\Yint#1{\mathchoice
    {\YYint\displaystyle\textstyle{#1}}%
    {\YYint\textstyle\scriptstyle{#1}}%
    {\YYint\scriptstyle\scriptscriptstyle{#1}}%
    {\YYint\scriptscriptstyle\scriptscriptstyle{#1}}%
      \!\iint}
\def\YYint#1#2#3{{\setbox0=\hbox{$#1{#2#3}{\iint}$}
    \vcenter{\hbox{$#2#3$}}\kern-.50\wd0}}
\def\longdash{-\mkern-9.5mu-} 
\def\tiltlongdash{\rotatebox[origin=c]{18}{$\longdash$}}
\def\fiint{\Yint\tiltlongdash}
\def\Xint#1{\mathchoice
    {\XXint\displaystyle\textstyle{#1}}%
    {\XXint\textstyle\scriptstyle{#1}}%
    {\XXint\scriptstyle\scriptscriptstyle{#1}}%
    {\XXint\scriptscriptstyle\scriptscriptstyle{#1}}%
      \!\int}
\def\XXint#1#2#3{{\setbox0=\hbox{$#1{#2#3}{\int}$}
    \vcenter{\hbox{$#2#3$}}\kern-.50\wd0}}
\def\hlongdash{-\mkern-13.5mu-}
\def\tilthlongdash{\rotatebox[origin=c]{18}{$\hlongdash$}}
\def\hint{\Xint\tilthlongdash}
\makeatletter
\def\namedlabel#1#2{\begingroup
   \def\@currentlabel{#2}%
   \label{#1}\endgroup
}
\makeatother

\makeatletter
\def\ps@pprintTitle{%
 \let\@oddhead\@empty
 \let\@evenhead\@empty
 \def\@oddfoot{}%
 \let\@evenfoot\@oddfoot}
\makeatother
\makeatletter
\newcommand{\rmh}[1]{\mathpalette{\raisem@th{#1}}}
\newcommand{\raisem@th}[3]{\hspace*{-1pt}\raisebox{#1}{$#2#3$}}
\makeatother

\newcommand{\lsb}[2]{#1_{\rmh{-3pt}{#2}}}
\newcommand{\lsbo}[2]{#1_{\rmh{-1pt}{#2}}}
\newcommand{\lsbt}[2]{#1_{\rmh{-2pt}{#2}}}
\newcommand{\redref}[2]{\texorpdfstring{\protect\hyperlink{#1}{\textcolor{black}{(}\textcolor{red}{#2}\textcolor{black}{)}}}{}}
\newcommand{\redlabel}[2]{\hypertarget{#1}{\textcolor{black}{(}\textcolor{red}{#2}\textcolor{black}{)}}}
\newcounter{desccount}
\newcommand{\descitem}[2]{\item[#1] \refstepcounter{desccount}\label{#2}}
\newcommand{\descref}[2]{\hyperref[#1]{\textnormal{\textcolor{black}{(}\textcolor{blue}{\bf #2}\textcolor{black}{)}}}}

\newcommand{\dref}[2]{\hyperref[#1]{\textcolor{black}{(}\textcolor{blue}{\bf #2}\textcolor{black}{)}}}

\newcommand{\tvh}{\tilde{v}_h}

\newcommand{\tg}{\tilde{g}}

\newcommand{\tv}{\tilde{v}}

\newcommand{\ty}{\tilde{y}}

\newcommand{\mfx}{\mathfrak{x}}

\newcommand{\mfz}{\mathfrak{z}}

\newcommand{\mfs}{\mathfrak{s}}
\newcommand{\mft}{\mathfrak{t}}

\newcommand{\mfi}{\mathfrak{X}}

\newcommand{\mcf}{\mathcal{F}}
\newcommand{\mcg}{\mathcal{G}}

\newcommand{\mci}{\mathcal{I}}

\newcommand{\bm}{{\bf M}}
\newcommand{\br}{{\bf R}}
\newcommand{\bs}{{\bf S}}

\newcommand\RR{\mathbb{R}}

\newcommand\NN{\mathbb{N}}
\newcommand{\al}{\alpha}
\newcommand{\be}{\beta}
\newcommand{\ga}{\gamma}
\newcommand{\de}{\delta}
\newcommand{\ve}{\varepsilon}

\newcommand{\tht}{\theta}
\newcommand{\ep}{\epsilon}
\newcommand{\ka}{\kappa}

\newcommand{\om}{\omega}
\newcommand{\la}{\lambda}
\newcommand{\vt}{\vartheta}
\newcommand{\Ga}{\Gamma}

\newcommand{\Th}{\Theta}

\newcommand{\Om}{\Omega}

\newcommand{\La}{\Lambda}
\newcommand{\tTh}{{\Upsilon}}
\DeclareMathOperator{\dv}{div}
\DeclareMathOperator{\spt}{spt}

\DeclareMathOperator{\diam}{diam}

\newcommand{\avg}[2]{ \left\langle{#1}\right\rangle_{#2}}
\newcommand{\iprod}[2]{\langle #1 \ ,  #2\rangle}

\newcommand{\abs}[1]{\left| #1\right|}
\newcommand{\gh}[1]{\left( #1\right)}
\newcommand{\mgh}[1]{\left\{ #1\right\}}
\newcommand{\bgh}[1]{\left[ #1\right]}

\newcommand{\lbr}[1][(]{\left#1}
\newcommand{\rbr}[1][)]{\right#1}

\newcommand{\avgs}[2]{\lsbo{\lbr #1 \rbr}{#2}}
\newcommand{\avgsnoleft}[2]{\lsbo{( #2 )}{#1}}
\newcommand{\txt}[1]{\qquad \text{#1} \qquad}

\newcommand{\htq}{\hat{Q}}

\newcommand{\htb}{\hat{B}}

\newcommand{\htr}{\hat{r}}

\newcommand{\scaletexp}[2]{-1+d}

\newcommand{\scalex}[2]{#1^{-\frac{1}{p(#2)}+\frac{d}{2}}}
\newcommand{\scalet}[2]{#1^{-1+d}}

\newcommand{\scalexn}[2]{#1^{-\frac{n}{p(#2)}+\frac{nd}{2}}}

\newcommand{\nscalex}[2]{#1^{\frac{1}{p(#2)}-\frac{d}{2}}}
\newcommand{\nscalet}[2]{#1^{1-d}}
\newcommand{\omt}{ \Om \times (-T,T)}
\newcommand{\omtbr}{\lbr \Om \times (-T,T)\rbr}
\newcommand{\ov}{\overline{V}}
\newcommand{\cc}{\mathcal{D}}

\newcommand{\zv}{\zeta_{\ve}}
\newcommand{\vh}{v_h}
\newcommand{\integral}[3]{\int \hspace{-0.3cm} \int_{#1} #2 \hspace{0.1cm} #3}

\newcommand{\mint}[3]{\fiint_{#1} #2 \hspace{0.1cm} #3}

\newcommand{\Norm}[1]{\left|\hspace{-0.2mm}\left| #1 \right|\hspace{-0.2mm}\right|}

\newcommand{\OO}{\Omega}

\newcommand{\F}{\mathcal{F}}
\newcommand{\G}{\mathcal{G}}
\newcommand{\I}{\mathcal{I}}
\newcommand{\J}{\mathcal{J}}

\newcounter{whitney}
\refstepcounter{whitney}

\newcounter{ineqcounter}
\refstepcounter{ineqcounter}
\usepackage[titletoc,toc,page]{appendix}

\begin{document}

\begin{frontmatter}

\title{End point gradient estimates for quasilinear parabolic equations with variable exponent growth on nonsmooth domains}

\author[myaddress]{Karthik Adimurthi\corref{mycorrespondingauthor}\tnoteref{thanksfirstauthor}}
\cortext[mycorrespondingauthor]{Corresponding author}
\ead{karthikaditi@gmail.com and kadimurthi@snu.ac.kr}
\tnotetext[thanksfirstauthor]{Supported by the National Research Foundation of Korea grant NRF-2015R1A2A1A15053024.}

\author[myaddress,myaddresstwo]{Sun-Sig Byun\tnoteref{thankssecondauthor}}
\ead{byun@snu.ac.kr}
\tnotetext[thankssecondauthor]{Supported by the National Research Foundation of Korea grant  NRF-2015R1A4A1041675. }

\author[myaddressthree]{Jung-Tae Park\tnoteref{thanksthirdauthor}}
\ead{ppark00@kias.re.kr}
\tnotetext[thanksthirdauthor]{Supported by the National Research Foundation of Korea grant NRF-2017R1C1B1010966.}

\address[myaddress]{Department of Mathematical Sciences, Seoul National University, Seoul 08826, Korea.}
\address[myaddresstwo]{Research Institute of Mathematics, Seoul National University, Seoul 08826, Korea.}
\address[myaddressthree]{Korea Institute for Advanced Study, Seoul 02455, Korea.}

\begin{abstract}
In this paper, we study quasilinear parabolic equations  with the nonlinearity structure modeled after the $p(x,t)$-Laplacian on nonsmooth domains. The main goal is to obtain end point Calder\'on-Zygmund type estimates in the variable exponent setting. In a recent work \cite{byun2016nonlinear}, the estimates  obtained  were strictly above the natural exponent $p(x,t)$ and hence there was a gap between the natural energy estimates and the estimates above $p(x,t)$ (see \eqref{energy} and \eqref{byunok}).  Here, we bridge this gap to obtain the end point case of the estimates obtained in \cite{byun2016nonlinear}. 
To this end, we make use of the parabolic Lipschitz truncation developed in \cite{KL} and  obtain significantly  improved a priori estimates below the natural exponent with stability of the constants. An important feature of the techniques used here is that we make use of the unified intrinsic scaling introduced in \cite{adimurthi2018sharp},  which enables us to handle both the singular and degenerate cases simultaneously.
\end{abstract}


\end{frontmatter}

\tableofcontents

\section{Introduction}
Calder\'on-Zygmund theory was first developed for the  Poisson equation in \cite{CZ}, which related the integrability of the gradient  of the solution for the Poisson equation with the associated  data. This represented the starting point of obtaining a priori estimates in Sobolev spaces for elliptic and parabolic equations. Since we are interested in Calder\'on-Zygmund theory for parabolic equations in this paper, we shall discuss the history of the problem only for parabolic equations and refer the reader to \cite{adimurthi2017sharp} and references therein for the elliptic counterpart.

\emph{All the estimates mentioned in this introduction are quantitative in nature, but to avoid being too technical, we only recall the qualitative nature of the bounds. This is sufficient to highlight the nature of the results that we will prove in this paper. }

The starting point of Calder\'on-Zygmund theory for quasilinear parabolic equations was developed in \cite{AM2}, where they considered the following problem:
\[
 \begin{array}{rcll}
  u_t - \dv (a(x,t)|\nabla u|^{p-2} \nabla u) &=& - \dv (|\bff|^{p-2} \bff) & \quad \text{in} \ \omt,
 \end{array}
\]
with $a(x,t) \in \text{VMO}$ and $p > \frac{2n}{n+2}$, proving 
\[
 |\bff| \in L^{q}_{loc}\omtbr \Longrightarrow |\nabla u| \in L^q_{loc}\omtbr \qquad \text{for all} \ q >p. 
\]	

After this pioneering work, there have been numerous publications which extended these estimates to other quasilinear parabolic equations with constant $p$-growth. In \cite{bogelein2014global}, the authors improved the estimate in \cite{AM2} to obtain global a priori estimates (with non homogeneous boundary data) and proved
\[
 |\bff| \in L^{q}\lbr \Om \times (-T+\de, T)\rbr \Longrightarrow |\nabla u| \in L^q\lbr \Om \times (-T+\de, T)\rbr \qquad \text{for all} \ q >p \ \text{and some}\  \de \in (0,2T).
\]
This was subsequently extended in \cite{BOS1} to prove global a priori estimates for more general nonlinear structures satisfying a small BMO condition and Reifenberg-flat domains (see Section \ref{two} for the precise definitions).

In this paper, we are interested in obtaining Calder\'on-Zygmund type bounds for the problem
\begin{equation}
 \label{basic_pde}
\left\{ \begin{array}{rcll}
  u_t - \dv \aa(x,t,\nabla u) &=& -\dv (|\bff|^{p(x,t)-2} \bff) & \quad \text{in} \ \omt,\\
  u &=& 0 & \quad \text{on} \ \partial \Om \times (-T,T).
 \end{array}\right. 
\end{equation}
Here,  the quasilinear operator $\aa(x,t,\nabla u)$ is modeled after well known $p(x,t)$-Laplacian operator having the  form $ |\nabla u|^{p(x,t) - 2} \nabla u$ with $p(\cdot) > \frac{2n}{n+2}$. For more on the importance of variable exponent problems, see \cite{AS,CLR,HU,RR,Ru,VVZ} and the references therein.

In a recent paper \cite{baroni2014calderon}, the authors were able to show
\[
 |\bff|^{\pp} \in L^{q}_{loc}\lbr \Om \times (-T, T)\rbr \Longrightarrow |\nabla u|^{\pp} \in L^q_{loc}\lbr \Om \times (-T, T)\rbr \quad \text{for all} \ 1<q<\infty.
\]
This was subsequently improved to a global estimate in \cite{byun2016nonlinear}, where they proved
\begin{equation}
\label{byunok}
 |\bff|^{\pp} \in L^{\qq}\lbr \Om \times (-T, T)\rbr \Longrightarrow |\nabla u|^{\pp} \in L^{\qq}\lbr \Om \times (-T, T)\rbr \quad \text{for all} \ 1<q^- \leq \qq \leq q^+ <\infty.
\end{equation}
In particular, they could not take $q^- = 1$.

On the other hand, from the definition of weak solution, it is easy to see that the following energy-type estimate holds:
\begin{equation}
 \label{energy}
 |\bff|^{\pp} \in L^{1}\lbr \Om \times (-T, T)\rbr \Longrightarrow |\nabla u|^{\pp} \in L^{1}\lbr \Om \times (-T, T)\rbr.
\end{equation}

Comparing \eqref{byunok} and \eqref{energy}, it seems reasonable to expect that \eqref{byunok} should hold with $1 \leq q^-\leq \qq \leq q^+ < \infty$, i.e., \emph{it should be possible to take $q^- =1$}. 

In this paper, we  prove that we can indeed take $q^- = 1$ in \eqref{byunok}. In order to do this, we will obtain improved estimates below the natural exponent $\pp$ using the method of parabolic Lipschitz truncation developed in the seminal paper \cite{KL}, as well as the unified intrinsic scaling of \cite{adimurthi2018sharp}.

In order to prove our results,  we need to impose some restrictions on the variable exponent $p(x,t)$, on the nonlinear structure $\aa(x,t,\nabla u)$ as well as on the boundary of the domain $\partial \Om$. These restrictions will be described in detail in Section \ref{two}.

The plan of the paper is as follows: In Section \ref{two}, we collect all assumptions that will be needed on the structure of the nonlinearity $\aa$, on the domain $\Om$ and on the variable exponent $\pp$. In Section \ref{Weak solution}, we define the notion of weak solutions and collect some of their well known properties. In Section \ref{three}, we state the main results of this paper. In Section \ref{four}, we collect all the preliminary results and well known lemmas that will be needed in subsequent parts of the paper. In Section \ref{four-two}, we describe the approximations that will be made along the way. In Section \ref{five} and Section \ref{six}, we prove crucial difference estimates below the natural exponent for energy solutions. In Section \ref{eight}, we demonstrate some important covering arguments. In Section \ref{nine}, the proof of the main theorems will be provided. Finally in Appendix \ref{lipschitz_truncation} and Appendix \ref{lipschitz_truncation_B}, we will describe the construction of test functions having Lipschitz regularity which will be needed to prove the estimates in Section \ref{five} and Section \ref{six}, respectively.


\section{Regularity assumptions and notation}
\label{two}

In this section, we shall collect all the structure assumptions as well as recall several useful lemmas that are already available in existing literature. 
 
 \subsection{Metrics needed}
 Let us first collect a few metrics on $\RR^{n+1}$ that will be used throughout the paper.
 \begin{definition}
 \label{parabolic_metric}
 We define the parabolic metric $d_p$ on $\RR^{n+1}$ as follows: Let $z_1 = (x_1,t_1)$ and $z_2 = (x_2,t_2)$ be any two points on $\RR^{n+1}$, then 
 \begin{equation*}
 \label{par_met}
 d_p(z_1,z_2) := \max \mgh{|x_1-x_2|, \sqrt{|t_1-t_2|}}.
 \end{equation*}

 \end{definition}

 Since we will use intrinsically scaled cylinders where the scaling depends on the center of the cylinder, we will also need to consider the following localized parabolic metric:
 \begin{definition}
 \label{loc_parabolic_metric}
Given a function $1 < \pp < \infty$, some fixed point $z= (x,t) \in \RR^{n+1}$ and any $\tau > 0$, $d > 0$, we define the localized parabolic metric $d_z^{\tau,d}$ as follows: Let $z_1 = (x_1,t_1)$ and $z_2 = (x_2,t_2)$ be any two points on $\RR^{n+1}$, then 
 \begin{equation*}
 \label{loc_par_met}
 d_z^{\tau,d}(z_1,z_2) := \max \mgh{\nscalex{\tau}{z}|x_1-x_2|, \sqrt{\nscalet{\tau}{z}|t_1-t_2|}}.
 \end{equation*}
 \end{definition}
\subsection{Structure of the variable exponent}
\label{exponent_structure} 
 
\begin{definition}
\label{definition_p_log}
 We say that, a bounded measurable function $\pp : \RR^{n+1} \rightarrow \RR$  belongs to the $\log$-H\"older class $\logh$, if the following conditions are satisfied:
 \begin{itemize}
  \item There exist constants $p^-$ and $p^+$ such that $1< p^- \leq p(z) \leq p^+ < \infty$ for every $z \in \RR^{n+1}$.
  \item $ |p(z_1)  - p(z_2)| \leq \frac{L}{- \log |z_1-z_2|}$ holds  for every $ z_1,z_2 \in \RR^{n+1}$ with $ d_p(z_1,z_2) \leq \frac12 $ and for some  $L>0$.
 \end{itemize}
 
\end{definition}

\begin{remark}\label{remark_def_p_log} We remark that  $\pp$ is log-H\"{o}lder continuous in $\RR^{n+1}$ if and only if there is a nondecreasing continuous function ${\modp} : [0,\infty) \rightarrow [0,\infty)$ such that 
\begin{itemize}
 \item $\lim_{r\rightarrow 0} \modp(r) = 0$ and $|p(z_1)- p(z_2)| \leq \modp(d_p(z_1,z_2))$ for every $z_1,z_2 \in \RR^{n+1}$.
 \item $\modp(r) \log \lbr \frac{1}{r} \rbr \leq {L}$ holds for all $ 0< r \leq \frac12.$
\end{itemize}
The function $\modp$ is called the modulus of continuity of the variable exponent $\pp$. 
 \end{remark}

\subsection{Structure of the domain}
 The domain that we consider may be nonsmooth but should satisfy some regularity condition. This condition would essentially say that at each boundary point and every scale, we require the boundary of the domain to be between two hyperplanes separated by a distance  proportional to the scale.  

\begin{definition}
\label{reif_flat}
Given any $\ga \in (0,1)$ and any $\bs_0 >0$, we say that $\Om$ is $(\ga,\bs_0)$-Reifenberg flat domain if for every $x_0 \in \pa \Om$ and every $r \in (0,\bs_0]$, there exists a system of coordinates $\{y_1,y_2,\ldots,y_n\}$ (possibly depending on $x_0$ and $r$) such that in this coordinate system, $x_0 =0$ and 
\[
B_r(0) \cap \{y_n > \ga r\} \subset B_r(0) \cap \OO \subset B_r(0) \cap \{y_n > -\ga r\}.
\]
\end{definition}

The class of Reifenberg flat domains are standard in obtaining Calder\'on-Zygmund type estimates, in the elliptic case, see \cite{AP2,BO,BO1,BW-CPAM} and references therein, whereas for the parabolic case, see \cite{Bui1,MR3461425,BOS1,MR2836359} and the references therein. 

\begin{definition}
\label{measure_def}
We say that a bounded domain $\Om$ is said to satisfy a uniform measure density condition with a constant $m_e >0$ if for every $x \in \overline{\Om}$ and every $r >0$, there holds
\[
|\Om^c \cap B_r(x)| \geq m_e |B_r(x)|.
\]
\end{definition}

From the definition of $(\ga,\bs_0)$-Reifenberg flat domains, it is easy to see that the following property holds:
\begin{lemma}
\label{measure_density}
Let $\ga \in (0,1/8)$ and $\bs_0>0$ be given and suppose that $\Om$ is a $(\ga,\bs_0)$-Reifenberg flat domain. Then the following measure density conditions hold:
\begin{equation}\label{measure_one}\begin{array}{c}
\sup_{y \in \Om} \sup_{r \leq \bs_0} \frac{|B_r(y)|}{|B_r(y) \cap \Om|} \leq \lbr \frac{2}{1-\ga} \rbr^n \leq \lbr \frac{16}{7} \rbr^n, \\
\inf_{y \in \pa\Om} \inf_{r \leq \bs_0} \frac{|\Om^c \cap B_r(y)|}{|B_r(y)|} \geq \lbr \frac{1-\ga}{2} \rbr^n \geq \lbr \frac7{16} \rbr^n.
\end{array}\end{equation}
\end{lemma}


\subsection{Structure of the nonlinearity \texorpdfstring{$\aa$}.}
\label{nonlinear_structure}
We first assume that $\aa(\cdot,\cdot,\cdot)$ is a Carath\'eodory function in the sense:
\begin{gather}
(x,t) \mapsto \aa(x,t,\zeta) \ \text{is measurable for every } \ \zeta \in \RR^n \nonumber, \\
\zeta \mapsto\aa(x,t,\zeta) \ \text{is continuous for almost every } \ (x,t) \in \RR^n \times \RR \nonumber.
\end{gather}
Let $\mu \in [0,1]$ be given, then there exist two positive constants $\La_0,\La_1$ such that the following holds for almost every $x \in \Om$ and every $\zeta, \eta \in \RR^n$, 
\begin{gather}
(\mu^2 + |\zeta|^2)^{\frac12} |D_{\zeta} \aa(x,t,\zeta)| + |\aa(x,t,\zeta)| \leq \La_1 (\mu^2 + |\zeta|^2)^{\frac{p(x,t) -1}{2}} \label{abounded}, \\
(\mu^2 + |\zeta|^2 )^{\frac{p(x,t)-2}{2}} |\eta|^2 \La_0 \leq \iprod{D_{\zeta}\aa(x,t,\zeta)\eta}{\eta} \label{aellipticity}.
\end{gather}
%
%
%
%
We point out that from \eqref{aellipticity}, one can derive the following monotonicity bound: 
\begin{equation}
 \label{monotonicity}
 \iprod{\aa(x,t,\zeta) - \aa(x,t,\eta)}{\zeta - \eta} \geq \tilde{\La}_0 (\mu^2 + |\zeta|^2 + |\eta|^2)^{\frac{p(x,t)-2}{2}} |\zeta- \eta|^2,
\end{equation}
where $\tilde{\La}_0 = \tilde{\La}_0(\La_0,n,p^+,p^-)>0$.  By inserting $\eta=0$ into \eqref{monotonicity}, we also have the following coercivity bound:
\begin{equation*}\label{coercivity}
 \tilde{\La}_2 |\zeta|^{p(x,t)} \leq \iprod{\aa(x,t,\zeta)}{\zeta} + \tilde{\La}_1,
\end{equation*}
where $\tilde{\La}_1 = \tilde{\La}_1 (\La_1,\La_0,p^+,p^-,n)>0$ and $\tilde{\La}_2 = \tilde{\La}_2(\La_1,\La_0,p^+,p^-,n)>0$.

\subsection{Smallness assumption}
\label{smallness_assumption}
In order to prove the main results, we need to assume a smallness condition satisfied by  $(\pp,\aa,\Om)$. 
\begin{definition}\label{further_assumptions}
 Let $\ga\in (0,1/8)$ and $\bs_0>0$ be given, we then say $(\pp,\aa,\Om)$ is $(\ga,\bs_0)$-vanishing  if the following three structure conditions are satisfied:
 \begin{description}[leftmargin=*]
  \item[(i) Assumption on $\pp$:] The variable exponent $\pp$ with modulus of continuity $\modp$ as defined in Definition \ref{definition_p_log} with $p^- > \frac{2n}{n+2}$, is further assumed to satisfy the smallness condition:
  \begin{equation}
   \label{small_px}
   \sup_{0<r\leq \bs_0} {\modp}(r) \log \lbr \frac{1}{r} \rbr \leq \ga.
  \end{equation}

  \item[(ii) Assumption on $\aa$:] For a bounded open set $U \subset \RR^n$, let us denote
  \begin{equation*}
   \label{a_difference}
   \Theta(\aa,U)(x,t) := \sup_{\zeta \in \RR^n} \lbr[|] \frac{\aa(x,t,\zeta)}{(\mu^2 + |\zeta|^2)^{\frac{p(x,t)-1}{2}}} - \avg{\frac{\aa(\cdot,t,\zeta)}{(\mu^2 + |\zeta|^2)^{\frac{p(\cdot,t)-1}{2}}}}{U} \rbr[|],
  \end{equation*}
where we have used the notation $\avg{f}{U} := \fint_U f(y)\ dy$. Note that if $\mu=0$, then $\zeta \in \RR^n\setminus \{0\}$.

  We assume that the nonlinearity $\aa$ has small BMO with constant $\ga$ if there holds
\begin{equation}
 \label{small_aa}
\sup_{\substack{t_1,t_2 \in \RR,\\ t_1 < t_2}} \sup_{0<r\leq \bs_0} \sup_{y \in  \RR^n} \fint_{t_1}^{t_2}\fint_{B_r(y)} \Theta(\aa,B_r(y))(x,t) \ dx \ dt \leq \ga.
\end{equation}
\item[(iii) Assumption on $\pa \Om$:] The domain $\Om$ is $(\ga,\bs_0)$-Reifenberg flat in the sense of  Definition \ref{reif_flat}. 
 \end{description}
\end{definition}

\subsection{Notation}
We shall use the following notations throughout the paper:
\begin{itemize}
\item We will use $z,\mfz,\tz,\ldots$ to be points in $\RR^{n+1}$, symbols $x,\mfx,\tx,y,\ty,\ldots$ to denote space variables in $\RR^n$ and symbols $t,\mft,s,\mfs,\ldots$ to denote time variables. We will also specifically match symbols, i.e., $z = (x,t)$ or $\mfz = (\mfx,\mft)$ and so on.
\item In all subsequent sections, the subscript $[\cdot]_h$ will always denote the usual Steklov average.
\item In what follows, the function $\modp$ denotes the modulus of continuity of $p(\cdot)$ and we denote $\modq$ for the modulus of continuity of $q(\cdot)$.
\item We shall write $\pp$ as well as $p(\cdot,\cdot)$ depending on the necessity and we will switch between the two notations without notice throughout the paper.
\item For the variable exponent $\pp$, we shall denote by $\plog$ to include the constants $p^+$, $p^-$ and those that are part of the $\log$-H\"older continuity structure of $\pp$. Analogously, for variable exponents $\qq$, $r(\cdot)$ and $\sss$, we shall use $\qlog$, $\rlog$ and $\slog$ to denote corresponding constants. 
\item Capital alphabets with subscripts as in radii $\br_0,\br_1\ldots$, or  bounding values $\bm,\bm_0,\bm_1,\ldots$ will be fixed in subsequent sections once they are chosen. 
\item We shall use $\apprle$,  $\apprge$ and $\approx$ to suppress writing the constants that could possibly change from line to line as long as they depend only on the structure constants of the form $n,\plog,\qlog,\La_0,\La_1,\bs_0$ and related quantities.
\item We shall sometimes use $\sim$ to denote variables (without subscripts) that occur only within the proof of the concerned result, for example $\tilde{r}, \tilde{m}, \cdots$.
\item Given a variable exponent $\pp$, we shall use the following notation:
\begin{equation*}\label{notation_p_inf}
 p_E^- : = \inf_{x \in E} p(x) \qquad \text{and} \qquad p_E^+ := \sup_{x \in E} p(x).
\end{equation*}
We will drop the set $E$ and denote $p^+:= \sup_{z\in \RR^{n+1}} p(z)$ and $p^-:=\inf_{z\in \RR^{n+1}} p(z)$.
\item We will denote $\Om_T:= \Om \times (-T,T)$ which is the region on which \eqref{basic_pde} is considered. We will also use the notation $\pa_p$ to denote the parabolic boundary, i.e, \[\pa_p Q_{\rho,s}(x,t):= B_{\rho}(x) \times \{t-s\} \bigcup \pa B_{\rho}(x) \times [t-s,t+s).\]
\end{itemize}

\subsection{Unified intrinsic cylinders}\label{Intrinsic cylinders}

 We will describe the intrinsically scaled cylinders that will be used in this paper. Let $\Om$ be a bounded domain in $\RR^n$, and let $\rho >0,s>0$, $\la >0$ and $\mfz = (\mfx,\mft) \in \RR^{n+1}$ be given.  Furthermore, let $d$ be a fixed exponent satisfying
 \begin{equation}
 \label{def_d}
 \min\left\{ \frac{2}{p^+},1 \right\} > d > \frac{2n}{(n+2)p^-}.
 \end{equation}
We  define the following cylinders that will be used throughout the paper:
 \begin{gather*}
 Q_{\rho,s}(\mfz) := B_{\rho}(\mfx) \times (\mft - s^2, \mft + s^2),\\
 Q^{\la}_{\rho,s}(\mfz) := B_{\scalex{\la}{\mfz}\rho}(\mfx) \times (\mft - \scalet{\la}{\mfz}s^2, \mft + \scalet{\la}{\mft}s^2) := B_{\rho}^{\la}(\mfx) \times I_{s}^{\la}(\mft).
 \end{gather*}

 We will also use the following short notation:
 \begin{equation*}
 \begin{array}{ll}
 \Om_{\rho}(\mfx) := B_{\rho}(\mfx) \cap \Om, & K_{\rho,s}(\mfz) := Q_{\rho,s}(\mfz) \cap \Om_T, \\
 I_{\rho}(\mft) := (\mft - \rho^2, \mft + \rho^2), & I_{\rho}^{\la}(\mft) := (\mft - \scalet{\la}{\mfz}\rho^2, \mft + \scalet{\la}{\mfz}\rho^2),\\
 \Om^{\la}_{\rho}(\mfx) := B_{\scalex{\la}{\mfx,\mft}\rho(\mfx)} \cap \Om,&   K^{\la}_{\rho,s}(\mfz) := Q^{\la}_{\rho,s}(\mfz) \cap \Om_T, \\
 \pa_w \Om_{\rho}(\mfx):= B_{\rho}(\mfx) \cap \pa \Om, & \pa_w K_{\rho,s}(\mfz) := K_{\rho,s}(\mfz) \cap \left\{ \pa \Om \times (-T,T) \right\},\\
 \pa_w \Om^{\la}_{\rho}(\mfx):= B_{\scalex{\la}{\mfz}\rho}(\mfx) \cap \pa \Om, & \pa_w K^{\la}_{\rho,s}(\mfz) := K^{\la}_{\rho,s}(\mfz) \cap \left\{ \pa \Om \times (-T,T) \right\},\\
 \pa_p Q_{\rho,s}(\mfz):= B_{\rho}(\mfx) \times \{\mft-s^2\} \bigcup \pa B_{\rho}(x) \times I_s(\mft), & \pa_p Q^{\la}_{\rho,s}(\mfz):= B^{\la}_{\rho}(\mfx) \times \{\mft-\scalet{\la}{\mfz}s^2\} \bigcup \pa B^{\la}_{\rho}(x) \times I^{\la}_s(\mft),\\
 Q_{\rho}(\mfz) := Q_{\rho,\rho}(\mfz), \quad K_{\rho}(\mfz) := K_{\rho,\rho}(\mfz), & Q_{\rho}^{\la}(\mfz) := Q_{\rho,\rho}^{\la}(\mfz), \quad K_{\rho}^{\la}(\mfz) := K_{\rho,\rho}^{\la}(\mfz).
 \end{array}
 \end{equation*}

 We will also have to deal with half spaces, and use the following notation in that regard:
 \begin{equation*}
 \begin{array}{c}
 B_{\rho}^+(\mfx) := B_{\rho}(\mfx) \cap \{ \mfx_n > 0\}, \qquad B_{\rho}^{\la,+}(\mfx) := B_{\rho}^{\la}(\mfx) \cap \{ \mfx_n > 0\}, \\
 Q_{\rho}^{\la,+}(\mfz) := B_{\scalex{\la}{\mfz}\rho}^+(\mfx) \times \lbr \mft - \scalet{\la}{\mfz}\rho^2,\mft + \scalet{\la}{\mfz}\rho^2\rbr,\\
 T_{\rho}^{\la}(\mfz) := B_{\scalex{\la}{\mfz}\rho}(\mfx) \cap \{\mfx_n >0\} \times \lbr \mft - \scalet{\la}{\mfz}\rho^2,\mft + \scalet{\la}{\mfz}\rho^2\rbr.
 \end{array}
 \end{equation*}

 An important thing to note is that the cylinders considered above are intrinsically scaled both in space and time simultaneously. This enables us to handle both the singular case ($p(\cdot) <2$) and degenerate case ($p(\cdot) >2$) simultaneously.

\subsection{Restriction on radii}
\label{sub_radii}

In this subsection, let us collect all the restrictions we will make on some universal constants. First, let us describe all the restriction on the radii $\rho_0$: 
\begin{description}
\descitem{(R1)}{R1} Let $\rho_0 \leq \frac14$ such that $|Q_{\rho_0}| = (\rho_0)^{n+2} |B_1|\leq 1$.
\descitem{(R2)}{R2} Let $\rho_0$ be such that $\frac{\modp(8\rho_0)}{p^-} < \min\{\tilde{\be}_1,\tilde{\be}_2\}$, where $\tilde{\be}_1$ is from Theorem \ref{high_weak}  and $\tilde{\be}_2$ is from Theorem \ref{high_very_weak} applied with $\bm_{\vec{f}}=\bm_0$. Here $\bm_0$ is given in \eqref{def_M_0}.
\descitem{(R3)}{R3} Let $\rho_0 \leq \min \{ \tilde{\rho_1}, \tilde{\rho}_2\}$, where $\tilde{\rho_1}$ is from Theorem \ref{high_weak} and $\tilde{\rho_2}$ is from Theorem \ref{high_very_weak} applied with $\bm_{\vec{f}}=\bm_0$.
\descitem{(R4)}{R4} Let $1024\rho_0 \leq \min\left\{ \frac{1}{\bm_0},\frac{1}{\bm_u},\frac{1}{\bm_w}\right\}$, where $\bm_0$, $\bm_u$, and $\bm_w$ are from \eqref{def_M_0}, \eqref{size_u_f}, and \eqref{size_w}, respectively.
\descitem{(R5)}{R5} Let $\rho_0$ satisfy $\frac{\modp(12\rho_0)}{{p^{-}} -1}  \leq \be_0 \overset{\text{Section \ref{def_be_0}}}{\le}\min\{\tilde{\be}_1,\tilde{\be}_2\}$, where $\tilde{\be}_1$ is from Theorem \ref{high_weak} and $\tilde{\be}_2$ is from Theorem \ref{high_very_weak} applied with $\bm_{\vec{f}}=\bm_0$.
 \descitem{(R6)}{R6}  With $\bm_p = \max\{ \bm_0, \bm_u, \bm_w\}$, we will apply Lemma \ref{scaled_poincare} and Theorem \ref{measure_density_poincare} which will impose the restriction  $\rho_0 \leq \br_p$. 
\descitem{(R7)}{R7} Let $\rho_0 \le \frac{\bs_0}{\Ga^2}$, where $\Ga$ is given in \eqref{cv0-1} and $S_0$ is from Definition \ref{further_assumptions}.
\descitem{(R8)}{R8} Let $\modp(2\rho_0) \leq \min \lbr[\{] \frac{p^-\sigma}{2}, \frac{\La_0}{2\La_1}, \frac{(p^--1)\sigma}{4}, \frac14, d_0 p^-, d_0 p^- (p^--1) \rbr[\}]$, where $\sigma$ is given in Remark \ref{high_int_remark} and $d_0$ is defined in \eqref{def_d_0}.
\descitem{(R9)}{R9} Let $\modq(2\rho_0) \leq \min \left\{\frac{q^-\sigma}{4}, \frac14 \right\}$, where $\sigma$ is given in Remark \ref{high_int_remark} and $\qq$ is the exponent appearing in Theorem \ref{main_theorem1}.
\end{description}

\begin{remark}\label{remark_radius}Note that all the restrictions on $\rho_0$ are such that $\rho_0 = \rho_0(n,\lamot,\plog,\bm_0)\in (0,1/4)$ and henceforth we will always take the radius $\rho$ to satisfy $128\rho \leq \rho_0$. \end{remark}

\subsection{Fixing a few other exponents} \label{def_be_0}

We will first collect all the restrictions on the higher integrability exponent:
 \begin{description}
\descitem{(B1)}{B1} Let $0<\be_0 \leq \min\mgh{ \frac{1}{p^+}, \tilde{\be}_1, \tilde{\be}_2, \tilde{\be}_3, \tilde{\be}_4 }$, where $\tilde{\be}_1$, $\tilde{\be}_2$, $\tilde{\be}_3$, and $\tilde{\be}_4$ are given in Theorem \ref{high_weak}, Theorem \ref{high_very_weak}, Theorem \ref{first_diff_thm}, and Theorem \ref{second_diff_thm}, respectively.
 \descitem{(B2)}{B2} Once $\be_0$ is fixed, let $\sigma_0$ be a number chosen such that $0<\sigma_0 \le \min\mgh{\frac{\beta_0}{3(1-\beta_0)}, \frac{q^- - 1}{3}, 1}$ holds.
\descitem{(B3)}{B3} Let $\vt_0 = \max\{ \tilde{\vt}_1, \tilde{\vt}_2\}$, where $\tilde{\vt}_1$ and $\tilde{\vt}_2$ are from Theorem \ref{high_weak} and Theorem \ref{high_very_weak} with $\bm_{\vec{f}} = \bm_0$. Here $\bm_0$ is given in \eqref{def_M_0}.
 \end{description}

 \begin{remark} \label{high_int_remark}Henceforth, we will assume $0<\sigma \leq \sigma_0$ and $0<\be \leq \be_0$.\end{remark}

\section{Weak solution}\label{Weak solution}

\subsection{Sobolev spaces with variable exponents}
Let $\tom$ be a bounded domain in $\RR^{N}$ for some $N\geq 1$, and let $\sss$ be an admissible variable exponent as in Section \ref{exponent_structure}. Given a  positive integer $m$, the \emph{variable exponent Lebesgue space} $L^{\sss}(\tom,\RR^m)$ consists of all measurable functions $\bff: \tilde{\Om} \to \RR^m$ satisfying 
\[
 \int_{\tilde{\Om}} |\bff(z)|^{s(z)} \ dz < \infty,
\]
endowed with the Luxemburg norm
\[
 \|\bff\|_{L^{\sss}(\tom,\RR^m)} := \inf \left\{ \la >0 : \int_{\tilde{\Om}} \left|\frac{\bff(z)}{\la} \right|^{s(z)} \ dz \leq 1 \right\}.
\]

Analogously, we can define the \emph{variable exponent Sobolev space} as 
\[
 W^{1,\sss}(\tom,\RR^m) := \{ \bff \in L^{\sss}(\tom,\RR^m) : \nabla\bff \in L^{\sss}(\tom,\RR^{mN}) \},
\]
equipped with the norm 
\begin{equation}\label{norm_up}
 \| \bff \|_{W^{1,\sss}(\tom,\RR^m)} := \| \bff \|_{L^{\sss}(\tom,\RR^m)} + \| \nabla \bff \|_{L^{\sss}(\tom,\RR^{mN})}.
\end{equation}
We shall denote $W_0^{1,\sss}(\tom,\RR^m)$ to be the closure of $C_c^\infty(\tom,\RR^m)$ under the norm from \eqref{norm_up}. 
Then all function spaces mentioned above are separable Banach spaces. For $m=1$, we write $L^{\sss}(\tom)$ and $W^{1,\sss}(\tom)$ for simplicity. 
We will also use the following modular function:
\[ \varrho_{L^{\sss}(\tom)}(f) := \int_{\tom} |f(z)|^{s(z)} \ dz.\]
We mention the following useful relation between the modular and the norm in  variable exponent spaces (see \cite[Lemma 3.2.5]{diening} for details):
\begin{lemma}
 \label{integral_norm}
 For any $f \in L^{\sss}(\tom)$, the following holds:
 \begin{equation*}
  \label{norm_integral}
  \min \left\{ \varrho_{L^{\sss}(\tom)} (f)^{\frac{1}{s^-}}, \varrho_{{L^{\sss}(\tom)}} (f)^{\frac{1}{s^+}} \right\} \leq \| f\|_{L^{\sss}(\tom)} \leq \max \left\{ \varrho_{{L^{\sss}(\tom)}} (f)^{\frac{1}{s^-}}, \varrho_{{L^{\sss}(\tom)}} (f)^{\frac{1}{s^+}} \right\}.
 \end{equation*}
\end{lemma}

Let us now define some function spaces involving time. Let $\Om \subset \RR^n$ be a bounded domain, and the space $L^{\sss}\lbr-T,T; W^{1,\sss}(\Om)\rbr$ is defined as
\[
 L^{\sss}\lbr-T,T; W^{1,\sss}(\Om)\rbr := \left\{ f \in L^{\sss}(\Om_T) : \nabla_{\text{space}} f \in L^{\sss}(\Om_T,\RR^{n}) \right\},
\]
equipped with the norm 
\[
 \| f \|_{L^{\sss}\lbr-T,T; W^{1,\sss}(\Om)\rbr} := \| f \|_{L^{\sss}(\Om_T)} + \| \nabla f \|_{L^{\sss}(\Om_T,\RR^{n})}.
\]
We shall define $L^{\sss}\lbr-T,T; W_0^{1,\sss}(\Om)\rbr := L^{\sss}\lbr-T,T; W^{1,\sss}(\Om)\rbr \cap L^1\lbr-T,T; W_0^{1,1}(\Om)\rbr$, and let us denote $L^{\sss}\lbr-T,T; W^{1,\sss}(\Om)\rbr'$ the dual space of $L^{\sss}\lbr-T,T; W_0^{1,\sss}(\Om)\rbr$. We remark that if $\sss$ is constant function, then all function spaces considered above become well known classical parabolic Sobolev spaces.

\subsection{Definition of weak solution}
There is a well known difficulty in defining the notion of solution for \eqref{basic_pde} due to \emph{ a lack of time derivative of $u$}. To overcome this, one can either use Steklov average or convolution in time. In this paper, we shall use the former approach (see also \cite[Chapter 2]{DiB1} for further details).

Let us first define Steklov average as follows: let $h \in (0, 2T)$ be any positive number, then we define
\begin{equation}\label{stek1}
  [u]_{h}(\cdot,t) := \left\{ \begin{array}{ll}
                              \hint_t^{t+h} u(\cdot, \tau) \ d\tau \quad & t\in (-T,T-h), \\
                              0 & \text{else}.
                             \end{array}\right.
 \end{equation}
 Let us now define the notion of solution that will considered in this paper.
 \begin{definition}
\label{weak_solution}
 
 Let $h \in (0,2T)$ be given, we  then say $u \in L^2\lbr-T,T; L^2(\Om)\rbr \cap L^{\pp}\lbr-T,T; W_0^{1,\pp}(\Om)\rbr$ is a weak solution of \eqref{basic_pde} if for any $\phi \in W_0^{1,{\pp}}(\Om)$, the following holds:
 \begin{equation*}
 \label{def_weak_solution}
  \int_{\Om \times \{t\}} \frac{d [u]_{h}}{dt} \phi + \iprod{[\aa(x,t,\nabla u)]_{h}}{\nabla \phi} \ dx = 0 \quad \text{for almost every} \  -T < t < T-h.
 \end{equation*} 
\end{definition}

\subsection{Existence and uniqueness of weak solution}
\label{existence}
We begin with the following well known existence and uniqueness result:
\begin{proposition}[\cite{erhardt2013existence,DNR12}]
\label{ext_sol}
Let $\tom$ be any bounded domain satisfying a uniform measure density condition (see Definition \ref{measure_def}). Suppose that $\vec{f} \in L^{\pp}(\tom_T)$, $f \in L^{\pp}\lbr-T,T; W^{1,\pp}(\tom)\rbr$ with $\ddt{f} \in L^{\pp}\lbr-T,T; W^{1,\pp}(\tom)\rbr'$  and $f_0 \in L^{2}(\tom)$ are given. Then there is a unique weak solution $\phi \in C^0\lbr-T,T;L^2(\tom)\rbr\cap L^{\pp}\lbr-T,T; W^{1,\pp}(\tom)\rbr$ solving
\begin{equation*}
\label{ext_u}
\left\{ \begin{array}{rcll}
  \phi_t - \dv \cc(z,\nabla \phi) &=& -\dv |\vec{f}|^{\pp-2}\vec{f} & \quad \text{in} \ \tom_T, \\
  \phi &=& f & \quad \text{on} \  \pa \tom \times (-T,T),\\
  \phi(\cdot,-T) & = & f_0 & \quad \text{on} \ \tom,
 \end{array}\right. 
\end{equation*}
where $\cc$ is any operator satisfying all the assumptions in Section \ref{nonlinear_structure}.

Moreover if $f =0$,  we then  have the following energy estimate:
\begin{equation*}
\label{energy_phi}
\sup_{-T \leq t \leq T} \|\phi(\cdot,t)\|_{L^2(\tom)}^2 + \iint_{\tom_T} |\nabla \phi|^{\pp} \ dz \apprle_{(n,\plog,\lamot)} \lbr \iint_{\tom_T}\left[|\vec{f}|^{\pp} + 1\right] \ dz + \|f_0\|_{L^2(\tom)}^2\rbr.
\end{equation*}
\end{proposition}

Returning to our problem \eqref{basic_pde}, Proposition \ref{ext_sol} yields the existence and uniqueness result as follows:
\begin{corollary}
There exists a unique weak solution $u \in C^0\lbr-T,T;L^2(\Om)\rbr\cap L^{\pp}\lbr-T,T; W^{1,\pp}(\Om)\rbr$ solving \eqref{basic_pde} with  the estimate
\begin{equation}
\label{energy_u}
\sup_{-T \leq t \leq T} \|u(\cdot,t)\|_{L^2(\Om)}^2 + \iint_{\Om_T} |\nabla u|^{\pp} \ dz \leq C_{(n,\plog,\lamot)}  \iint_{\Om_T}\left[|\bff|^{\pp} + 1\right] \ dz.
\end{equation}
\end{corollary}

\section{Main results}
\label{three}

We now state the main results of this paper. Let us first set
\begin{gather}\label{vt_def} 
\vt(z) := \frac{1}{-\frac{n}{p(z)} + \frac{nd}{2} + d} \quad \text{and} \quad \vt^+ := \sup_{z\in \Om_T} \vt(z),
\end{gather}
where the constant $d$ is given in \eqref{def_d}.

The first theorem concerns the local estimate around small balls. 
\begin{theorem}
 \label{main_theorem1}
Assume that $u$ is the weak solution of the problem \eqref{basic_pde} under the structure conditions \eqref{abounded} and \eqref{aellipticity}. Let $0< \bs_0 < 1$, and $q(\cdot)$ be log-H\"{o}lder continuous satisfying $1 < q^- \le q(\cdot) \le q^+ < \infty$. There exist constants ${\gamma_0} \in (0,1/8)$ and ${\be_0} \in (0,1/4)$, both depending only on $\La_0$, $\La_1$, $\plog$, $\qlog$, $n$, such that if $(p(\cdot),\aa,\Om)$ is $(\gamma,\bs_0)$-vanishing for some $\ga \in (0,{\ga_0})$, then there exists a constant $C_0 = {C_0}_{(\La_0,\La_1,\plog, \qlog, n, \bs_0)}>0$ such that for any $\mfz \in \Om_T$, $\be \in (0,{\be_0})$  and $\rho \in (0,1/(C_0 \bm)]$, we have 
\begin{equation*}\label{main-r1}
\begin{aligned}
\fiint_{K_{\rho}(\mfz)} |\nabla u|^{p(z)(1-\be)q(z)} \ dz &\le C \left\{\fiint_{K_{4\rho}(\mfz)} |\nabla u|^{p(z)(1-\be)} \ dz  + \lbr \fiint_{K_{4\rho}(\mfz)} |\bff|^{p(z)(1-\be)q(z)} \ dz \rbr^{\frac{1}{q(\mfz)}} +1 \right\}^{1+ \vt(\mfz)(q(\mfz) -1)},
\end{aligned}
\end{equation*}
for some constant $C=C_{(\La_0,\La_1,\plog, \qlog, n)}>0$. Here $\bm$ and $\vt(\mfz)$ are given in \eqref{def_MM} and \eqref{vt_def}, respectively.
\end{theorem}

In the above theorem, it is important to note that the exponent $q^->1$, on the other hand, the above estimate has $p(z)(1-\be)q(z)$ as the exponent. In particular, the term $(1-\be)$ in the exponent provides sufficient gap in order to prove the end point version of the result as highlighted in the introduction. To do this, we use a standard covering argument followed by uniformizing the exponents which enables us to remove the term $(1-\be)$. Thus our main theorem now takes the following form:

\begin{theorem}
\label{main_theorem2}
Let  $M^+ > 1$ and let $r(\cdot)$ be log-H\"{o}lder continuous satisfying $1 \le r^- \leq r(\cdot) \leq r^+ < M^+ < \infty$. Then under the assumptions in Theorem \ref{main_theorem1}, there is a constant ${\gamma_0} \in (0,1/4)$  depending only on $\La_0$, $\La_1$, $\plog$, $r^{\pm}_{\log}$, $M^+$, $n$, such that if $(p(\cdot),\aa,\Om)$ is $(\gamma,\bs_0)$-vanishing for some $\ga \in (0,{\ga_0})$, then 
there exists a constant $C = C_{(\La_0, \La_1, \plog, r^{\pm}_{\log}, M^+, n, \Om_T, \bs_0)}>0$ such that the following global bound holds:
\begin{equation*}\label{main-r2}
\iint_{\Om_T} |\nabla u|^{p(z)r(z)} \ dz \le C \left\{ \lbr \iint_{\Om_T} |\bff|^{p(z)r(z)} \ dz \rbr^{\gh{1+\vt^+ \gh{M^+ -1}}(n+3)M^+ - (n+2)} +1\right\},
\end{equation*}
where the constant $\vt^+$ is given in \eqref{vt_def}.
\end{theorem}

\begin{remark}
\textcolor{black}{ If $p(\cdot) \equiv p$,  then we can take $d = \min \mgh{\frac2p,1}$ in \eqref{def_d} (see \cite{adimurthi2018sharp} for more details). Substituting this into \eqref{vt_def} yields
\begin{equation*}
\vt(z) \equiv \vt = \left\{
\begin{array}{lrl}
\frac{p}{2} &\text{if} &p \ge 2,\\
\frac{2p}{p(n+2)-2n} &\text{if} &\frac{2n}{n+2} < p < 2.
\end{array}
\right.
\end{equation*}
This is the standard scaling deficit coefficient introduced in \cite{AM}.}
\end{remark}

\section{Some useful inequalities}
\label{four}

In this section, we shall collect and prove in some cases well known estimates that will be used in subsequent sections. We first recall an integral version of Poincar\'e's inequality which was proved in \cite[Lemma 4.12]{adimurthi2017sharp}:
\begin{lemma}
\label{scaled_poincare}
 Let $\sss \in \logh$ and let $\bm_{p} \ge 1$ be given. Define  $\br_p := \min \left\{\frac{1}{2\bm_p}, \frac1{\om_n^{1/n}}, \frac12 \right\}$. Then for any $\phi \in W^{1,\sss}(B_{4r})$ with $4r < \br_p$ satisfying 
 \begin{equation*}\label{assumption_on_phi_no_weight} 
 \int_{B_{4r}} |\nabla \phi (x)|^{s(x)} \ dx + 1 \leq \bm_p, 
 \end{equation*} 
 the following estimate holds: 
 \[
  \int_{B_r} \lbr \frac{|\phi - \avg{\phi}{B_r}|}{\diam(B_r)}\rbr^{s(x)} \ dx \apprle_{(n,\slog)}   \int_{B_{r}} |\nabla \phi(x)|^{s(x)} \ dx + |B_{r}|,
 \]
where we have used the notation $\avg{\phi}{B_r} := \fint_{B_r} \phi(y)\ dy$.
%
%
 Since $\diam(B_r) = 2r \leq \br_p <1$, we also obtain
 \[
  \int_{B_r} {|\phi - \avg{\phi}{B_r}|}^{s(x)} \ dx \apprle_{(n,\slog)}  \int_{B_{r}} |\nabla \phi(x)|^{s(x)} \ dx + |B_{r}|.
 \]
\end{lemma}

Another Poincar\'e's inequality that will be needed is one where the function has a reasonable large zero set:
\begin{theorem}
 \label{measure_density_poincare}
 Let $\sss \in \logh$ and let $\bm_p\geq 1$ and $\varepsilon \in (0,1)$ be given. Define  $\br_p := \min \left\{\frac{1}{2\bm_p}, \frac1{\om_n^{1/n}}, \frac12 \right\}$.  For any $\phi \in W^{1,\pp}(B_{2r})$ with $2r < \br_p$ satisfying 
 \begin{gather*}
 |\{ N(\phi)\}| := |\{ x \in B_r : \phi(x) =0\}|> \varepsilon |B_r|\txt{and} \int_{B_{2r}} |\nabla \phi (x)|^{s(x)} \ dx + 1 \leq \bm_p, 
 \end{gather*}
 the following estimate holds:
 \[
  \int_{B_r} \lbr \frac{|\phi|}{\diam(B_r)}\rbr^{s(x)} \ dx \apprle_{(\slog,n,\ve)}  \int_{B_{r}} |\nabla \phi(x)|^{s(x)} \ dx + |B_{r}|.
 \]
\end{theorem}

We note that Theorem \ref{measure_density_poincare} is slightly different than the one proved in \cite[Theorem 4.13]{adimurthi2017sharp}. In order to obtain this improvement where the ball $B_r$ is the same on both sides of the inequality, we can repeat the arguments in the proof of \cite[Theorem 4.13]{adimurthi2017sharp} and combine them with the technical lemma from \cite[Lemma 3.4]{han2011elliptic}.

The next lemma that we need  is an estimate in $L\log L$-space which can be found in \cite{AM} and references therein:
\begin{lemma}
 \label{llogl}
 Let $\be >0$ and let $s>1$. Then for any $f \in L^s(\tom)$, we have
 \[
  \fint_{\tom} |f| \lbr[[] \log \lbr e + \frac{|f|}{\avg{|f|}{\tom}}\rbr \rbr[]]^{\be}\ dx \apprle_{(n,s,\be)} \lbr \fint_{\tom} |f|^s \ dx \rbr^{\frac{1}{s}},
 \]
where we have used the notation ${\avg{|f|}{\tom}} := \fint_{\tom} |f(y)|\ dy$.
\end{lemma}

We record some useful property as follows:
\begin{lemma}\label{useful int}
Let $\tom$ be an open set in $\RR^N$ and let $q > s \ge 0$. For $g \in L^1(\tom)$, we have
\begin{equation}\label{useful int1}
\int_{\tom} |g|_k^{q-s} |g|\ dx = (q-s) \int_0^k \al^{q-s-1} \int_{\{y\in \tom : |g(y)| > \al\}} |g(x)| \ dx d\al,
\end{equation}
where the truncation function $|g|_k := \min\mgh{|g|, k}$ for some constant $k>0$. If $g \in L^{q-s+1}(U)$, then \eqref{useful int1} also holds for $k=\infty$.
\end{lemma}

\begin{proof}
By Fubini's theorem, it is easy to check that Lemma \ref{useful int} holds.
\end{proof}

We also use the following technical lemma which was proved in \cite[Lemma 4.3]{han2011elliptic}:
\begin{lemma}\label{useful tech}
Let $g$ be a bounded nonnegative function in $[\tau_0, \tau_1]$ with $\tau_0 \ge 0$. Suppose that for $\tau_0 \le s_1 < s_2 \le \tau_1$, we have
\begin{equation*}
f(s_1) \le \theta f(s_2) + \frac{P_1}{(s_2-s_1)^k} + P_2,
\end{equation*}
for some $k,P_1,P_2 \ge 0$ and $\theta \in [0,1)$. Then for any $\tau_0 \le s_1 < s_2 \le \tau_1$, there holds
$$
f(s_1) \apprle_{(k,\theta)} \mgh{\frac{P_1}{(s_2-s_1)^k} + P_2}.
$$
\end{lemma}

\subsection{Maximal Function}

For any $f \in L^1(\RR^{n+1})$, let us now define the strong maximal function in $\RR^{n+1}$ as follows:
\begin{equation}
 \label{par_max}
 \mm(|f|)(x,t) := \sup_{\tQ \ni(x,t)} \fiint_{\tQ} |f(y,s)| \ dy \ ds,
\end{equation}
where the supremum is  taken over all parabolic cylinders $\tQ_{a,b}$ with $a,b \in \RR^+$ such that $(x,t)\in \tQ_{a,b}$. An application of the Hardy-Littlewood maximal theorem in $x-$ and $t-$ directions shows that the Hardy-Littlewood maximal theorem still holds for this type of maximal function (see \cite[Lemma 7.9]{Gary} for details):
\begin{lemma}
\label{max_bound}
 If $f \in L^1(\RR^{n+1})$, then for any $\al >0 $, there holds
 \[
  |\{ z \in \RR^{n+1} : \mm(|f|)(z) > \al\}| \leq \frac{5^{n+2}}{\al} \|f\|_{L^1(\RR^{n+1})},
 \]
 and if $f \in L^{\vartheta}(\RR^{n+1})$ for some $1 < \vartheta \leq \infty$, then there holds
 \[
  \| \mm(|f|) \|_{L^{\vartheta}(\RR^{n+1})} \leq C_{(n,\vartheta)} \| f \|_{L^{\vartheta}(\RR^{n+1})}.
 \]

\end{lemma}

\section{Approximations}
\label{four-two}

In this section, we describe gradient higher integrability type results and the approximations that will be made.

\subsection{Gradient higher integrability estimates}
In this subsection, let us collect a few important higher integrability results that will be used throughout the paper. In order to state the general theorems, let $\phi \in L^{\pp}\lbr -T,T;W_0^{1,\pp}(\Om)\rbr$ be a weak solution of 
\begin{equation}
 \label{pde_high}
\left\{ \begin{array}{rcll}
  \phi_t - \dv \aa(x,t,\nabla \phi) &=& -\dv (|\vec{f}|^{p(x,t)-2} \vec{f}) & \quad \text{in} \ \omt,\\
  \phi &=& 0 & \quad \text{on} \ \partial \Om \times (-T,T),
 \end{array}\right. 
\end{equation}
where the nonlinearity is assumed to satisfy \eqref{abounded} and \eqref{monotonicity}. Here the domain $\OO$ is assumed to satisfy a uniform measure density condition with constant $m_e$ as defined in Definition \ref{measure_def}.  Let us define 
\begin{equation}
\label{deb_bm_f}
\bm_{\vec{f}}:= \iint_{\Om_T} \lbr[[]|\vec{f}|^{p(z)}+1\rbr[]]\ dz + 1,
\end{equation}
which combined with \eqref{energy_u} shows
\[
\bm_{\phi}:= \iint_{\Om_T} \lbr[[]|\nabla \phi|^{p(z)}+1\rbr[]] \ dz + 1 \leq C_{(n,\plog,\lamot)}\bm_{\vec{f}}.
\]

The first result we recall is the higher integrability above the natural exponent. In the interior case, this was proved in \cite{AZ05,bogelein2011higher} whereas in the boundary case, using the measure density condition satisfied by $\Om$, the result was proved in \cite[Lemma 3.5]{byun2016nonlinear}. Using the unified intrinsic scaling approach, we can obtain the following modified higher integrability above the natural exponent:
\begin{theorem}
\label{high_weak}
Let $\tilde{\sigma}>0$ be given, then there exists $\tilde{\be}_1 = \tilde{\be}_1(n,\lamot,\plog,\Om) \in (0,\tilde{\sigma}]$ such that if $\vec{f} \in L^{\pp(1+\tilde{\sigma})}(\Om_T)$ and $\phi\in L^{\pp}\lbr-T,T;W_0^{1,\pp}(\Om)\rbr$ is a weak solution to \eqref{pde_high}, then $|\nabla \phi| \in L^{\pp(1+\be)}(\Om_T)$ for all $\be \in (0,\tilde{\be}_1]$. Moreover, with $\bm_{\vec{f}}$ defined as \eqref{deb_bm_f}, there exists a radius $\tilde{\rho}_1  = \tilde{\rho}_1(n,\plog,\lamot,\bm_{\vec{f}})$ such that for any $2\rho \in (0,\tilde{\rho}_1]$ and any $\mfz \in \overline{\Om} \times (-T,T)$, there holds
\[
\fiint_{K_{\rho}(\mfz)} |\nabla \phi|^{\pp(1+\be)} \ dz \apprle_{(n,\lamot,\plog,\Om)} \lbr \fiint_{K_{2\rho}(\mfz)}\lbr  |\nabla \phi|+ |\vec{f}|\rbr^{\pp} \ dz \rbr^{1+\be \tilde{\vt}_1} + \fiint_{K_{2\rho}(\mfz)}\lbr  |\vec{f}|+1\rbr^{\pp(1+\be)} \ dz,
\]
where the constant $\tilde{\vt}_1 = \tilde{\vt}_1(p(\mfz),n)\geq 1$.
\end{theorem}

We will also need an improved higher integrability result below the natural exponent. The following theorem was proved for a weaker class of solutions called \emph{very weak solutions}, but also holds true for \emph{weak solutions} as considered in this paper. The interior regularity in the singular case, i.e., when $\frac{2n}{n+2}<p^+\leq 2$, the result was proved in \cite{li2017very} and the interior regularity in the degenerate case, i.e., when $p^- \geq 2$, the result was proved in \cite{bogelein2014very}. Subsequently, using the \emph{unified intrinsic scaling}, this restriction can be removed and the full result up to the boundary with $\frac{2n}{n+2} < p^- \leq \pp\leq p^+ < \infty$ was proved in \cite{adimurthi2018sharp} for domains satisfying a uniform measure density condition as in Definition \ref{measure_def}. 

\begin{theorem}[\cite{adimurthi2018sharp}]
\label{high_very_weak}
Let $\tilde{\sigma}>0$ be given and suppose $\vec{f} \in L^{\pp(1+\tilde{\sigma})}(\Om_T)$ and $\phi\in L^{\pp}\lbr -T,T;W_0^{1,\pp}(\Om)\rbr$ is a weak solution to \eqref{pde_high}. With $\bm_{\vec{f}}$ defined as \eqref{deb_bm_f}, there exist radius $\tilde{\rho}_2  = \tilde{\rho}_2(n,\plog,\lamot,\bm_{\vec{f}})$ and  $\tilde{\be}_2 = \tilde{\be}_2(n,\lamot,\plog) \in (0,\tilde{\sigma}]$ with $\tilde{\be}_2 \leq \frac14$ such that for any $2\rho \in (0,\tilde{\rho}_2]$, $\be \in (0,\tilde{\be}_2]$ and any $\mfz \in \overline{\Om} \times (-T,T)$, there holds
\[
\fiint_{K_{\rho}(\mfz)} |\nabla \phi|^{\pp} \ dz \apprle_{(n,\lamot,\plog)} \lbr \fiint_{K_{2\rho}(\mfz)}\lbr  |\nabla \phi|+ |\vec{f}|\rbr^{\pp(1-\be)} \ dz \rbr^{1+\be \tilde{\vt}_2} + \fiint_{K_{2\rho}(\mfz)}\lbr  |\vec{f}|+1\rbr^{\pp} \ dz,
\]
where the constant $\tilde{\vt}_2 = \tilde{\vt}_2(n,p(\mfz))\ge1$.
\end{theorem}

\begin{remark}
For weak solutions, from the papers \cite{bogelein2011higher} and \cite{byun2016nonlinear}, the exponent $\tilde{\vt}_1$ in Theorem \ref{high_weak} was explicitly given by
\begin{equation*}
\tilde{\vt}_1 :=
\left\{
\begin{array}{ll}
\frac{p(\mfz)}{2} & \txt{if} p(\mfz) \geq 2,\\
\frac{2p(\mfz)}{p(\mfz)(n+2) -2n} &  \txt{if} \frac{2n}{n+2} < p(\mfz) < 2.
\end{array}\right.
\end{equation*}

On the other hand, using the unified intrinsic scaling approach and recalculating the estimates from \cite{bogelein2011higher}, we can obtain the following unified exponent $\tilde{\vt}_1 = \frac{1}{-\frac{n}{p(\mfz)}+\frac{d(n+2)}{2}}$ (recall the exponent $d$ from \eqref{def_d}) which holds in the full range $\frac{2n}{n+2} < p(\mfz) < \infty$.

For very weak solutions, in \cite{adimurthi2018sharp}, the exponent $\tilde{\vt}_2 := \frac{1}{-\frac{n}{p(\mfz)} + \frac{(n+2)d}{2} - \be}$ for any $\frac{2n}{n+2} < p(\mfz) < \infty$. Note that since $\be \leq \frac14$, one can uniformize the exponent $\tilde{\vt}_2 = \tilde{\vt}_2(n,p(\mfz))$ only, i.e., it does not depend on $\be$.

For the purposes of this paper, the explicitly computed exponents $\tilde{\vt}_1$ and $\tilde{\vt}_2$ will not be needed except for the following two  properties: firstly, we observe that $\tilde{\vt}_1, \tilde{\vt}_2 \geq 1$ and secondly, $\tilde{\vt}_1$ and $\tilde{\vt}_2$ can be made to  depend only on $n$ and $p(\mfz)$.
\end{remark}

Before we end this subsection, let us prove the following important corollary:
\begin{corollary}
\label{normalized_higher_integrability}
Let $\mfz \in \Om_T$ be any fixed point, and let $\al \geq 1$ be given. Suppose that $\phi$ and $\vec{f}$ solve
\begin{equation}
 \label{pde_local_norm}
\left\{ \begin{array}{rcll}
  \phi_t - \dv \aa(x,t,\nabla \phi) &=& -\dv (|\vec{f}|^{p(x,t)-2} \vec{f}) & \quad \text{in} \ K_{3r}^{\al}(\mfz),\\
  \phi &=& 0 & \quad \text{on} \ \partial_w K_{3r}^{\al}(\mfz).
 \end{array}\right. 
\end{equation}
Let $\be \leq \min\{ \tilde{\be}_1, \tilde{\be}_2\}$ where $\tilde{\be}_1$ is from Theorem \ref{high_weak} and $\tilde{\be}_2$ is from Theorem \ref{high_very_weak}.

Assume the following  are satisfied for some constants $\tilde{\bm}\geq 1$, $c_{\ast}$, $c_p$ and $\Ga$:
\begin{equation}\label{bm_1}
\iint_{K_{3r}^{\al}(\mfz)} |\nabla \phi|^{\pp(1-\be)} + |\vec{f}|^{\pp(1-\be)} + 1 \ dz \leq \tilde{\bm},
\end{equation}
\begin{equation}\label{hyp_one}
\fiint_{K_{3r}^{\al}(\mfz)} |\nabla \phi|^{\pp(1-\be)} + \lbr \fiint_{K_{3r}^{\al}(\mfz)} |\vec{f}|^{\pp(1-\be)\ka} \ dz \rbr^{\frac{1}{\ka}} \leq c_{\ast} \al^{1-\be} \quad \text{for some}\ \  \ka \geq \frac{1+\be}{1-\be}.
\end{equation}

Let $3r\leq \min\{\tilde{\rho}_1, \tilde{\rho}_2\}$ where $\tilde{\rho}_1$ is from Theorem \ref{high_weak} and $\tilde{\rho}_2$ is from Theorem \ref{high_very_weak},  furthermore, for the strictly positive constant (see \eqref{def_d}) defined by 
\begin{equation}
\label{def_d_0}
d_0 := \frac{d(n+2)}{2} - \frac{n}{p^-} >0,
\end{equation}
assume the following assumptions hold:
\begin{equation}
\label{one_5.18}
p^+_{K_{3r}^{\al}(\mfz)} - p^-_{K_{3r}^{\al}(\mfz)} \leq \modp(32r) \leq \min\left\{d_0p^-, d_0 p^- (p^--1)\right\} \txt{and} \al^{p^+_{K_{3r}^{\al}(\mfz)} - p^-_{K_{3r}^{\al}(\mfz)}} \leq c_p.
\end{equation}

Then for any  $\sigma \in (0,\be]$, the following estimate holds:
\begin{equation}\label{one_bnd}
\fiint_{K_{r}^{\al}(\mfz)} |\nabla \phi|^{\pp(1+\sigma)} \ dz \apprle_{(c_{\ast},c_p,p^-,p(\mfz))} \al^{1+\sigma}.
\end{equation}

\end{corollary}

\begin{proof}
From \eqref{pde_local_norm}, we see that  under the change of variables, $x := \scalex{\al}{\mfz} y$ and $t := \scalet{\al}{\mfz} \tau$, with 
\begin{gather*}
\phi_1(y,\tau) := \frac{\phi(y,\tau)}{\al^{\frac{d}{2}}}, \qquad  \vec{f}_1 := \al^{\frac{1-p(\mfz)}{p(\mfz)(p(y,\tau) -1)}} \vec{f}(y,\tau)  \txt{and} \bar{\bf a}(y,\tau,\zeta) := \al^{\frac{1-p(\mfz)}{p(\mfz)}} \aa(y,\tau,\al^{\frac{1}{p(\mfz)}} \zeta), 
\end{gather*}
the following equation is satisfied:
\begin{equation*}
\left\{ \begin{array}{rcll}
  \frac{d \phi_1(y,\tau)}{d\tau} - \dv_y \bar{\bf a}(y,\tau,\nabla_y \phi_1(y,\tau)) &=& -\dv_y (|\vec{f}_1(y,\tau)|^{p(y,\tau)-2} \vec{f}_1(y,\tau)) & \quad \text{in} \ K_{3r}(\mfz),\\
  \phi_1 &=& 0 & \quad \text{on} \ \partial_w K_{3r}(\mfz).
 \end{array}\right. 
\end{equation*}

From the assumptions \eqref{def_d_0}, \eqref{one_5.18} and \eqref{def_d}, it is easy to see that the following bounds hold:
\begin{gather}
-\frac{\pp(1-\be)}{p(\mfz)} + \frac{n}{p(\mfz)} - \frac{d(n+2)}{2} + 1 \leq \frac{p(\mfz)-\pp}{p(\mfz)} + \frac{n}{p^-} - \frac{d(n+2)}{2} \leq \frac{\modp(32r)}{p^-} - d_0 \overset{\eqref{one_5.18}}{\leq} 0, \label{one5.19}\\
\frac{(1-p(\mfz))\pp}{p(\mfz)(\pp -1)}+ \frac{n}{p(\mfz)} - \frac{d(n+2)}{2} + 1 \leq \frac{\pp - p(\mfz)}{p(\mfz) (\pp -1)} -d_0 \leq \frac{\modp(32r)}{p^-(p^--1)} - d_0 \overset{\eqref{one_5.18}}{\leq} 0.\label{one5.21}
\end{gather}

From a simple change of variables and using the fact that $\al \geq 1$, we see that 
\begin{equation}
\label{one5.17}
\begin{array}{rcl}
\iint_{K_{3r}(\mfz)} |\nabla_y \phi_1(y,\tau)|^{p(y,\tau)(1-\be)} \ dy\ d\tau 
& = & \iint_{K_{3r}^{\al}(\mfz)} \al^{-\frac{p(x,t)(1-\be)}{p(\mfz)}+ \frac{n}{p(\mfz)} -\frac{d(n+2)}{2} +1} \abs{\nabla_x \phi(x,t)}^{p(x,t)(1-\be)} \ dx\ dt \\
& \overset{\eqref{one5.19}}{\leq} &  \iint_{K_{3r}^{\al}(\mfz)}  \abs{\nabla_x \phi(x,t)}^{p(x,t)(1-\be)} \ dx\ dt.
\end{array}
\end{equation}

%

Analogously, we get
\begin{equation}
\label{5.20}
\begin{array}{rcl}
\iint_{K_{3r}(\mfz)} |\vec{f}_1(y,\tau)|^{p(y,\tau)(1-\be)} \ dy\ d\tau 
& = & \iint_{K_{3r}^{\al}(\mfz)} \al^{\frac{(1-p(\mfz))p(x,t)}{p(\mfz)(p(x,t) -1)}+ \frac{n}{p(\mfz)} - \frac{d(n+2)}{2} + 1}\abs{ \vec{f}(x,t)}^{p(x,t)(1-\be)} \ dx\ dt \\
& \overset{\eqref{one5.21}}{\leq} &  \iint_{K_{3r}^{\al}(\mfz)}  \abs{\vec{f}(x,t)}^{p(x,t)(1-\be)} \ dx\ dt.
\end{array}
\end{equation}

Thus combining \eqref{one5.17} and  \eqref{5.20} and using the hypothesis \eqref{bm_1}, we get
\begin{equation*}
\label{one5.26}
\iint_{K_{3r}(\mfz)}\lbr[[]|\nabla \phi_1(y,\tau)|^{p(y,\tau)(1-\be)} + |\vec{f}_1(y,\tau)|^{p(y,\tau)(1-\be)} + 1 \rbr[]]\ dy \ ds \leq \tilde{\bm}.
\end{equation*}

For the sake of simplicity, let us denote $p(y,\tau) = \tilde{p}(z)$ and $p(x,t) = p(z)$. We will now proceed with proving \eqref{one_bnd} as follows: 
\begin{equation}
\label{one5.27}
\begin{array}{rcl}
\fiint_{K_{r}^{\al}(\mfz)} |\nabla \phi|^{p(z)(1+\sigma)} \ dz & = & \fiint_{K_{r}(\mfz)} \al^{\frac{\tilde{p}(z)(1+\sigma)}{p(\mfz)} - \frac{n}{p(\mfz)} + \frac{d(n+2)}{2} -1}|\nabla \phi_1|^{\tilde{p}(z)(1+\sigma)} \ dz\\
& \leq & \fiint_{K_{r}(\mfz)} \al^{\frac{(\tilde{p}(z)-p(\mfz))(1+\sigma)}{p(\mfz)} - \frac{n}{p^+} + \frac{d(n+2)}{2} +\sigma}|\nabla \phi_1|^{\tilde{p}(z)(1+\sigma)} \ dz\\
& \overset{\redlabel{a5.25}{a}}{\apprle} & c_p^{\frac{2}{p^-}}\fiint_{K_{r}(\mfz)} \al^{- \frac{n}{p^+} + \frac{d(n+2)}{2} +\sigma}|\nabla \phi_1|^{\tilde{p}(z)(1+\sigma)} \ dz\\
& \overset{\redlabel{b5.25}{b}}{\apprle} & c_p^{\frac{2}{p^-}}\fiint_{K_{r}(\mfz)} \al^{1 +\sigma}|\nabla \phi_1|^{\tilde{p}(z)(1+\sigma)} \ dz.
\end{array}
\end{equation}
To obtain \redref{a5.25}{a}, we made use of \eqref{one_5.18} and the fact $1+\sigma \leq 2$ and to obtain \redref{b5.25}{b}, we made use of the bound $- \frac{n}{p^+} + \frac{d(n+2)}{2} \leq 1$ which follows from \eqref{def_d}.


We can now apply Theorem \ref{high_weak} to obtain the higher integrability from $\tilde{p}(z)$ to $\tilde{p}(z)(1+\sigma)$ and  apply Theorem \ref{high_very_weak} to obtain the higher integrability from $\tilde{p}(z)(1-\be)$ to $\tilde{p}(z)$. Thus the expression on the right of \eqref{one5.27} can be estimated as 
\begin{equation}
\label{one5.30}
\begin{array}{rcl}
\fiint_{K_{r}(\mfz)} |\nabla \phi_1|^{\tilde{p}(z)(1+\sigma)} \ dz  
 & \apprle & \bgh{  \lbr \fiint_{K_{3r}(\mfz)} (|\nabla \phi_1|+ |\vec{f}_1|)^{\tilde{p}(z)(1-\be)}\ dz\rbr^{1+ \be \tilde{\vt}_2}  +  \fiint_{K_{3r}(\mfz)}  |\vec{f}_1|^{\tilde{p}(z)}  \ dz }^{1+ \sigma \tilde{\vt}_1} \\ && \qquad +  \fiint_{K_{3r}(\mfz)} |\vec{f}_1|^{\tilde{p}(z) (1+\sigma)} \ dz + 1.
\end{array}
\end{equation}

In order to prove \eqref{one_bnd}, it is sufficient to bound \eqref{one5.30} by a constant from which the result will follow by using \eqref{one5.27}.
In order to do this, we scale back to get
\begin{equation}\label{on5.26}
\begin{array}{rcl}
\fiint_{K_{3r}(\mfz)} |\nabla \phi_1|^{\tilde{p}(z)(1-\be)} \ dz& = & \frac{|K_{3r}^{\al}(\mfz)|}{|K_{3r}(\mfz)|}\fiint_{K_{3r}^{\al}(\mfz)} \al^{-\frac{\tilde{p}(z)(1-\be)}{p(\mfz)}+\frac{n}{p(\mfz)}- \frac{d(n+2)}{2} +1}|\nabla \phi|^{p(z)(1-\be)}\ dz\\
& \overset{\redlabel{a5.26}{c}}{\leq} & \al^{(1-\be) \lbr\frac{p(\mfz)-{p}^-_{K_{3\rho}^{\al}(\mfz)}}{p(\mfz)}\rbr} \al^{-(1-\be)}\fiint_{K_{3r}^{\al}(\mfz)} |\nabla \phi|^{p(z)(1-\be)}\ dz\\
& \overset{\redlabel{b5.26}{d}}{\apprle} & c_{\ast} c_p^{\frac{1}{p^-}}.
\end{array}
\end{equation}
To obtain \redref{a5.26}{c}, we used the fact that $\frac{|K_{3r}^{\al}(\mfz)|}{|K_{3r}(\mfz)|} = \scalexn{\al}{\mfz} \scalet{\al}{\mfz}$ and to obtain \redref{b5.26}{d}, we made use of \eqref{hyp_one} and \eqref{one_5.18} along with the trivial bound $1-\be \leq 1$.

To estimate the terms containing $\vec{f}_1$ in \eqref{one5.30}, let us denote $\varpi$ to be either $(1-\be)$, $1$ or  $(1+\sigma)$ and estimate $\fiint_{K_{3r}(\mfz)} |\vec{f}_1|^{\tilde{p}(z)\varpi}\ dz$ as follows:

\begin{equation}\label{on5.27}
\begin{array}{rcl}
\fiint_{K_{3r}(\mfz)} |\vec{f}_1|^{\tilde{p}(z)\varpi} & \overset{\redlabel{a5}{e}}{=} & \frac{|K_{3r}^{\al}(\mfz)|}{|K_{3r}(\mfz)|}\fiint_{K_{3r}^{\al}(\mfz)} \al^{\frac{(1-p(\mfz))p(x,t)\varpi}{p(\mfz)(p(x,t) -1)}+\frac{n}{p(\mfz)}- \frac{d(n+2)}{2} +1}|\vec{f}|^{p(z)\ka}\ dz\\
& \overset{\redlabel{b5}{f}}{\leq} & \al^{-\frac{\lbr {p}^+_{K_{3\rho}^{\al}(\mfz)}\rbr'\varpi}{p(\mfz)'}} \fiint_{K_{3r}^{\al}(\mfz)} |\vec{f}|^{p(z)\varpi}\ dz\\
& \overset{\redlabel{c5}{g}}{\leq} & \al^{\varpi \lbr \frac{p^+_{K_{3\rho}^{\al}(\mfz)}-p(\mfz)}{p(\mfz) ({p}^+_{K_{3\rho}^{\al}(\mfz)}-1}\rbr}\al^{-\varpi} \lbr \fiint_{K_{3r}^{\al}(\mfz)} |\vec{f}|^{p(z)(1-\be)\ka}\ dz\rbr^{\frac{\varpi}{(1-\be)\ka}}\\
& \overset{\redlabel{d5}{h}}{\apprle} & c_{\ast}c_p^{\frac{2}{(p^-)^2}}.
\end{array}
\end{equation}
To obtain \redref{a5}{e}, we performed the usual change of variables, to obtain \redref{b5}{f}, we used the fact that $\frac{|K_{3r}^{\al}(\mfz)|}{|K_{3r}(\mfz)|} = \scalexn{\al}{\mfz} \scalet{\al}{\mfz}$, to obtain \redref{c5}{g}, we used the fact that $\ka (1-\be) \geq \varpi$ from \eqref{hyp_one} and finally to obtain \redref{d5}{h}, we made use of \eqref{hyp_one} and \eqref{one_5.18} along with the bound $\varpi \leq 2$.

Thus combining \eqref{on5.26} and \eqref{on5.27} into \eqref{one5.30} and finally substituting the resulting expression into \eqref{one5.27}, we see that for some $\tilde{\vt} = \tilde{\vt}(n,p(\mfz))$, there holds
\[
\fiint_{K_{r}^{\al}(\mfz)} |\nabla \phi|^{\pp(1+\sigma)} \ dz \leq C_{(c_{\ast},c_p,p^-,p(\mfz))} \al^{1+\sigma},
\]
which completes the proof.
\end{proof}

\subsection{Approximations}
In this subsection, let $\al \geq 1$ be a given constant, let $\rho$ be as in Remark \ref{remark_radius}, and let $\mfz = (\mfx,\mft) \in \Om_T$ be any fixed point. Also note that the existence  of all the solutions considered below follows from Proposition \ref{ext_sol}.

First, let us consider the unique weak solution  $w \in C^0\lbr I_{4\rho}^{\al}(\mft);L^2(\Om_{4\rho}^{\al}(\mfx)\rbr  \cap L^{\pp}\lbr I_{4\rho}^{\al}(\mft);W^{1,\pp}(\Om_{4\rho}^{\al}(\mfx))\rbr$ solving
\begin{equation}
 \label{wapprox_int}
\left\{ \begin{array}{rcll}
  w_t - \dv \aa(x,t,\nabla w) &=& 0 & \quad \text{in} \ K_{4\rho}^{\al}(\mfz),\\
  w &=&u & \quad \text{on} \ \partial_p K_{4\rho}^{\al}(\mfz).
 \end{array}\right. 
\end{equation}
This is possible, since \eqref{basic_pde} shows $u \in L^{\pp}\lbr I_{4\rho}^{\al}(\mft);W^{1,\pp}(\Om_{4\rho}^{\al}(\mfx)\rbr$ and $\ddt{u} \in L^{\pp}\lbr I_{4\rho}^{\al}(\mft);W^{1,\pp}(\Om_{4\rho}^{\al}(\mfx)\rbr'$.

We can now compare the  solutions of \eqref{basic_pde} and \eqref{wapprox_int} to  get the following lemma:
\begin{lemma}
\label{energy_diff_est}
For any $\rho >0$ and any weak solution $w$ to \eqref{wapprox_int}, the following estimate holds:
\begin{gather}
\iint_{K_{4\rho}^{\al}(\mfz)} |\nabla w - \nabla u|^{p(z)} \ dz \apprle_{(n,\plog,\lamot)} \iint_{K_{4\rho}^{\al}(\mfz)} |\nabla u|^{p(z)} + |\bff|^{p(z)} + 1 \ dz, \label{diff_energy_w}\\
\iint_{K_{4\rho}^{\al}(\mfz)} |\nabla w|^{p(z)} \ dz \apprle_{(n,\plog,\lamot)}  \iint_{K_{4\rho}^{\al}(\mfz)} |\nabla u|^{p(z)} + |\bff|^{p(z)} + 1 \ dz. \label{energy_w}
\end{gather}
\end{lemma}
The proof of Lemma \ref{energy_diff_est} follows by taking $u-w$ as a test function in \eqref{basic_pde} and \eqref{wapprox_int} (see for example \cite[(4.11)]{byun2016nonlinear} for the proof of \eqref{diff_energy_w}). A simple application of triangle inequality  to \eqref{diff_energy_w} implies \eqref{energy_w}.

\begin{lemma}
\label{sobolev_reg_w}
Let $2\rho\leq \rho_0$ with $\rho_0$ as in Remark \ref{remark_radius}, then any weak solution $w\in  L^{\pp}\lbr I_{4\rho}^{\al}(\mft);W^{1,\pp}(\Om_{4\rho}^{\al}(\mfx))\rbr$ has the improved regularity $\nabla w \in L^{p(\mfz)} \gh{K_{3\rho}^{\al}(\mfz)}$. 
\end{lemma}
\begin{proof}
Since $\rho$ satisfies Remark \ref{remark_radius}, we can apply  Theorem \ref{high_weak} to \eqref{wapprox_int} which implies $\nabla w \in L^{\pp(1+\be)}K_{3\rho}^{\al}(\mfz)$ for any $\be \in (0,\be_0]$ with $\be_0$ as in Remark \ref{high_int_remark}. As a consequence, we have the following sequence of estimates
\begin{equation*}
\begin{array}{rcl}
\fiint_{K_{3\rho}^{\al}(\mfz)} |\nabla w|^{p(\mfz)} \ dz & =& \fiint_{K_{3\rho}^{\al}(\mfz)} |\nabla w|^{p(\mfz)\frac{\pp(1+\be_0)}{\pp(1+\be_0)}} \ dz \\
& \apprle& \fiint_{K_{3\rho}^{\al}(\mfz)} \lbr  |\nabla w|+1\rbr^{\pp(1+\be_0)\frac{p^+_{K_{3\rho}^{\al}(\mfz)}}{p^-_{K_{3\rho}^{\al}(\mfz)}(1+\be_0)}} \ dz \\
& \overset{\redlabel{a1}{a}}{\apprle}&  \fiint_{K_{3\rho}^{\al}(\mfz)} \lbr  |\nabla w|+1\rbr^{\pp(1+\be_0)} \ dz\\
& \overset{\redlabel{b1}{b}}{\apprle} & \lbr \fiint_{K_{4\rho}^{\al}(\mfz)} \lbr  |\nabla w|+1\rbr^{\pp} \ dz\rbr^{1+\be_0 \vt_0}.
\end{array}
\end{equation*}
To obtain \redref{a1}{a}, we made use of \descref{R2}{R2} which implies ${\frac{p^+_{K_{3\rho}^{\al}(\mfz)}}{p^-_{K_{3\rho}^{\al}(\mfz)}(1+\be)}}\leq 1$ and to obtain \redref{b1}{b}, we made use of Theorem \ref{high_weak} along \descref{B3}{B3}.

\end{proof}
 
 We will also need the following regularity with respect to the time derivative of the weak solution $w$ to \eqref{wapprox_int} which will enable us to use $w$ as boundary data so that Proposition \ref{ext_sol} can be applied. 
 \begin{lemma}
 \label{time_reg_w}
 We have $\ddt{w} \in L^{p(\mfz)}\lbr I_{3\rho}^{\al}(\mft);W^{1,p(\mfz)}(\Om_{3\rho}^{\al}(\mfx))\rbr'$.
 \end{lemma}
 \begin{proof}
 In order to prove the lemma, from \eqref{wapprox_int}, we see that it is sufficient to show $\aa(x,t,\nabla w) \in L^{\frac{p(\mfz)}{p(\mfz)-1}}(K_{3\rho}^{\al}(\mfz))$. We show this as follows:
 \begin{equation}
 \label{5.15}
 \begin{array}{rcl}
 \iint_{K_{3\rho}^{\al}(\mfz)} |\aa(x,t,\nabla w)|^{\frac{p(\mfz)}{p(\mfz)-1}} \ dz & \overset{\eqref{abounded}}{\apprle}&  \iint_{K_{3\rho}^{\al}(\mfz)} (|\nabla w|+1)^{(\pp-1)\frac{p(\mfz)}{p(\mfz)-1}} \ dz\\
 & \overset{\redlabel{a2}{a}}{\leq} & \iint_{K_{3\rho}^{\al}(\mfz)} (|\nabla w|+1)^{\pp\lbr 1+ \frac{p^+_{K_{3\rho}^{\al}(\mfz)}-p^-_{K_{3\rho}^{\al}(\mfz)}}{p^-_{K_{3\rho}^{\al}(\mfz)}-1}\rbr} \ dz. 
 \end{array}
 \end{equation}
 To obtain \redref{a2}{a}, we used the following sequence of estimates on $K_{3\rho}^{\al}(\mfz)$:
 \begin{equation*}
 (\pp -1) \frac{p(\mfz)}{p(\mfz) -1} \leq (p^+_{K_{3\rho}^{\al}(\mfz)} -1) \frac{p(\mfz)}{p(\mfz) -1} \leq (p^+_{K_{3\rho}^{\al}(\mfz)} -1) \frac{p^-_{K_{3\rho}^{\al}(\mfz)}}{p^-_{K_{3\rho}^{\al}(\mfz)} -1} \leq \pp \lbr 1 + \frac{p^+_{K_{3\rho}^{\al}(\mfz)}-p^-_{K_{3\rho}^{\al}(\mfz)}}{p^-_{K_{3\rho}^{\al}(\mfz)}-1}\rbr.
 \end{equation*}

Using Remark \ref{remark_def_p_log} with the observation $\al \geq 1$ which implies $K_{3\rho}^{\al}(\mfz) \subset K_{3\rho}(\mfz)$,  we see that
\begin{equation}
\label{5.16}
\frac{p^+_{K_{3\rho}^{\al}(\mfz)}-p^-_{K_{3\rho}^{\al}(\mfz)}}{p^-_{K_{3\rho}^{\al}(\mfz)}-1} \leq\frac{\modp(6\rho)}{p^-_{K_{3\rho}^{\al}(\mfz)}-1}\leq\frac{\modp(6\rho)}{p^- -1} \overset{\text{\descref{R5}{R5}}}{\leq} \be_0.
\end{equation}

Substituting \eqref{5.16} into \eqref{5.15} and making use of Theorem \ref{high_weak} (where $\tilde{\be}_1$ is obtained), we get
\[
\begin{array}{rcl}
\iint_{K_{3\rho}^{\al}(\mfz)} |\aa(x,t,\nabla w)|^{\frac{p(\mfz)}{p(\mfz)-1}} \ dz & \apprle&  \iint_{K_{3\rho}^{\al}(\mfz)} (|\nabla w|+1)^{\pp\lbr 1+\be_0 \rbr} \ dz\\
& \overset{\redlabel{a3}{b}}{\apprle} & |K_{3\rho}^{\al}(\mfz)| \lbr \fiint_{K_{3\rho}^{\al}(\mfz)} (1 + |\nabla w|)^{\pp(1-\be_0)} \ dz \rbr^{(1+\be_0\vt_0)(1+\be_0\vt_0)}\\
& \overset{\redlabel{b2}{c}}{\apprle} & |K_{3\rho}^{\al}(\mfz)| \lbr \fiint_{K_{3\rho}^{\al}(\mfz)} (1 + |\nabla w|)^{\pp(1-\be_0)} \ dz \rbr^{1+\be_0c_0}.
\end{array}
\]
To obtain \redref{a3}{b}, we have used Theorem \ref{high_weak} and Theorem \ref{high_very_weak} along with \descref{B3}{B3} and to obtain \redref{b2}{c}, we have used the fact that $\be_0 <1$ and $\vt_0 \geq 1$. 
This completes the proof of the lemma. 
\end{proof}

Let us now construct an averaged operator which will be needed.
For any $\al\ge 1$ and any $4\rho \leq \rho_0$, let us define the following vector valued function $\bb:  K_{3\rho}^{\al}(\mfz) \times \RR^n \rightarrow \RR^n$ by
\begin{equation}
\label{def_bb}
\bb(z,\zeta):= \aa(z,\zeta) \lbr \mu^2 + |\zeta|^2 \rbr^{\frac{p(\mfz)-p(z)}{2}}.
\end{equation}

From direct computations (see \cite[(4.18)]{byun2016nonlinear}), we see that the following bounds are satisfied:
\begin{equation}\label{bbounded}\begin{array}{c}
(\mu^2 + |\zeta|^2)^{\frac12} |D_{\zeta} \bb(z,\zeta)| + |\bb(z,\zeta)| \leq 3\La_1 (\mu^2 + |\zeta|^2)^{\frac{p(\mfz) -1}{2}}, \\
(\mu^2 + |\zeta|^2 )^{\frac{p(\mfz)-2}{2}} |\eta|^2 \frac{\La_0}{2} \leq \iprod{D_{\zeta}\bb(z,\zeta)\eta}{\eta}.
\end{array}\end{equation}
In particular, the operator $\bb(\cdot,\zeta)$ which is defined on $K_{3\rho}^{\al}(\mfz)$ is a constant exponent operator.

\begin{description}[leftmargin=*]
\item[Interior case:] Subsequently, in this case, i.e., when $K_{3\rho}^{\al}(\mfz) = Q_{3\rho}^{\al}(\mfz)\subset \Om_T$, we define another averaged operator $\bbb: \RR^n \times (\mft - \scalet{\al}{\mfz} 9\rho^2, \mft + \scalet{\al}{\mfz} 9\rho^2) \to \RR^n$ by 
\begin{equation*}
\label{avg_bbb}
\bbb(t,\zeta):= \fint_{B_{\scalex{\al}{\mfz}3\rho}(\mfx)} \bb(y,\mft,\zeta) \ dy.
\end{equation*}
From \eqref{small_aa}, we see that 
\[
\fiint_{K^{\al}_{3\rho}(\mfz)} \sup_{\zeta \in \RR^n} \frac{\abs{\bbb(t,\zeta)- \bb(z,\zeta)}}{\lbr\mu^2 + |\zeta|^2 \rbr^{\frac{p(\mfz)-1}{2}}} \ dz \leq \fiint_{Q^{\al}_{3\rho}(\mfz)} \Th(\aa, B_{4\rho}^{\al}(\mfx))(z) \ dz \leq \ga.
\]
In the above estimate, we have used the fact $\al\geq 1$ which implies $\scalex{\al}{\mfz} \leq 1$. 

\item[Boundary case:] Subsequently, in this case we make use of the $(\ga, \bs_0)$-Reifenberg flat condition, i.e., when $K_{3\rho}^{\al}(\mfz) = B_{3\rho}^{\al}(\mfx) \cap \Om \times I_{3\rho}^{\al}(\mft)$ and 
\[
B^{\al,+}_{3\rho}(\mfx) \subset \Om_{3\rho}(\mfx) \subset B_{3\rho}^{\al} \cap \{x_n > -3 \scalex{\al}{\mfz} \ga \rho\},
\]
we define another averaged operator $\bbb: (\mft - \scalet{\al}{\mfz} 9\rho^2, \mft + \scalet{\al}{\mfz} 9\rho^2) \times \RR^n  \to \RR^n$ by 
\begin{equation*}
\label{avg_bbbb}
\bbb(t,\zeta):= \fint_{B_{\scalex{\al}{\mfz}3\rho}^+(\mfx)} \bb(y,\mft,\zeta) \ dy.
\end{equation*}
From \eqref{small_aa}, we see that 
\[
\fiint_{Q^{\al,+}_{3\rho}(\mfz)} \sup_{\zeta \in \RR^n} \frac{\abs{\bbb(t,\zeta)- \bb(z,\zeta)}}{\lbr\mu^2 + |\zeta|^2 \rbr^{\frac{p(\mfz)-1}{2}}} \ dz = \fiint_{Q^{\al,+}_{3\rho}(\mfz)} \Th(\aa, B_{3\rho}^{\al,+}(\mfx))(z) \ dz \leq 4 \fiint_{Q^{\al}_{3\rho}(\mfz)} \Th(\aa, B_{3\rho}^{\al}(\mfx))(z) \ dz\leq 4 \ga.
\]
In the above estimate, we have used the fact $\al\geq 1$ which implies $\scalex{\al}{\mfz} \leq 1$. 

\end{description}

From Lemma \ref{sobolev_reg_w} and Lemma \ref{time_reg_w}, we can now define the following approximation:
\begin{equation}
 \label{vapprox_bnd}
\left\{ \begin{array}{rcll}
  v_t - \dv \bbb(t,\nabla v) &=& 0 & \quad \text{in} \ K_{3\rho}^{\al}(\mfz),\\
  v &=& w & \quad \text{on} \ \pa_p K_{3\rho}^{\al}(\mfz),
 \end{array}\right. 
\end{equation}
which admits a unique weak solution $v \in C^0\lbr I_{3\rho}^{\al}(\mft);L^2(\Om_{3\rho}^{\al}(\mfx)\rbr \cap L^{p(\mfz)}\lbr I_{3\rho}^{\al}(\mft);W^{1,p(\mfz)}(\Om_{3\rho}^{\al}(\mfx))\rbr$ since Proposition \ref{ext_sol} is applicable.
%
%
%

In the interior case, it is well known that the weak solution $v$ has locally Lipschitz bounds (see \cite{DiB1} for details). On the other hand, in the boundary case, we need to make one further approximation in which we consider a weak solution $\ov \in  C^0\lbr I_{2\rho}^{\al}(\mft);L^2(\Om_{2\rho}^{\al,+}(\mfx)\rbr \cap L^{p(\mfz)}\lbr I_{2\rho}^{\al}(\mft);W^{1,p(\mfz)}(\Om_{2\rho}^{\al,+}(\mfx))\rbr$ solving
\begin{equation}
 \label{Vapprox_bnd}
\left\{ \begin{array}{rcll}
  \ov_t - \dv \bbb(t,\nabla \ov) &=& 0 & \quad \text{in} \ Q_{2\rho}^{\al,+}(\mfz),\\
  \ov &=& 0 & \quad \text{on} \  \pa_w Q_{2\rho}^{\al,+}(\mfz).
 \end{array}\right. 
\end{equation}

\begin{lemma}
\label{existence_ov}
For any $\ve \in (0,1)$, there exists $\ga = \ga(n,\lamot,\plog,\ve)>0$ such that if $v$ is the weak solution of \eqref{vapprox_bnd}, then there is a weak solution $\ov \in  C^0\lbr I_{2\rho}^{\al}(\mft);L^2(\Om_{2\rho}^{\al,+}(\mfx)\rbr \cap L^{p(\mfz)}\lbr I_{2\rho}^{\al}(\mft);W^{1,p(\mfz)}(\Om_{2\rho}^{\al,+}(\mfx))\rbr$ solving \eqref{Vapprox_bnd} such that 
\begin{equation}\label{ov est1}
\fiint_{Q^{\al,+}_{2\rho}(\mfz)} |\nabla v - \nabla \ov|^{p(\mfz)} \ dz \le \ve^{p(\mfz)} \fiint_{K_{3\rho}^{\al}(\mfz)} |\nabla v|^{p(\mfz)} \ dz.
\end{equation}
Furthermore, we have
\begin{gather}\label{ov est2}
\sup_{Q^{\al,+}_{\rho}(\mfz)} |\nabla \ov| \apprle_{(n,\plog,\lamot)} \lbr \fiint_{Q^{\al,+}_{2\rho}(\mfz)} |\nabla \ov|^{p(\mfz)} \ dz + 1\rbr^{\frac{1}{p(\mfz)}}.
\end{gather}
\end{lemma}
\begin{proof}
We will prove the lemma by scaling. Define the rescaled functions	
\begin{gather*}
V_{\al,\rho} (y,s) := \frac{1}{\al^{\frac{d}{2}}\rho} \ov\lbr \scalex{\al}{\mfz} \rho x, \scalet{\al}{\mfz} \rho^2 t \rbr \txt{and} {\overline{\mathbf{b}}}_{\al,\rho}(t,\zeta):= \al^{\frac{1-p(\mfz)}{p(\mfz)}} \bbb\lbr\al^{\frac{1}{p(\mfz)}}\zeta, \scalet{\al}{\mfz}\rho^2 t\rbr,
\end{gather*}
under the change of variables $x = \scalex{\al}{\mfz} ry$ and $t = \scalet{\al}{\mfz}\rho^2 s$. We then see that $(x,t) \in Q_{2\rho}^{\al,+}(\mfz)$ implies $(y,s) \in Q_{2}^{+}(\mfz)$. From that fact that $\ov$ solves \eqref{Vapprox_bnd}, we have
\begin{equation*}
\begin{array}{lll}
0 & = \ddt{\ov}(x,t) - \dv_x \bbb(\nabla_x \ov(x,t),t) \\
& = \frac{1}{\al^{-1+\frac{d}{2}}\rho} \lbr \dds{V_{\al,\rho}}(y,s) - \dv_y  {\overline{\mathbf{b}}}_{\al,\rho}\lbr \nabla_y V_{\al,\rho}(y,s),s\rbr \rbr \txt{for} (y,s) \in Q_{2}^{+}(\mfz).
\end{array}
\end{equation*}

In particular, we see that $V_{\al,\rho}(y,s)$ is a weak solution of 
\[
\left\{ \begin{array}{ll}
\dds{V_{\al,\rho}}(y,s) - \dv_y  {\overline{\mathbf{b}}}_{\al,\rho}\lbr \nabla_y V_{\al,\rho}(y,s),s\rbr = 0 & \txt{in} Q_{2}^{+}(\mfz), \\
V_{\al,\rho} = 0 & \txt{on} Q_2(\mfz) \cap \{y_n = 0\}. 
\end{array}\right.
\]

From \cite[Theorem 1.6]{lieberman1993boundary}, we obtain the estimate
\[
\sup_{Q^+_{1}(\mfz)} |\nabla V_{\al,\rho}| \leq C_{(n,\plog,\lamot)} \lbr \fiint_{Q^{+}_{2}(\mfz)} |\nabla V_{\al,\rho}|^{p(\mfz)} \ dz + 1\rbr^{\frac{1}{p(\mfz)}},
\]
which implies the estimate \eqref{ov est2}. Moreover, a similar argument of \cite[Lemma 3.8]{BOS1} yields the estimate \eqref{ov est1}.
\end{proof}

\subsection{Fixing the size of solutions}
\label{size_fix}

Let us define 
\begin{equation}
\label{def_M_0}
\bm_0 := \iint_{\Om_T} \lbr[[]|\bff|^{p(z)} + 1\rbr[]] \ dz + 1.
\end{equation}

 From \eqref{energy_u}, we see that
 \begin{equation}
 \label{size_u_f}
 \bm_u \leq C_{(n,\plog,\lamot)} \bm_0 \txt{where we have set}\bm_u := \iint_{\Om_T} \lbr[[]|\nabla u|^{p(z)}+1\rbr[]] \ dz + 1.
 \end{equation}

From \eqref{energy_w} (which holds for any $\rho>0$), we see that there holds
 \begin{equation}
 \label{size_w}
 \bm_w \leq C_{(n,\plog,\lamot)} \bm_0 \txt{where we have set}\bm_w := \iint_{K_{4\rho}^{\al}(\mfz)} \lbr[[]|\nabla w|^{p(z)}+1\rbr[]] \ dz + 1.
 \end{equation}

\section{First difference estimate below the natural exponent}
\label{five}

In this section, we will prove a difference estimate between the weak solution of \eqref{basic_pde} and the weak solution of \eqref{wapprox_int}. To do this, we will use the method of Lipschitz truncation developed by \cite{KL} which is modified for use in the current setting in Appendix \ref{lipschitz_truncation}.
\begin{theorem}\label{first_diff_thm}
Let $\al \geq 1$ be fixed, then there exists $\tilde{\rho}_3 = \tilde{\rho}_3(n,\plog,\lamot,\bm_0)$ such that for any $128 \rho \leq \tilde{\rho}_3$  and for any $\ve \in (0,1]$, there exists $\tilde{\be}_3 = \tilde{\be}_3(n,\lamot,\plog)$ such that for any $\be \in (0,\tilde{\be}_3]$, there holds the estimate
\begin{equation*}
\label{diff_est_one}
\fiint_{K_{4\rho}^{\al}(\mfz)} |\nabla u - \nabla w|^{\pp(1-\be)} \ dz \leq \ve \fiint_{K_{4\rho}^{\al}(\mfz)} |\nabla u|^{\pp(1-\be)}\ dz  + C_{(n,\lamot,\plog)} \fiint_{K_{4\rho}^{\al}(\mfz)} \lbr[[]|\bff|^{\pp(1-\be)} + 1 \rbr[]]\ dz. 
\end{equation*}
Here $u$ is the weak solution of \eqref{basic_pde} and $w$ is the weak solution to \eqref{wapprox_int}. 
\end{theorem}

\begin{proof}

Let us denote 
\begin{equation*}
\label{def_s_a}
s:= \scalet{\al}{\mfz} (4\rho)^2,
\end{equation*}
and we consider the following cut-off function {$\zv \in C^{\infty} (\RR)$} such that $0 \leq \zv(t) \leq 1$ and 
\begin{equation*}
\label{def_zv}
\zv(t) = \left\{ \begin{array}{ll}
                1 & \text{for} \ t \in (\mft-s+\ve,\mft+s-\ve),\\
                0 & \text{for} \ t \in (-\infty,\mft-s)\cup (\mft+s,\infty).
                \end{array}\right.
\end{equation*}

It is easy to see that 
\begin{equation*}
\label{bound_zv}
\begin{array}{c}
\zv'(t) = 0 \ \txt{for} \  t \in (-\infty,\mft-s) \cup (\mft-s+\ve,\mft+s-\ve)\cup (\mft+s,\infty), \\
|\zv'(t)| \leq \frac{c}{\ve}\  \txt{for} \  t \in (\mft-s,\mft-s+\ve) \cup (\mft+s-\ve,\mft+s). 
\end{array}
\end{equation*}
Without loss of generality, we shall always take $2h \leq \ve$ since we will take limits in the following order $\lim_{\ve \rightarrow 0} \lim_{h \rightarrow 0}$.

We shall use $\vlh(z) \zv(t)$ as a test function in \eqref{basic_pde} and \eqref{wapprox_int} where $\vlh$ is as constructed in Appendix \ref{lipschitz_truncation} (more specifically in \eqref{lipschitz_function}). Thus we get
\begin{equation*}
\begin{array}{l}
L_1 + L_2 :=\iint_{{K_{4\rho}^{\al}(\mfz)}} \ddt{[u-w]_h} \vlh \zv \ dx \ dt + \iint_{{K_{4\rho}^{\al}(\mfz)}} \iprod{[A(x,t,\nabla u) - A(x,t,\nabla w)]_h}{\nabla \vlh} \zv \ dx \ dt \\
\hspace*{6cm} = \iint_{{K_{4\rho}^{\al}(\mfz)}} \iprod{[|\bff|^{\pp-2} \bff]_h}{\nabla \vlh} \zv  \ dx \ dt=: L_3.
\end{array}
\end{equation*}

\begin{description}
\item[Estimate for $L_1$:] Setting $\elam^{\tau} = \{(x,t) \in \elam: t=\tau\}$ where $\elam$ is as defined in \eqref{elambda}, we get
\begin{equation*}
   \label{6.38}
   \begin{array}{ll}
    L_1     
    & = \int_{\mft-s}^{\mft+s} \int_{{\Om_{4\rho}^{\al}(\mfx)}\setminus \elam^{\tau}}  \dds{\vlh}  (\vlh-\vh) \zv(s)\ dy \ d\tau +    \int_{\mft-s}^{\mft+s} \int_{{\Om_{4\rho}^{\al}(\mfx)}}  \frac{d{\lbr \lbr[[](\vh)^2 - (\vlh - \vh)^2\rbr[]]\zv(\tau) \rbr }}{d\tau} \ dy \ d\tau \\
    & \qquad - \int_{\mft-s}^{\mft+s} \int_{{\Om_{4\rho}^{\al}(\mfx)}} \dds{\zv} \lbr \vh^2 - (\vlh - \vh)^2 \rbr \ dy \ d\tau\\
    & := J_2 + J_1(\mft+s) - J_1(\mft-s) - J_3,
   \end{array}
  \end{equation*}
  where we have set 
  \begin{equation*}
  \label{def_i_1}J_1(\tau) := \frac12 \int_{{\Om_{4\rho}^{\al}(\mfx)}} ( (\vh)^2 - (\vlh - \vh)^2 ) (y,\tau) \zv(\tau) \ dy.
  \end{equation*}
Note that $J_1(\mft-s) = J_1(\mft+s) =0$ since $\zv(\mft-s) =\zv(\mft+s) =  0$.

Applying the bound from Lemma \ref{lemma6.8-2},  we have 
\begin{equation*}
    \label{6.39}
    \begin{array}{ll}
     |J_2| & \apprle \iint_{{K_{4\rho}^{\al}(\mfz)}\setminus \elam}   \left| \dds{\vlh}  (\vlh-\vh)\right| \ dy \ d\tau \apprle\la |\RR^{n+1} \setminus \elam| .
    \end{array}
   \end{equation*}

\item[Estimate for $L_2$:] We split $L_2$ and make use of the fact that $\vlh(z) = \vh(z)$ for all $z\in \elam\cap {K_{4\rho}^{\al}(\mfz)}.$ 
\begin{equation*}
\begin{array}{ll}
L_2 & = \iint_{{K_{4\rho}^{\al}(\mfz)}\cap \elam} \iprod{[A(x,t,\nabla u) - A(x,t,\nabla w)]_h}{\nabla \vlh} \zv\ dz  \\
& \qquad + \iint_{{K_{4\rho}^{\al}(\mfz)}\setminus \elam} \iprod{[A(x,t,\nabla u) - A(x,t,\nabla w)]_h}{\nabla \vlh} \zv\ dz\\
& = \iint_{{K_{4\rho}^{\al}(\mfz)}\cap \elam} \iprod{[A(x,t,\nabla u) - A(x,t,\nabla w)]_h}{\nabla [u-w]_h} \zv\ dz  \\
& \qquad + \iint_{{K_{4\rho}^{\al}(\mfz)}\setminus \elam} \iprod{[A(x,t,\nabla u) - A(x,t,\nabla w)]_h}{\nabla \vlh} \zv\ dz\\
& =: L_2^1 + L_2^2.
\end{array}
\end{equation*}

\begin{description}
\item[Estimate for $L_2^1$:] Using \eqref{monotonicity}, we get
\begin{equation*}
\begin{array}{ll}
L_2^1 & = \iint_{{K_{4\rho}^{\al}(\mfz)}\cap \elam} \iprod{[A(x,t,\nabla u) - A(x,t,\nabla w)]_h}{\nabla [u-w]_h} \zv\ dz \\
& \apprge \iint_{{K_{4\rho}^{\al}(\mfz)} \cap \elam} |\nabla [u-w]_h|^2 \lbr \mu^2 + |\nabla [u]_h|^2 + |\nabla [w]_h|^2 \rbr^{\frac{\pp-2}{2}} \zv \ dz.
\end{array}
\end{equation*}

\item[Estimate for $L_2^2$:] Using the bound  from Lemma \ref{lemma6.7-1}, \eqref{abounded}, we get
\begin{equation}
\label{7.10}
\begin{array}{ll}
L_2^2 & \apprle \iint_{{K_{4\rho}^{\al}(\mfz)}\setminus \elam} \left|[A(x,t,\nabla u) - A(x,t,\nabla w)]_h\right| |\nabla \vlh| \ dz\\
& \apprle \sum_{i\in\NN} \la^{\frac{1}{p(z_i)}} \iint_{2Q_i} \lbr[[]\lbr \mu^2 + |\nabla u|^2 + |\nabla w|^2 \rbr^{\frac{\pp-1}{2}}\rbr[]]_h \ dz \\
& \apprle \sum_{i \in \NN} \la^{\frac{1}{p(z_i)}} \la^{\frac{p^+_{2Q_i}-1}{p^-_{2Q_i}}}|\htq_i|\\
& \apprle \la |\RR^{n+1} \setminus \elam|.
\end{array}
\end{equation}
In the last inequality, we made use of $\la^{\frac{1}{p(z_i)} + \frac{p^+_{2Q_i}}{p^-_{2Q_i}}-\frac{1}{p^-_{2Q_i}}-1} \leq C(\plog,n)$.

\end{description}

\item[Estimate for $L_3$:] Analogously to estimate $L_2$, we split $L_3$ as follows:
\begin{equation*}
\begin{array}{ll}
L_3 & = \iint_{{K_{4\rho}^{\al}(\mfz)}\cap \elam} \iprod{[|\bff|^{\pp-2} \bff]_h}{\nabla \vlh} \zv\ dz + \iint_{{K_{4\rho}^{\al}(\mfz)}\setminus \elam} \iprod{[|\bff|^{\pp-2} \bff]_h}{\nabla \vlh} \zv \ dz\\
& = \iint_{{K_{4\rho}^{\al}(\mfz)}\cap \elam} \iprod{[|\bff|^{\pp-2} \bff]_h}{\nabla [u-w]_h} \zv\ dz + \iint_{{K_{4\rho}^{\al}(\mfz)}\setminus \elam} \iprod{[|\bff|^{\pp-2} \bff]_h}{\nabla \vlh} \zv\ dz\\
& =: L_3^1 + L_3^2.
\end{array}
\end{equation*}

\begin{description}
\item[Estimate for $L_3^1$:] Using the fact that $\vlh(z) = \vh(z)$ for all $z\in \elam\cap {K_{4\rho}^{\al}(\mfz)}$, we get
\begin{equation*}
\begin{array}{ll}
L_3^1 & = \iint_{{K_{4\rho}^{\al}(\mfz)}\cap \elam} \iprod{[|\bff|^{\pp-2} \bff]_h}{\nabla [u-w]_h} \zv \ dz \\
& \leq \iint_{{K_{4\rho}^{\al}(\mfz)} \cap \elam} [|\bff|^{\pp-1}]_h |\nabla [u-w]_h| \ dz.
\end{array}
\end{equation*}

\item[Estimate for $L_3^2$:] Similar to the bound in \eqref{7.10}, we get
\begin{equation*}
\begin{array}{ll}
L_3^2 & \apprle \la |\RR^{n+1} \setminus \elam|. 
\end{array}
\end{equation*}

\end{description}

\end{description}

Combining all the above estimates, we get
\begin{equation*}
\label{combined_1}
\begin{array}{l}
- \int_{\mft-s}^{\mft+s} \int_{{\Om_{4\rho}^{\al}(\mfx)}} \dds{\zv} \lbr \vh^2 - (\vlh - \vh)^2 \rbr \ dy \ d\tau +  \iint_{{K_{4\rho}^{\al}(\mfz)} \cap \elam} |\nabla [u-w]_h|^2 \lbr\mu^2 +  |\nabla [u]_h|^2 +|\nabla [w]_h|^2 \rbr^{\frac{\pp-2}{2}}\zv \ dz \\
\hspace*{6cm} \apprle \iint_{{K_{4\rho}^{\al}(\mfz)}\cap \elam} [|\bff|^{p-1}]_h{|\nabla [u-w]_h|} \ dz +  \la |\RR^{n+1} \setminus \elam| .
\end{array}
\end{equation*}

In order to estimate $- \int_{\mft-s}^{\mft+s} \int_{{\Om_{4\rho}^{\al}(\mfx)}} \dds{\zv} \lbr \vh^2 - (\vlh - \vh)^2 \rbr \ dy \ d\tau$, we observe that on $\elam$, there holds $\vl = v$.
Taking limits first in $h\searrow 0$ followed by $\ve \searrow 0$, we get
\begin{equation*}
\begin{array}{ll}
- \int_{\mft-s}^{\mft+s} \int_{{\Om_{4\rho}^{\al}(\mfx)}} \dds{\zv} \lbr \vh^2 - (\vlh - \vh)^2 \rbr \ dy \ d\tau  \xrightarrow{\lim_{\ve \searrow 0}\lim_{h \searrow 0}} & \int_{{\Om_{4\rho}^{\al}(\mfx)}} (v^2 - (\vl - v)^2 )(x,\mft+s) \ dx \\
& - \int_{{\Om_{4\rho}^{\al}(\mfx)}} (v^2 - (\vl - v)^2 )(x,\mft-s) \ dx.
\end{array}
\end{equation*}

For the second term, we observe that on $\elam$, we have $\vl = v$; and on $\elam^c$, we have $\vl(\cdot,\mft-s) = v(\cdot,\mft-s) = 0$. Thus, the second term vanishes because on $\elam$, we can use the initial boundary condition; and on $\elam^c$, it is zero by construction. Thus we get
\begin{equation*}
- \int_{\mft-s}^{\mft+s} \int_{{\Om_{4\rho}^{\al}(\mfx)}} \dds{\zv} \lbr \vh^2 - (\vlh - \vh)^2 \rbr \ dy \ d\tau  \xrightarrow{\lim_{\ve \searrow 0} \lim_{h \searrow 0}} \int_{{\Om_{4\rho}^{\al}(\mfx)}} (v^2 - (\vl - v)^2 )(x,\mft+s) \ dx.
\end{equation*}

In fact, if we consider a cut-off function $\zv^{t_0} (\tau)$ for some $t_0 \in (\mft-s,\mft+s)$, where 
\begin{equation*}
\zv^{t_0}(\tau) = \left\{ \begin{array}{ll}
                1 & \text{for} \ \tau \in (-t_0+\ve,t_0-\ve),\\
                0 & \text{for} \ \tau \in (-\infty,-t_0)\cup (t_0,\infty),
                \end{array}\right.
\end{equation*}
we would have obtained the  following estimate after taking limits:
\begin{equation*}
\label{combined_1_new_11}
\begin{array}{l}
\int_{{\Om_{4\rho}^{\al}(\mfx)}} (v^2 - (\vl - v)^2 )(x,t_0) \ dx +  \int_{-t_0}^{t_0} \int_{{\Om_{4\rho}^{\al}(\mfx)} \cap \elam} |\nabla (u-w)|^2 \lbr\mu^2 +  |\nabla u|^2 +|\nabla w|^2 \rbr^{\frac{\pp-2}{2}} \ dx \ dt \\
\hfill \apprle  \iint_{{K_{4\rho}^{\al}(\mfz)}\cap \elam} |\bff|^{\pp-1}|\nabla (u-w)| \ dz +  \la  |\RR^{n+1} \setminus \elam| .
\end{array}
\end{equation*}

In particular, we get for any $t_0 \in (\mft-s,\mft+s)$
\begin{equation}
\label{combined_1_new_1}
\begin{array}{l}
\int_{{\Om_{4\rho}^{\al}(\mfx)}} (v^2 - (\vl - v)^2 )(x,t_0) \ dx +   \iint_{{K_{4\rho}^{\al}(\mfz)} \cap \elam} |\nabla (u-w)|^2 \lbr\mu^2+ |\nabla u|^2 +|\nabla w|^2 \rbr^{\frac{\pp-2}{2}} \ dx \ dt \\
 \hfill \apprle \iint_{{K_{4\rho}^{\al}(\mfz)}\cap \elam} |\bff|^{\pp-1}|\nabla (u-w)| \ dz +  \la |\RR^{n+1} \setminus \elam| .
\end{array}
\end{equation}

Using  Lemma \ref{cruc_3}, for any $t \in (\mft-s,\mft+s)$, there holds
\begin{equation*}
 \label{4.13}
  \begin{array}{ll}
    \int_{{\Om_{4\rho}^{\al}(\mfx)}} | (v)^2 - (\vl - v)^2 | (y,t) \ dy   
   & \apprge   - \la |\RR^{n+1} \setminus \elam|.
  \end{array}
 \end{equation*}
 Furthermore, using the above estimate in \eqref{combined_1_new_1} gives
%
%
%
\begin{equation}
\label{fully_combined}
\begin{array}{ll}
\iint_{{K_{4\rho}^{\al}(\mfz)} \cap \elam} |\nabla (u-w)|^2 \lbr \mu^2 + |\nabla u|^2 +|\nabla w|^2 \rbr^{\frac{\pp-2}{2}} \ dx \ dt &\apprle \iint_{{K_{4\rho}^{\al}(\mfz)}\cap \elam} |\bff|^{\pp-1}|\nabla (u-w)| \ dz 
+\la |\RR^{n+1} \setminus \elam|.
\end{array}
\end{equation} 
 
Let us now multiply \eqref{fully_combined} with $\la^{-1-\be}$ and integrate over $(1,\infty)$ to  get
 \begin{equation*}
 \label{K_expression}
K_1 +K_2\apprle K_3,
 \end{equation*}
 where we have set
 \begin{equation*}
  \begin{array}{@{}r@{}c@{}l@{}}
  K_1 \ &:=& \ \int_{1}^{\infty} \la^{-1-\be} \iint_{{K_{4\rho}^{\al}(\mfz)} \cap \elam} |\nabla (u-w)|^2 \lbr\mu^2+ |\nabla u|^2 +|\nabla u|^2 \rbr^{\frac{\pp-2}{2}}\ dz \ d\la,\\
  K_2 \ &:=& \ \int_{1}^{\infty} \la^{-1-\be} \iint_{{K_{4\rho}^{\al}(\mfz)} \cap \elam} |\bff|^{\pp-1}|\nabla (u-w)| \ dz \ d\la,  \\
  K_3\  &:=& \ \int_{1}^{\infty} \la^{-1-\be}  \la |\RR^{n+1} \setminus \elam| \  d\la. \\
  \end{array}
 \end{equation*}

Let us define $\tg = \max\{ g, 1\}$ where $g$ is from \eqref{def_g_A}, then we estimate each of the above terms as follows:
 \begin{description}[leftmargin=*]
 \item[Estimate for $K_1$:] Applying Fubini's theorem, we get
 \begin{equation*}
 \label{3.13}
 K_1 \apprge \frac{1}{\be} \iint_{{K_{4\rho}^{\al}(\mfz)}} \tg(z)^{-\be}|\nabla (u-w)|^2 \lbr \mu^2 + |\nabla u|^2 +|\nabla u|^2 \rbr^{\frac{\pp-2}{2}}\ dz.
 \end{equation*}


 Let us define 
 \[
 K^+(\mfz) := \{ z \in K_{4\rho}^{\al}(\mfz): p(z) \geq 2\} \txt{and}K^-(\mfz) := \{ z \in K_{4\rho}^{\al}(\mfz): p(z) \leq 2\},
 \]
and consider the following two subcases:
 \begin{description}[leftmargin=*]
 \item[Subcase $K^-(\mfz)$:] We have the following simple decomposition:
 \begin{equation}
\label{p-great-twoo}
\begin{array}{rcl}
|\nabla u - \nabla w|^{p(z)(1-\be)}  &=&  \left[ (\mu^2+ |\nabla u|^2 + |\nabla w|^2)^{\frac{p(z)-2}{2}} |\nabla u - \nabla w |^2 g^{-\be} \right]^{\frac{p(z)(1-\be)}{2}}  \times \\ && \times \left( \mu^2 + |\nabla u|^2 + |\nabla w|^2 \right)^{\frac{p(z)(1-\be)(2-p(z))}{4}} \times \tg^{\frac{p(z)(1-\be)}{2}\be}.
\end{array}
\end{equation}

Integrating \eqref{p-great-twoo} over $K^-(\mfz)$ and making use of Young's inequality  with exponents $\frac{2}{p(z)(1-\be)}, \frac{2}{2-p(z)} $ and $\frac{2}{p(z)\be}$, we get
\begin{equation}
\label{p-less-twoo}
\begin{array}{ll}
\iint_{K^-(\mfz)}  |\nabla u - \nabla w|^{\pp(1-\be)} \, dx & \apprle \ep_1 \iint_{K^-(\mfz)} (\mu^2 +|\nabla u|^2 + |\nabla w|^2)^{\frac{\pp(1-\be)}{2}} \, dz  + \ep_2 \iint_{K^-(\mfz)} \tg(z)^{1-\be} \, dz   \\
 &\qquad + C_{(\ep_1,\ep_2)} \iint_{K^-(\mfz)} ( \mu^2 + |\nabla u|^2 + |\nabla w|^2)^{\frac{\pp-2}{2}} |\nabla u - \nabla w |^2 \tg(z)^{-\be} \, dz.
\end{array}
\end{equation}

From the strong maximal function bound of Lemma \ref{max_bound}, we see that 
\begin{equation}
\label{7.27}
\begin{array}{rcl}
\iint_{K^-(\mfz)} \tg(z)^{1-\be} \, dz & \apprle & \iint_{\RR^{n+1}} \tg^{1-\be} \ dz  + |K^-(\mfz)|\\
& \apprle & \iint_{K_{4\rho}^{\al}(\mfz)} \lbr \frac{|u-w|}{\rho} + |\nabla u| + |\nabla w| + |\bff| + \mu+ 1\rbr^{{\pp(1-\be)}} \ dz + |K^-(\mfz)|\\
& \apprle & \iint_{K_{4\rho}^{\al}(\mfz)} \lbr |\nabla u| + |\nabla w| + |\bff| +\mu+  1\rbr^{\pp(1-\be)} \ dz.
\end{array}
\end{equation}

Combining \eqref{p-less-twoo} and \eqref{7.27}, we get
\begin{equation*}
\label{p-less-twoo-young}
\begin{array}{ll}
\iint_{K^-(\mfz)}  |\nabla u - \nabla w|^{\pp(1-\be)} \, dz &  \apprle (\ep_1 + \ep_2) C_{(\plog,n,\lamot)}  \iint_{{K_{4\rho}^{\al}(\mfz)}} |\nabla u|^{\pp(1-\be)} + |\nabla w-\nabla u|^{\pp(1-\be)}\, dz \\
 &\qquad + C_{(\ep_1,\ep_2)} \iint_{K^-(\mfz)} ( \mu^2 + |\nabla u|^2 + |\nabla w|^2)^{\frac{\pp-2}{2}} |\nabla u - \nabla w |^2 \tg(z)^{-\be} \, dz \\
 & \qquad +  \iint_{{K_{4\rho}^{\al}(\mfz)}}\lbr[[]|\bff|^{\pp(-\be)} + 1\rbr[]]\, dz.\\
\end{array}
\end{equation*}

 \item[Subcase $K^+(\mfz)$:] In this case, we proceed as follows:
 \begin{equation*}
 \label{p-geq-two}
 \begin{array}{rcl}
 \iint_{K^+(\mfz)} |\nabla u - \nabla w|^{\pp(1-\be)} \ dz & \leq &C_{(\ve_3)}\iint_{{K_{4\rho}^{\al}(\mfz)}} \tg(z)^{-\be} |\nabla u - \nabla w|^{\pp} \ dz + \ep_3 \iint_{{K_{4\rho}^{\al}(\mfz)}} \tg(z)^{1-\be} \ dz \\
 &  \overset{\eqref{7.27}}{\apprle}& 
  C_{(\ep_3)} \iint_{K^+(\mfz)} \tg(z)^{-\be} ( \mu^2 + |\nabla u|^2 + |\nabla w|^2)^{\frac{\pp-2}{2}} |\nabla u - \nabla w |^2  \, dz\\
 && \qquad + \ep_3 \iint_{{K_{4\rho}^{\al}(\mfz)}}  |\nabla u-\nabla w|^{\pp(1-\be)} + |\nabla u|^{\pp(1-\be)} + |\bff|^{\pp(1-\be)}+ 1\, dz.
 \end{array}
 \end{equation*}

 \end{description}

 \item[Estimate for $K_2$:] Again by Fubini's theorem, we get
 \begin{equation*}
 \label{3.18}
 K_2  = \frac{1}{\be} \iint_{{K_{4\rho}^{\al}(\mfz)}} \tg(z)^{-\be} \iprod{|\bff|^{\pp-2} \bff}{\nabla u - \nabla w} \ dz.
 \end{equation*}
From the definition of $g(z)$, we see that for $z \in {K_{4\rho}^{\al}(\mfz)}$, we have $\tg(z) \geq |\nabla u - \nabla w|(z)$ which implies $\tg(z)^{-\be} \leq |\nabla u - \nabla w|^{-\be}(z)$. We can now apply Young's inequality with exponents $\frac{\pp(1-\be)}{\pp-1}$ and $\frac{\pp(1-\be)}{1-\pp\be}$  to get:
\begin{equation*}
\label{3.19}
\begin{array}{ll}
K_2 & 
 \apprle \frac{C_{(\ep_4)}}{\be} \iint_{{K_{4\rho}^{\al}(\mfz)}} |\bff|^{\pp(1-\be)} \ dz + \frac{\ep_4}{\be} \iint_{{K_{4\rho}^{\al}(\mfz)}} |\nabla u - \nabla w|^{\pp(1-\be)} \ dz.
\end{array}
\end{equation*}

 \item[Estimate for $K_3$:] Applying the layer cake representation followed by Lemma \ref{max_bound}, we get
 \begin{equation*}
 \label{3.22}
 \begin{array}{rcl}
 K_3 & = &\frac{1}{1-\be} \iint_{\RR^{n+1}} \tg(z)^{1-\be} \ dz \\
 & \overset{\eqref{7.27}}{\apprle} & \iint_{{K_{4\rho}^{\al}(\mfz)}} \lbr |\nabla u - \nabla w| + |\nabla u| + |\bff|+\mu  + 1\rbr^{\pp(1-\be)} \ dz.
 \end{array}
 \end{equation*}

 \end{description}

 Combining everything, we get the following estimate:
 \begin{equation*}
 \begin{array}{rcl}
 \iint_{{K_{4\rho}^{\al}(\mfz)}} |\nabla u - \nabla w|^{p-\be} \ dz 
 & \apprle & (\ve_1+\ve_2+\ve_3+C_{(\ve_1,\ve_2,\ve_3)}\be) \iint_{K_{4\rho}^{\al}(\mfz)} |\nabla u|^{\pp(1-\be)} \ dz  \\
 & & + (\ve_1+\ve_2+\ve_3+\ve_4+ C_{(\ve_1,\ve_2,\ve_3)}\be) \iint_{K_{4\rho}^{\al}(\mfz)} |\nabla u - \nabla w|^{\pp(1-\be)} \ dz  \\
 && + C_{(\ve_1,\ve_2,\ve_3,\ve_4,\be)} \iint_{K_{4\rho}^{\al}(\mfz)} \lbr[[]|\bff|^{\pp(1-\be)} + 1\rbr[]] \ dz.
 \end{array}
 \end{equation*}
 Choosing $\ve_1,\ve_2,\ve_3$ and $\ve_4$ small followed by $\be \in (0,\tilde{\be}_3]$, we get the proof of the estimate. 
\end{proof}


\section{Second difference estimate below the natural exponent}
\label{six}
In this section, we will prove a difference estimate between the weak solution of \eqref{wapprox_int} and the weak solution of \eqref{vapprox_bnd}. To do this, we will use the method of Lipschitz truncation from Appendix \ref{lipschitz_truncation_B}. 

For this section, let us denote 
\begin{equation*}
\label{def_s_b}
s:= \scalet{\al}{\mfz} (3\rho)^2.
\end{equation*}

\begin{theorem}
\label{second_diff_thm}
Let $(p(\cdot),\aa,\Om)$ be $(\gamma,\bs_0)$-vanishing. Suppose that $w$ and $v$ are weak solutions of \eqref{wapprox_int} and \eqref{vapprox_bnd}, respectively, and let $\al \geq 1$ be given such that the following assumptions hold:
\begin{equation}
\label{hypothesis}
\fiint_{K_{4\rho}^{\al}(\mfz)} |\nabla w|^{p(\mfz)(1-\be)} \ dz \le c_{\ast} \al^{1-\be} \txt{and} \al^{p^+_{K_{4\rho}^{\al}(\mfz)}-p^-_{K_{4\rho}^{\al}(\mfz)}}\leq c_p.
\end{equation}
Further assume that  
\begin{equation}\label{more_hyp} \al^{-\frac{n}{p(\mfz)} + \frac{nd}{2} + d} \leq \Ga^2 (4\rho)^{-(n+2)} \txt{and} p^+_{K_{4\rho}^{\al}(\mfz)}-p^-_{K_{4\rho}^{\al}(\mfz)} \leq \modp(4\rho\Ga),\end{equation} 
for some $\Ga, c_p, c_{\ast} >1$ to be selected as fixed constants in Section \ref{eight}.

Then there exists $\tilde{\rho}_4 = \tilde{\rho}_4(n,\plog,\lamot,\bm_0)$ such that for any $128 \rho \leq \tilde{\rho}_4$  and for any $\ve \in (0,1]$, there exist $\tilde{\be}_4 = \tilde{\be}_4(\ve, n,\lamot,\plog)$ and $\tilde{\ga}_0 = \tilde{\ga}_0(\ve, n,\lamot,\plog)$ such that for any $\be \in (0,\tilde{\be}_4]$ and any $\ga \in (0,\tilde{\ga}_0)$, the following estimate holds:

\begin{equation}
\label{diff_est_two}
\fiint_{K_{3\rho}^{\al}(\mfz)} |\nabla w - \nabla v|^{p(\mfz)(1-\be)} \ dz \leq \ve \al \txt{and} \fiint_{K_{3\rho}^{\al}(\mfz)} |\nabla v|^{p(\mfz)(1-\be)} \ dz \apprle \al.
\end{equation}

\end{theorem}

\begin{proof}
The first estimate in \eqref{diff_est_two} and \eqref{hypothesis} directly implies the second estimate in \eqref{diff_est_two} after making use of  the triangle inequality. Thus we only prove the first estimate in \eqref{diff_est_two}.

Consider the following cut-off function \textcolor{black}{\bf $\zv \in C^{\infty} (\RR)$} such that $0 \leq \zv(t) \leq 1$ and 
\begin{equation*}
\label{def_zv_B}
\zv(t) = \left\{ \begin{array}{ll}
                1 & \text{for} \ t \in (\mft-s+\ve,\mft+s-\ve),\\
                0 & \text{for} \ t \in (-\infty,\mft-s)\cup (\mft+s,\infty).
                \end{array}\right.
\end{equation*}

It is easy to see that 
\begin{equation*}
\label{bound_zv_B}
\begin{array}{c}
\zv'(t) = 0 \ \txt{for} \  t \in (-\infty,\mft-s) \cup (\mft-s+\ve,\mft+s-\ve)\cup (\mft+s,\infty), \\
|\zv'(t)| \leq \frac{c}{\ve}\  \txt{for} \  t \in (\mft-s,\mft-s+\ve) \cup (\mft+s-\ve,\mft+s). 
\end{array}
\end{equation*}
Without loss of generality, we shall always take $2h \leq \ve$ since we will take limits in the following order $\lim_{\ve \rightarrow 0} \lim_{h \rightarrow 0}$.

We shall use $\vlh(z) \zv(t)$ as a test function where $\vlh$ is as constructed in Appendix \ref{lipschitz_truncation_B} (more specifically in \eqref{lipschitz_function_B}). This is valid since $\vlh \in C^{0,1}(K_{3\rho}^{\al}(\mfz))$. Using this, we get
\begin{equation*}
\begin{array}{l}
\iint_{{K_{3\rho}^{\al}(\mfz)}} \ddt{[w-v]_h} \vlh \zv \ dx \ dt + \iint_{{K_{3\rho}^{\al}(\mfz)}} \iprod{[\bbb(t,\nabla v) - \bbb(t,\nabla w)]_h}{\nabla \vlh} \zv \ dx \ dt \\
\hspace*{6cm} = \iint_{{K_{3\rho}^{\al}(\mfz)}} \iprod{[\bbb(t,\nabla w) - \bb(x,t,\nabla w)]_h}{\nabla \vlh} \zv  \ dx \ dt\\
\hspace*{6.5cm} + \iint_{{K_{3\rho}^{\al}(\mfz)}} \iprod{[\bb(x,t,\nabla w) - \aa(x,t,\nabla w)]_h}{\nabla \vlh} \zv  \ dx \ dt.
\end{array}
\end{equation*}

Proceeding as in Theorem \ref{first_diff_thm}, after taking limits, we get for any $t_0 \in (\mft-s,\mft+s)$, the estimate
\begin{equation}
\label{B_7.8}
\begin{array}{l}
\int_{\Om_{3\rho}^{\al}(\mfx)} (v^2 - (\vl - v)^2)(x,t_0) \ dx +\iint_{{K_{3\rho}^{\al}(\mfz)}\cap \elam} \iprod{\bbb(t,\nabla v) - \bbb(t,\nabla w)}{\nabla (v-w)} \zv \ dx \ dt  \\
\hspace*{3cm}= -\iint_{{K_{3\rho}^{\al}(\mfz)}\setminus \elam} \iprod{[\bbb(t,\nabla v) - \bbb(t,\nabla w)]_h}{\nabla \vlh} \zv \ dx \ dt  \\
 \hspace*{4cm}+ \iint_{{K_{3\rho}^{\al}(\mfz)}\setminus \elam} \iprod{[\bbb(t,\nabla w) - \bb(x,t,\nabla w)]_h}{\nabla \vlh} \zv  \ dx \ dt\\
 \hspace*{4cm}+\iint_{{K_{3\rho}^{\al}(\mfz)}\setminus \elam} \iprod{[\bb(x,t,\nabla w) - \aa(x,t,\nabla w)]_h}{\nabla \vlh} \zv  \ dx \ dt  \\
 \hspace*{4cm}+ \iint_{{K_{3\rho}^{\al}(\mfz)}\cap \elam} \iprod{[\bbb(t,\nabla w) - \bb(x,t,\nabla w)]_h}{\nabla (v-w)} \zv  \ dx \ dt\\
 \hspace*{4cm}+\iint_{{K_{3\rho}^{\al}(\mfz)}\cap \elam} \iprod{[\bb(x,t,\nabla w) - \aa(x,t,\nabla w)]_h}{\nabla (v-w)} \zv  \ dx \ dt.
\end{array}
\end{equation}

Let us multiply \eqref{B_7.8} by $\la^{-1-\be}$ and integrate over $[1,\infty)$ to get
\[
K_1 + K_2 \leq K_3 + K_4 + K_5 + K_6 + K_7,
\]

where 
\begin{equation*}\label{define_K}\begin{array}{rcl}
K_1 &:=& \int_1^{\infty} \la^{-1-\be} \int_{\Om_{\rho}(\mft)} (v^2 - (\vl - v)^2)(x,t_0) \ dx \ d\la,\\
K_2 &:=& \int_1^{\infty} \la^{-1-\be} \iint_{{K_{3\rho}^{\al}(\mfz)}\cap \elam} \iprod{\bbb(t,\nabla v) - \bbb(t,\nabla w)}{\nabla (v-w)} \zv \ dx \ d\la, \\
K_3 &:=& -\int_1^{\infty} \la^{-1-\be} \iint_{{K_{3\rho}^{\al}(\mfz)}\setminus \elam} \iprod{[\bbb(t,\nabla v) - \bbb(t,\nabla w)]_h}{\nabla \vlh} \zv \ dx \ dt \ d\la,\\
K_4 &:=& \int_1^{\infty} \la^{-1-\be} \iint_{{K_{3\rho}^{\al}(\mfz)}\setminus \elam} \iprod{[\bbb(t,\nabla w) - \bb(x,t,\nabla w)]_h}{\nabla \vlh} \zv  \ dx \ dt \ d\la, \\
K_5 &:=& \int_1^{\infty} \la^{-1-\be} \iint_{{K_{3\rho}^{\al}(\mfz)}\setminus \elam} \iprod{[\bb(x,t,\nabla w) - \aa(x,t,\nabla w)]_h}{\nabla \vlh} \zv  \ dx \ dt \ d\la, \\
K_6 &:=& \int_1^{\infty} \la^{-1-\be} \iint_{{K_{3\rho}^{\al}(\mfz)}\cap \elam} \iprod{[\bbb(t,\nabla w) - \bb(x,t,\nabla w)]_h}{\nabla (v-w)} \zv  \ dx \ dt \ d\la, \\
K_7 &:=& \int_1^{\infty} \la^{-1-\be} \iint_{{K_{3\rho}^{\al}(\mfz)}\cap \elam} \iprod{[\bb(x,t,\nabla w) - \aa(x,t,\nabla w)]_h}{\nabla (v-w)} \zv  \ dx \ dt \ d\la.
\end{array}\end{equation*}

Let us set  $\tg(z) := \max\{1,g(z)\}$ where $g(z)$ is defined in \eqref{def_g_B} and  estimate each of the terms as follows:
\begin{description}
\item[Estimate for $K_1$:] Using Lemma \ref{cruc_3_B}, we see that 
\[
\int_{\Om_{\rho}(t_0)} (v^2 - (\vl - v)^2)(x,t_0) \ dx  \geq - \la |\RR^{n+1} \setminus \elam|.
\]
Using this along with Fubini's theorem, we see that 
\begin{equation*}
\label{estimate_K1}
\begin{array}{rcl}
K_1 & \apprge & - \int_1^{\infty} \la^{-1-\be} \la |\{ z \in \RR^{n+1}: \tg(z) \geq \la| \ d\la \\
& = &-\frac{1}{1-\be}\iint_{\RR^{n+1}} \tg(z)^{1-\be} \ dz \\
& \apprge & - \iint_{K_{3\rho}^{\al}(\mfz)}\lbr |\nabla w - \nabla v| + |\nabla w| + 1\rbr^{p(\mfz)(1-\be)} \ dz. 
\end{array}
\end{equation*}

\item[Estimate for $K_2$:] Similar to the estimates in Theorem \ref{first_diff_thm}, we see that
\begin{equation*}
\label{estimate_K2}
\iint_{K_{3\rho}^{\al}(\mfz)} |\nabla w - \nabla v|^{p(\mfz)(1-\be)} \ dz \apprle C_{(\ve_1)} \be K_2 + \ve_1 \iint_{K_{3\rho}^{\al}(\mfz)} |\nabla w|^{p(\mfz)(1-\be)} + 1\ dz.
\end{equation*}

\item[Estimate for $K_3$:] Using the bound from Lemma \ref{lemma6.7-1_B}, we get
\begin{equation*}
\begin{array}{rcl}
\iint_{{K_{3\rho}^{\al}(\mfz)}\setminus \elam} \iprod{[\bbb(t,\nabla v) - \bbb(t,\nabla w)]_h}{\nabla \vlh} \zv \ dx \ dt & \apprle & \sum_{i \in \NN} \la^{\frac{1}{p(\mfz)}} \iint_{2Q_i} \lbr \mu^2 + |\nabla w|^2 + |\nabla v|^2 \rbr^{\frac{p(\mfz)-1}{2}} \ dz \\
& \apprle & \la^{\frac{1}{p(\mfz)}} \la^{\frac{p(\mfz) -1}{p(\mfz)}} |16Q_i| 
 \apprle  \la |\RR^{n+1} \setminus \elam|.
\end{array}
\end{equation*}
Using the above bound in $K_3$ followed by applying Fubini's theorem, we get
\begin{equation}
\label{estimate_K3}
\begin{array}{rcl}
K_3 & \apprle &  \int_1^{\infty} \la^{-1-\be} \la |\{ z \in \RR^{n+1}: \tg(z) \geq \la| \ d\la =\frac{1}{1-\be}\iint_{\RR^{n+1}} \tg(z)^{1-\be} \ dz \\
& \apprle &  \iint_{K_{3\rho}^{\al}(\mfz)}\lbr |\nabla w - \nabla v| + |\nabla w| + 1\rbr^{p(\mfz)(1-\be)} \ dz.
\end{array}
\end{equation}

\item[Estimate for $K_4$:] Similar to the estimate for $K_3$, we get
\begin{equation*}
\label{estimate_K4}
K_4 \apprle \iint_{K_{3\rho}^{\al}(\mfz)}\lbr |\nabla w - \nabla v| + |\nabla w| + 1\rbr^{p(\mfz)(1-\be)} \ dz. 
\end{equation*}

\item[Estimate for $K_5$:] In this case, we proceed as follows:
\begin{equation*}
\begin{array}{rcl}
&&\hspace*{-4cm}\iint_{{K_{3\rho}^{\al}(\mfz)}\setminus \elam} \iprod{[\bb(x,t,\nabla w) - \aa(x,t,\nabla w)]_h}{\nabla \vlh} \zv  \ dx \ dt \\ 
&  \apprle & \sum_{i\in\NN} \la^{\frac{1}{p(\mfz)}} \iint_{2Q_i} \abs{\bb(x,t,\nabla w) - \aa(x,t,\nabla w)}  \ dx \ dt \\
&  \overset{\eqref{aa_bb}}{\apprle} & \sum_{i\in\NN}|2Q_i|\la^{\frac{1}{p(\mfz)}} \fiint_{2Q_i} \lbr |\nabla w| + |\nabla v| + 1\rbr^{p(\mfz) -1} \ dx \ dt \\
& \overset{\eqref{elambda_B}}{\apprle} &   \sum_{i\in\NN}|2Q_i|\la^{\frac{1}{p(\mfz)}} \la^{\frac{p(\mfz) -1}{p(\mfz)}} \apprle \la |\RR^{n+1} \setminus \elam|.
\end{array}
\end{equation*}

This is bounded exactly as in \eqref{estimate_K3} to get 
\begin{equation*}
\label{estimate_K5}
K_5 \apprle \iint_{K_{3\rho}^{\al}(\mfz)}\lbr |\nabla w - \nabla v| + |\nabla w| + 1\rbr^{p(\mfz)(1-\be)} \ dz. 
\end{equation*}

\item[Estimate for $K_6$:] Applying Fubini's theorem, we see that 
\begin{equation*}
\begin{array}{rcl}
K_6: &=& \frac{1}{\be} \iint_{K_{3\rho}^{\al}(\mfz)} |\bbb(t,\nabla w) - \bb[x,t,\nabla w)| |\nabla (v-w)|\tg^{-\be}(z) \ dz  \\
& \apprle & \frac{1}{\be} \iint_{K_{3\rho}^{\al}(\mfz)} \Th(\aa, B_{4\rho}^{\al}(\mfx))(1+|\nabla w|)^{p(\mfz)-1} |\nabla v-\nabla w|\tg^{-\be}(z) \ dz  \\
& \apprle & \frac{1}{\be} \iint_{K_{3\rho}^{\al}(\mfz)} \Th(\aa, B_{4\rho}^{\al}(\mfx))(1+|\nabla w|)^{p(\mfz)-1} |\nabla v-\nabla w|^{1-\be} \ dz  \\
& \apprle &\frac{\ve_2}{\be} \iint_{K_{3\rho}^{\al}(\mfz)} |\nabla w - \nabla v|^{p(\mfz)(1-\be)} \ dz + \frac{C_{(\ve_2)}}{\be} \underbrace{\iint_{K_{3\rho}^{\al}(\mfz)} \Th(\aa, B_{4\rho}^{\al}(\mfx))^{\frac{p(\mfz)}{p(\mfz)-1}}(1+|\nabla w|)^{p(\mfz)} \ dz}_{\tilde{K}_6}.\\
\end{array}
\end{equation*}
We shall estimate the second term as follows: ($\sigma$ is to be chosen appropriately later on)
\begin{equation*}
\begin{array}{rcl}
 \frac{\tilde{K}_6}{|K_{3\rho}^{\al}(\mfz)|} & \apprle &  \lbr \fiint_{K_{3\rho}^{\al}(\mfz)} \Th(\aa, B_{4\rho}^{\al}(\mfx))^{\frac{p(\mfz)}{p(\mfz)-1}\frac{4+\sigma}{\sigma}}\ dz \rbr^{\frac{\sigma}{4+\sigma}} \lbr \fiint_{K_{3\rho}^{\al}(\mfz)}(1+|\nabla w|)^{p(\mfz)\frac{4+\sigma}{4}} \ dz\rbr^{\frac{4}{4+\sigma}}.
\end{array}
\end{equation*}
If we restrict $p^+_{K_{4\rho}(\mfz)} - p^-_{K_{4\rho}(\mfz)} \leq \frac{(p^--1)\sigma}{4}$, we see that the following two bounds hold:
\begin{equation}
\label{bounds_pp_mfz}
\begin{array}{c}
p(\mfz) \leq \frac{p(\mfz)}{p(\mfz) - 1} (p^+_{K_{4\rho}(\mfz)} -1) \leq \frac{p^-_{K_{4\rho}(\mfz)}(p^+_{K_{4\rho}(\mfz)}-1)}{p^-_{K_{4\rho}(\mfz)}-1} \leq \pp \lbr 1 + \frac{p^+_{K_{4\rho}(\mfz)} -p^-_{K_{4\rho}(\mfz)}}{p^--1} \rbr \leq \pp \lbr 1 + \frac{\sigma}{4} \rbr,\\
p(\mfz)\lbr 1 + \frac{\sigma}{4} \rbr\leq \frac{p(\mfz)}{p(\mfz) - 1} (p^+_{K_{4\rho}(\mfz)} -1) \lbr 1 + \frac{\sigma}{4} \rbr \leq  \pp \lbr 1 + \frac{p^+_{K_{4\rho}(\mfz)} -p^-_{K_{4\rho}(\mfz)}}{p^--1} \rbr \lbr 1 + \frac{\sigma}{4} \rbr  \leq \pp \lbr 1 + {\sigma} \rbr.
\end{array}
\end{equation}

Let us set $a = \frac{p^+_{K_{4\rho}(\mfz)} -p^-_{K_{4\rho}(\mfz)}}{p^--1} \lbr 1 + \frac{\sigma}{4} \rbr  + \frac{\sigma}{4} \leq \sigma$, then we get from Corollary \ref{normalized_higher_integrability} that 
\begin{equation}
\label{8.25}
\begin{array}{rcl}
\fiint_{K_{3\rho}^{\al}(\mfz)}(1+|\nabla w|)^{p(\mfz)\frac{4+\sigma}{4}} \ dz & \apprle& \fiint_{K_{3\rho}^{\al}(\mfz)}(1+|\nabla w|)^{\pp ( 1+a)} \ dz\\
& \overset{\text{Corollary \ref{normalized_higher_integrability}}}{\apprle} & \al^{1+a}\\
& = & \al^{1+\frac{\sigma}{4}}\al^{\frac{p^+_{K_{4\rho}(\mfz)} -p^-_{K_{4\rho}(\mfz)}}{p^--1} \lbr 1 + \frac{\sigma}{4} \rbr}\\
& \overset{\eqref{hypothesis}}{\apprle} & c_p^{\frac{1}{p^--1} \lbr 1 + \frac{\sigma}{4} \rbr} \al^{1+\frac{\sigma}{4}}.
\end{array}
\end{equation}

From \eqref{abounded} and \eqref{small_aa}, we see that 
\begin{equation*}
\lbr \fiint_{K_{3\rho}^{\al}(\mfz)} \Th(\aa, B_{4\rho}^{\al}(\mfx))^{\frac{p(\mfz)}{p(\mfz)-1}\frac{4+\sigma}{\sigma}}\ dz \rbr^{\frac{\sigma}{4+\sigma}}  \apprle \ga^{\frac{\sigma}{4+\sigma}}.
\end{equation*}

Combining everything, we see that 
\begin{equation*}
K_6 \apprle \frac{\ve_2}{\be} \iint_{K_{3\rho}^{\al}(\mfz)} |\nabla w - \nabla v|^{p(\mfz)(1-\be)} \ dz + \frac{C_{(\ve_2)}}{\be} \ga^{\frac{\sigma}{4+\sigma}} |K_{3\rho}^{\al}(\mfz)|	c_p^{\frac{1}{p^--1}}\al.
\end{equation*}

\item[Estimate for $K_7$:] Applying Fubini's theorem, we get 
\begin{equation*}
\begin{array}{rcl}
K_7 &=& \frac{1}{\be}\iint_{K_{3\rho}^{\al}(\mfz)} |\bb(x,t,\nabla w) - \aa(x,t,\nabla w)| |\nabla (v-w)|\tg^{-\be}(z) \ dz  \\
&\apprle & \frac{1}{\be}\iint_{K_{3\rho}^{\al}(\mfz)} |\bb(x,t,\nabla w) - \aa(x,t,\nabla w)| |\nabla (v-w)|^{1-\be} \ dz \\
&\overset{\eqref{def_bb},\eqref{abounded}}{\apprle} & \frac{1}{\be}\iint_{K_{3\rho}^{\al}(\mfz)} (\mu^2 + |\nabla w|^2)^{\frac{p(z)-1}{2}} \left|1 - (\mu^2 + |\nabla w|^2)^{\frac{p(\mfz)-p(z)}{2}} \right| |\nabla (v-w)|^{1-\be} \ dz. \\
\end{array}
\end{equation*}

Applying Young's inequality, we get for $E:=\{z \in K_{3\rho}^{\al}(\mfz): \mu^2 + |\nabla w(z)|^2>0\}$, the bound
\begin{equation}
\label{w_v_11}
\begin{array}{rcl}
\frac{K_7}{|K_{3\rho}^{\al}(\mfz)|} & \apprle & \frac{1}{\be} \ve_3 \fiint_{K_{3\rho}^{\al}(\mfz)} |\nabla w - \nabla v|^{p(\mfz)(1-\be)} \ dz  \\
& & \qquad + \frac{C_{(\ve_3)}}{\be |K_{3\rho}^{\al}(\mfz)|} \underbrace{\iint_E  \bgh{(\mu^2 + |\nabla w|^2)^{\frac{p(z)-1}{2}} \left|1 - (\mu^2 + |\nabla w|^2)^{\frac{p(\mfz)-p(z)}{2}} \right|}^{\frac{p(\mfz)}{p(\mfz)-1}} \ dz}_{J}.
\end{array}
\end{equation}

We shall now proceed with estimating the second term in \eqref{w_v_11} as follows: 
For each $z \in E$, in view of the mean value theorem applied to $(\mu^2 + |\nabla w|^2)^{\frac{p(\mfz)-p(z)}{2}\mathfrak{a}}$, there exists $\mathfrak{a}_z \in [0,1]$ such that we get
\begin{equation}
 \label{w_v_13}
 (\mu^2 + |\nabla w|^2)^{\frac{p(\mfz)-p(z)}{2}} -1 = \frac{p(\mfz)-p(z)}{2} (\mu^2 + |\nabla w|^2)^{\frac{p(\mfz)-p(z)}{2}\mathfrak{a}_z} \log (\mu^2 + |\nabla w|^2). 
\end{equation}
This implies
\begin{equation*}
 \label{w_v_13.1}
 \begin{array}{l}
 (\mu^2 + |\nabla w|^2)^{\frac{p(z)-1}{2}} \left|1 - (\mu^2 + |\nabla w|^2)^{\frac{p(\mfz)-p(z)}{2}} \right| 
 \apprle \modp(4\rho\Ga) (\mu^2 + |\nabla w|^2)^{\frac{(p(\mfz)-p(z))\mathfrak{a}_z + p(z) -1}{2}} \log (\mu^2 + |\nabla w|^2). 
 \end{array}
\end{equation*}

Let us now define the sets 
\begin{equation}
\label{split_sets}
E^1 : = \{ z \in K_{3\rho}^{\al}(\mfz) : |\nabla w(x)| \leq 1\} \quad \text{and} \quad  E^2 : = \{ x \in K_{3\rho}^{\al}(\mfz) : |\nabla w(x)| > 1\}. 
\end{equation}
Recall that $\mu \leq 1$ and hence using the inequality $t^{\be} |\log t| \leq \max \left\{ \frac{1}{e^{\be}}, 2^{\be}\log 2\right\}$ which holds for all $t \in (0,2]$ and any $\be >0$, we get for $z \in E^1$
\begin{equation}
 \label{w_v_14}
 (\mu^2 + |\nabla w|^2)^{\frac{p(z)-1}{2}} \left|1 - (\mu^2 + |\nabla w|^2)^{\frac{p(\mfz)-p(z)}{2}} \right| \apprle \modp(4\rho\Ga) \max \left\{ \frac{1}{e^{\frac{p^--1}{2}}}, 2^{\frac{p^+-1}{2}}\log 2 \right\}.
\end{equation}
To obtain the above estimate, with  $\be(z) := \frac{\mathfrak{a}_z (p(\mfz)-p(z)) + p(z) -1}{2}$, there holds 
\begin{equation*}
 \label{bound_beta}
 \frac{p^--1}{2} \leq \be(z) \leq \frac{p(\mfz)-1}{2} \leq \frac{p^+-1}{2}.
\end{equation*}

Hence using \eqref{split_sets} and combining \eqref{w_v_14} into \eqref{w_v_13}, we get
\begin{equation}
 \label{w_v_15}
 \begin{array}{ll}
  |\bb(z,\nabla w) - \aa(z,\nabla w)|& \apprle \lsb{\chi}{E^1}\modp(4\rho\Ga) \max \left\{ \frac{1}{e^{\frac{p^--1}{2}}}, 2^{\frac{p^+-1}{2}}\log 2 \right\}  \\
  & \qquad \qquad + \lsb{\chi}{E^2} {\modp(4\rho\Ga)} |\nabla w|^{(p(\mfz)-p(z))\mathfrak{a}_z + p(z) -1} \log(e + |\nabla w|). 
 \end{array}
\end{equation}

Combining \eqref{w_v_15} and \eqref{w_v_11}, we get
\begin{equation*}
 \label{w_v_16}
   J \apprle \modp(4\rho\Ga)^{\frac{p(\mfz)}{p(\mfz)-1}}|K_{3\rho}^{\al}(\mfz)| \lbr 1+ J_1\rbr. 
\end{equation*}
where $J_1:= \fiint_{K_{3\rho}^{\al}(\mfz)} |\nabla w|^{\mathfrak{b}} [\log(e+|\nabla w|^{\mathfrak{b}})]^{\frac{p(\mfz)}{p(\mfz)-1}} \ dz$ with $\mathfrak{b}:={\frac{p(\mfz)}{p(\mfz)-1}(p^+_{K_{3\rho}^{\al}(\mfz)}-1)}$.  Using the inequality $\log(e + ab) \leq \log(e+a) + \log(e+b)$ for  $a,b>0$ along with the simple bound $\frac{p(\mfz)}{p(\mfz)-1} \leq \frac{p^-}{p^--1}$,  we get
\begin{equation*}
 \label{w_v_17}
 \begin{array}{ll}
  J_1 & \apprle \fiint_{K_{3\rho}^{\al}(\mfz)} |\nabla w|^{\mathfrak{b}} \lbr[[]\log \lbr e + \frac{|\nabla w|^{\mathfrak{b}}}{\avg{|\nabla w|^{\mathfrak{b}}}{K_{3\rho}^{\al}(\mfz)}} \rbr \rbr[]]^{\frac{p^-}{p^--1}}\ dz \\
  & \qquad + \fiint_{K_{3\rho}^{\al}(\mfz)} |\nabla w|^{\mathfrak{b}} \lbr[[]\log \lbr e + {\avg{|\nabla w|^{\mathfrak{b}}}{K_{3\rho}^{\al}(\mfz)}} \rbr \rbr[]]^{\frac{p(\mfz)}{p(\mfz)-1}}\ dz \\
  & =: J_2 + J_3.
 \end{array}
\end{equation*}
\begin{description}
 \item[Estimate for $J_2$:] We now apply Lemma \ref{llogl} with $f = |\nabla w|^{\mathfrak{b}}$,  $\be = \frac{p^-}{p^--1}$ and $s=1+\frac{\sigma}{4}$ to get
 \begin{equation}
  \label{w_v_18}
  \begin{array}{rcl}
  J_2 &\apprle&  \lbr \fiint_{K_{3\rho}^{\al}(\mfz)} |\nabla w|^{\mathfrak{b} \lbr 1 + \frac{\sigma}{4} \rbr} \ dz \rbr^{\frac{4}{4+\sigma}} \leq \lbr \fiint_{K_{4\rho}^{\al}(\mfz)} (1+|\nabla w|)^{\pp(1+a)} \ dz \rbr^{\frac{4}{4+\sigma}} \\
  & \overset{\eqref{8.25}}{\apprle} & c_p^{\frac{1}{p^--1}}\al,
  \end{array}
 \end{equation}
where $a = \frac{p^+_{K_{4\rho}(\mfz)} -p^-_{K_{4\rho}(\mfz)}}{p^--1} \lbr 1 + \frac{\sigma}{4} \rbr  + \frac{\sigma}{4}$ satisfying $a \leq \sigma$.
 \item[Estimate for $J_3$:] From \eqref{bounds_pp_mfz} and \eqref{8.25}, we see that 
 \begin{equation}
  \label{w_v_19}
  \begin{array}{rcl}
   \log \lbr e + \avg{|\nabla w|^{\mathfrak{b}}}{K_{3\rho}^{\al}(\mfz)}\rbr & \leq &\log \lbr e + c_p^{\frac{1}{p^--1}}\al\rbr  = \log(e + c_1 \al)
   \apprle C_{(c_1)}( \log \al + 1) \\
  & \apprle &C_{(c_1)} \lbr \log \al^{-\frac{n}{p(\mfz)} + \frac{nd}{2} + d} + 1\rbr\\
    &  \overset{\eqref{more_hyp}}{ \leq} &C_{(c_1)} \left\{ \log \lbr \Ga^2 (4\rho)^{-(n+2)} \rbr +1 \right\}.
  \end{array}
 \end{equation}
 Here we have denoted $c_1 = c_p^{\frac{1}{p^--1}}$ where $c_p$ is from \eqref{8.25}. 
 Substituting \eqref{w_v_19} into $J_3$ and making use of the bound from \eqref{w_v_18}, we get
 
 \begin{equation*}
  \label{w_v_21}
   J_3 \apprle  C_{(c_1)} \left\{ \log \lbr \Ga^2 (4\rho)^{-(n+2)} \rbr +1 \right\}^{\frac{p(\mfz)}{p(\mfz)-1}}\al.
 \end{equation*}

\end{description}
Then we have
\begin{equation*}
 \label{estimate_K7_pre}
 \begin{array}{ll}
   |K_7| & \apprle \frac{\ep_3}{\be} \iint_{K_{3\rho}^{\al}(\mfz)} |\nabla w - \nabla v|^{p(\mfz)(1-\be)} \ dz \\
   & \qquad \qquad + \frac{C_{(\ep_3)}}{\be}\modp(4\rho\Ga)^{\frac{p^-}{p^--1}} |K_{3\rho}^{\al}(\mfz)| C_{(c_1)} \left\{ \log \lbr \Ga^2 (4\rho)^{-(n+2)} \rbr +1 \right\}^{\frac{p(\mfz)}{p(\mfz)-1}}\al.
   \end{array}
\end{equation*}
The restriction $\rho \leq \frac{1}{4e \Ga^{n+5}}$ implies
\begin{equation*}\label{two_7.34}
\begin{array}{rcl}
\modp(4\rho\Ga) \left\{ \log \lbr \Ga^2 (4\rho)^{-(n+2)} + 1 \rbr \right\}& =&  \modp(4\rho\Ga) \log\lbr 4\rho e \Ga^2 (4\rho)^{-(n+3)}\rbr\\
& \leq & \modp(4\rho\Ga) \log\lbr \Ga^{-(n+3)} (4\rho)^{-(n+3)}\rbr\\
& \leq & (n+3) \modp(4\rho\Ga) \log\lbr \frac{1}{4\rho\Ga}\rbr\\
& \apprle & \ga.
\end{array}
\end{equation*}

Using this, we get
\begin{equation*}
 \label{estimate_K7}
   |K_7| \apprle \frac{\ep_3}{\be} \int_{K_{3\rho}^{\al}(\mfz)} |\nabla w - \nabla v|^{p(\mfz)(1-\be)} \ dz + \frac{C_{(\ep_3)}}{\be}\ga^{\frac{p^-}{p^--1}} |K_{3\rho}^{\al}(\mfz)| C_{(c_1)}\al.
\end{equation*}	

\end{description}

Combining all the estimates, we get, 
\begin{equation*}
\begin{array}{rcl}
\iint_{K_{3\rho}^{\al}(\mfz)} |\nabla w - \nabla v|^{p(\mfz)(1-\be)} \ dz & \apprle & (\ve_1 + \ve_2 + \ve_4 + \be) \iint_{K_{3\rho}^{\al}(\mfz)} |\nabla w - \nabla v|^{p(\mfz)(1-\be)} \ dz  \\
& & + 3\be \iint_{K_{3\rho}^{\al}(\mfz)} |\nabla w|^{p(\mfz)(1-\be)} \ dz  \\
&&  + {C_{(\ep_3)}}\ga^{\frac{p^-}{p^--1}} |K_{3\rho}^{\al}(\mfz)| C_{(c_1)}\al.
\end{array}
\end{equation*}

Now choosing $\ve_1,\ve_2,\ve_4$ and $\be$ small, we get for any $\ve >0$, the estimate
\begin{equation*}
\begin{array}{rcl}
\fiint_{K_{3\rho}^{\al}(\mfz)} |\nabla w - \nabla v|^{p(\mfz)(1-\be)} \ dz & \le &  (\ve_1 + \ve_2 + \ve_4 + \be) \fiint_{K_{3\rho}^{\al}(\mfz)} |\nabla w|^{p(\mfz)(1-\be)} \ dz  + \ga^{\frac{p^-}{p^--1}}  C_{(c_1)}\al\\
   & \overset{\eqref{hypothesis}}{\le} & \ve \al^{1-\be} + \ga^{\frac{p^-}{p^--1}}  \lbr 1+c_p^{\frac{1}{p^--1}}\rbr\al.
\end{array}
\end{equation*}
Note that $\al \geq 1$ which implies $\al^{1-\be} \leq \al$. Now choose $\ga$ sufficiently small such that for any $\ve \in (0,1)$, there holds
\[
\fiint_{K_{3\rho}^{\al}(\mfz)} |\nabla w - \nabla v|^{p(\mfz)(1-\be)} \ dz \leq \ve \al,
\]
which completes the proof.
\end{proof}


\section{Covering arguments}
\label{eight}
Let $\beta\ \in (0,\be_0)$ and let $\bs_0 >0$ be given, where $\be_0$ is from Section \ref{def_be_0}.
Assume that  $(\pp,\aa,\Om)$ is $(\ga,\bs_0)$-vanishing in the sense of Definition \ref{further_assumptions}. Let $q(\cdot)$ be log-H\"older continuous in the sense of Definition \ref{definition_p_log}. We fix any $\rho \le \frac{\rho_0}{4}$, where $\rho_0$ is given in Remark \ref{remark_radius}, and fix any $\mfz = (\mfx,\mft) \in \OO_T$ with $\gh{\mft-(4\rho)^2, \mft+(4\rho)^2} \subset (-T,T)$.


We observe that from Theorem \ref{high_weak} and Theorem \ref{high_very_weak} that 
\begin{equation}\label{hi1}
\mint{K_\rho(\mfz)}{|\nabla u|^{p(z)(1-\beta)(1+\sigma)}}{dz} \lesssim \gh{\mint{K_{2\rho}(\mfz)}{\gh{|\nabla u|+|\bff|}^{p(z)(1-\beta)}}{dz}}^{1+\sigma\tilde{\theta}} + \mint{K_{2\rho}(\mfz)}{|\bff|^{p(z)(1-\beta)(1+\sigma)}}{dz} + 1,
\end{equation}
where $\be$ and $\sigma$ are given in Remark \ref{high_int_remark} and for some $\tilde{\theta} = \tilde{\theta}(n,p(\mfz)) > 0$. 


It follows from Section \ref{sub_radii} that for each $z \in K_{4\rho}(\mfz)$,
\begin{gather}
\label{ca1}
\frac{p(z)(1-\beta)q(z)}{q_{K_{4\rho}(\mfz)}^-} \le p(z)(1-\beta)\gh{1+\frac{\modq(8\rho)}{q^-}} \le p(z)(1-\beta)(1+\sigma),\\
\label{ca2}
\frac{p(z)(1-\beta)q(z)(1+\sigma)}{q_{K_{4\rho}(\mfz)}^-} \le p(z)(1-\beta)(1+3\sigma) \le \min\mgh{p(z), p(z)(1-\beta)q^-}.
\end{gather}

We first verify some parabolic localization properties under our unified intrinsic cylinders.
\begin{lemma}\label{cv lemma}
Let $c_a > 1$ and let $\bm_0$ be given in \eqref{def_M_0}. Then there is a constant $c_1 = c_1(n,\La_0, \La_1, \plog, \qlog) \ge 1$ such that for any $\la \ge 1$, any $\tilde{\mfz} \in K_{2\rho}(\mfz)$, and any $\tilde{\rho}>0$,
\begin{equation}\label{cv0-1}
\tilde{\rho} \le \Ga^{-2}\bs_0, \quad \text{where}\ \ \Ga := 2c_1 c_a \bm_0 \gamma^{-1} \ge 2,
\end{equation}
satisfying $K_{\tilde{\rho}}^\alpha(\tilde{\mfz}) \subset K_{2\rho}(\mfz)$, if
\begin{equation}\label{cv0-2}
\al \le c_a \mgh{ \mint{K_{\tilde{\rho}}^\alpha(\tilde{\mfz})}{|\nabla u|^{\frac{p(z)(1-\beta)q(z)}{q_{K_{4\rho}(\mfz)}^-}}}{dz} + \frac{1}{\gamma} \gh{\mint{K_{\tilde{\rho}}^\alpha(\tilde{\mfz})}{|\bff|^{\frac{p(z)(1-\beta)q(z)(1+\sigma)}{q_{K_{4\rho}(\mfz)}^-}}}{dz}}^{\frac{1}{1+\sigma}} },
\end{equation}
then we have
\begin{gather}
\label{cv0-3} 
\al^{-\frac{n}{p(\tilde{\mfz})} + \frac{nd}{2} + d} \le \Ga^2 \tilde{\rho}^{-(n+2)}, \qquad p_{Q_{\tilde{\rho}}^\alpha(\tilde{\mfz})}^+ - p_{Q_{\tilde{\rho}}^\alpha(\tilde{\mfz})}^- \le \modp(\Ga \tilde{\rho}), \qquad \al^{p_{Q_{\tilde{\rho}}^\alpha(\tilde{\mfz})}^+ - p_{Q_{\tilde{\rho}}^\alpha(\tilde{\mfz})}^-} \le e^\frac{2n+5}{-\frac{n}{p^-} + \frac{nd}{2} + d} =: c_p,\\
\label{cv0-4}
q_{Q_{\tilde{\rho}}^\alpha(\tilde{\mfz})}^+ - q_{Q_{\tilde{\rho}}^\alpha(\tilde{\mfz})}^- \le \modq(\Ga \tilde{\rho}), \qquad \al^{q_{Q_{\tilde{\rho}}^\alpha(\tilde{\mfz})}^+ - q_{Q_{\tilde{\rho}}^\alpha(\tilde{\mfz})}^-} \le e^\frac{(2n+5) L}{-\frac{n}{p^-} + \frac{nd}{2} + d} =: c_q.
\end{gather}
\end{lemma}

\begin{proof}
Fix $K_{\tilde{\rho}}^\alpha(\tilde{\mfz}) \subset K_{2\rho}(\mfz)$. 
We compute
\begin{equation}\label{cv0-5}
\begin{array}{rcl}
& &\hspace*{-3cm} \mint{K_{2\rho}(\mfz)}{|\nabla u|^{\frac{p(z)(1-\beta)q(z)}{q_{K_{4\rho}(\mfz)}^-}}}{dz} + \frac{1}{\gamma} \gh{\mint{K_{2\rho}(\mfz)}{|\bff|^{\frac{p(z)(1-\beta)q(z)(1+\sigma)}{q_{K_{4\rho}(\mfz)}^-}}}{dz}}^{\frac{1}{1+\sigma}}\\
& \overset{\eqref{ca1},\eqref{ca2}}{\le} & \frac{1}{\gamma} \mgh{ \mint{K_{2\rho}(\mfz)}{|\nabla u|^{p(z)(1-\beta)\gh{1+\frac{\modq(8\rho)}{q^-}}}}{dz} + \mint{K_{2\rho}(\mfz)}{|\bff|^{p(z)}}{dz} + 1 }\\
&\overset{\eqref{hi1}}{\apprle} &\frac{1}{\gamma} \mgh{ \gh{\mint{K_{4\rho}(\mfz)}{\gh{|\nabla u|+|\bff|}^{p(z)(1-\beta)}}{dz}}^{1+\frac{\modq(8\rho)}{q^-}\tilde{\theta}} + \mint{K_{4\rho}(\mfz)}{|\bff|^{p(z)}}{dz} + 1}\\
& \overset{\eqref{def_M_0}}{\apprle} & \frac{\bm_0}{\gamma|K_{4\rho}(\mfz)|} \mgh{\gh{\frac{\bm_0}{|K_{4\rho}(\mfz)|}}^{\frac{\modq(8\rho)}{q^-}\tilde{\theta}} + 1} \overset{\text{Section \ref{sub_radii}}}{\apprle} \frac{\bm_0}{\gamma|K_{4\rho}(\mfz)|}.
\end{array}\end{equation}

Then we see
\begin{equation*}
\begin{aligned}
\al^{-\frac{n}{p(\tilde{\mfz})} + \frac{nd}{2} + d} &\overset{\eqref{cv0-2}}{\le} \frac{c_a \al^{-\frac{n}{p(\tilde{\mfz})} + \frac{nd}{2} -1 + d}}{|K_{\tilde{\rho}}^\alpha(\tilde{\mfz})|} \mgh{ \integral{K_{2\rho}(\mfz)}{|\nabla u|^{\frac{p(z)(1-\beta)q(z)}{q_{K_{4\rho}(\mfz)}^-}}}{dz} + \frac{1}{\gamma} \gh{\integral{K_{2\rho}(\mfz)}{|\bff|^{\frac{p(z)(1-\beta)q(z)(1+\sigma)}{q_{K_{4\rho}(\mfz)}^-}}}{dz}}^{\frac{1}{1+\sigma}} }\\
&\overset{\eqref{cv0-5}}{\le} \frac{c_1 c_a \bm_0}{\gamma \tilde{\rho}^{n+2}} \overset{\eqref{cv0-1}}{\le} \Ga \tilde{\rho}^{-(n+2)},
\end{aligned}
\end{equation*}
for some $c_1 = c_1(n,\La_0, \La_1, \plog, \qlog) \ge 1$.

On the other hand, it follows from Remark \ref{remark_def_p_log} that
$$
p_{Q_{\tilde{\rho}}^\alpha(\tilde{\mfz})}^+ - p_{Q_{\tilde{\rho}}^\alpha(\tilde{\mfz})}^- \le \modp \gh{\max\mgh{\al^{-\frac{1}{p(\tilde{\mfz})} + \frac{d}{2}}, \al^{\frac{-1+d}{2}} } 2\tilde{\rho}} \le \modp(2 \tilde{\rho}) \le \modp(\Ga \tilde{\rho}),
$$
which implies
$$
\Ga^{p_{Q_{\tilde{\rho}}^\alpha(\tilde{\mfz})}^+ - p_{Q_{\tilde{\rho}}^\alpha(\tilde{\mfz})}^-} \le \Ga^{\modp(\Ga \tilde{\rho})} \overset{\eqref{cv0-1}}{\le} \gh{\frac{\Ga}{\bs_0}}^{\modp\gh{\frac{\bs_0}{\Ga}}} \overset{\eqref{small_px}}{\le} e^\gamma \le e.
$$
Then we discover
$$
\al^{p_{Q_{\tilde{\rho}}^\alpha(\tilde{\mfz})}^+ - p_{Q_{\tilde{\rho}}^\alpha(\tilde{\mfz})}^-} \le \gh{\Ga \tilde{\rho}^{-(n+2)}}^{\frac{\modp(\Ga \tilde{\rho})}{-\frac{n}{p(\tilde{\mfz})} + \frac{nd}{2} + d}} \le \Ga^{\frac{(n+3)\modp(\Ga \tilde{\rho})}{-\frac{n}{p(\tilde{\mfz})} + \frac{nd}{2} + d}}  \gh{\frac{1}{\Ga \tilde{\rho}}}^{\frac{(n+2)\modp(\Ga \tilde{\rho})}{-\frac{n}{p(\tilde{\mfz})} + \frac{nd}{2} + d}} \le \gh{e^{2n+5}}^{\frac{1}{-\frac{n}{p^-} + \frac{nd}{2} + d}}.
$$
Similarly, we can also obtain the inequalities \eqref{cv0-4}.
\end{proof}

We now consider a Vitali type covering lemma for intrinsic parabolic cylinders as follow:
\begin{lemma}\label{cv argument}
Let $\al, c_p, c_q >1$ and let $\F := \mgh{Q_{\rho_j}^{\al_j}(\mfz_j)}_{j \in \J} \subset Q_{2r}(\mfz)$ be any collection of intrinsic parabolic cylinders, where $\al_j := \al^{\frac{q_{Q_{4r}(\mfz)}^-}{q(\mfz_j)}}$ and $\rho_j >0$, satisfying
\begin{equation}\label{cv0-6}
\al_j^{p_{Q_{\rho_j}^{\alpha_j}(\mfz_j)}^+ - p_{Q_{\rho_j}^{\alpha_j}(\mfz_j)}^-} \le c_p \quad \text{and} \quad \al_j^{q_{Q_{\rho_j}^{\alpha_j}(\mfz_j)}^+ - q_{Q_{\rho_j}^{\alpha_j}(\mfz_j)}^-} \le c_q \quad \text{for every} \ \ j \in \J.
\end{equation}
Then there exists a countable subcollection $\G = \mgh{Q_{\rho_i}^{\al_i}(\mfz_i)}_{i \in \I}, \I \subset \J$, of mutually disjoint cylinders such that
\begin{equation*}
\bigcup_{j \in \J} Q_{\rho_j}^{\al_j}(\mfz_j) \subset \bigcup_{i \in \I} Q_{\chi \rho_i}^{\al_i}(\mfz_i),
\end{equation*}
for some constant $\chi = \chi_{(n,c_p,c_q,\plog,\qlog)} \ge 1$.
\end{lemma}

\begin{proof}
The proof is similar to that of the standard Vitali covering lemma except in the setting of the unified intrinsic cylinders. See \cite[Lemma 5.3]{BO} and \cite[Lemma 7.1]{bogelein2014very} for other intrinsic cylinder cases. For completeness, we give the proof. 

Write $D := \sup_{j\in\J} \rho_j$. Set
$$
\F_k := \mgh{Q_{\rho_j}^{\al_j}(\mfz_j) \in \F : \frac{D}{2^k} < \rho_j \le \frac{D}{2^{k-1}}} \quad (k=1,2,\cdots).
$$
We define $\G_k \subset \F_k$ as follows:
\begin{itemize}
\item Let $\G_1$ be any maximal disjoint collection of intrinsic cylinders in $\F_1$.
\item Assuming that $\G_1, \cdots, G_{k-1}$ have been selected, we choose $\G_k$ to be any maximal disjoint subcollection of 
$$
\mgh{Q \in \F_k : Q \cap Q' = \emptyset \ \ \text{for all} \ \ Q' \in \bigcup_{l=1}^{k-1} \G_l}.
$$
\item Finally, we define
$$
\G :=\bigcup_{k=1}^\infty \G_k.
$$
\end{itemize}
Clearly $\G$ is a countable collection of disjoint intrinsic cylinders and $\G \subset \F$. Now it suffices to show that for each intrinsic cylinder $Q_{\rho_j}^{\al_j}(\mfz_j) \in \F$, there exists an intrinsic cylinder $Q_{\rho_i}^{\al_i}(\mfz_i) \in \G$ such that $Q_{\rho_j}^{\al_j}(\mfz_j) \cap Q_{\rho_i}^{\al_i}(\mfz_i) \neq \emptyset$ and $Q_{\rho_j}^{\al_j}(\mfz_j) \subset Q_{\chi \rho_i}^{\al_i}(\mfz_i)$.

Fix $Q_{\rho_j}^{\al_j}(\mfz_j) \in \F$. Then there is an index $k$ such that $Q_{\rho_j}^{\al_j}(\mfz_j) \in \F_k$. By the maximality of $\G_k$, there exists an intrinsic cylinder $Q_{\rho_i}^{\al_i}(\mfz_i) \in \bigcup_{l=1}^{k} \G_l$ with $Q_{\rho_j}^{\al_j}(\mfz_j) \cap Q_{\rho_i}^{\al_i}(\mfz_i) \neq \emptyset$. Since $\rho_i > \frac{D}{2^k}$ and $\rho_j \le \frac{D}{2^{k-1}}$, we know $\rho_j < 2 \rho_i$.
Choose $\mfz_0 \in Q_{\rho_j}^{\al_j}(\mfz_j) \cap Q_{\rho_i}^{\al_i}(\mfz_i)$. We compute 
\begin{equation*}
\begin{array}{rcl}
\al_j^{-1+d} &=& \al_i^{-1+d} \al^{q_{Q_{4r}(\mfz)}^- \frac{(1-d)(q(\mfz_j)-q(\mfz_0))-(1-d)(q(\mfz_0)-q(\mfz_i))}{q(\mfz_i)q(\mfz_j)}}\\
&\le & \al_i^{-1+d} \al_j^{\frac{(1-d)\gh{q_{Q_{\rho_j}^{\alpha_j}(\mfz_j)}^+ - q_{Q_{\rho_j}^{\alpha_j}(\mfz_j)}^-}}{q(\mfz_i)}} \al_i^{\frac{(1-d)\gh{q_{Q_{\rho_i}^{\alpha_i}(\mfz_i)}^+ - q_{Q_{\rho_i}^{\alpha_i}(\mfz_i)}^-}}{q(\mfz_j)}}\\
&\overset{\eqref{cv0-6}}{\le} &c_q^{\frac{2(1-d)}{q^-}} \al_i^{-1+d}.
\end{array}
\end{equation*}
where $d$ is given in \eqref{def_d}. Similarly, it follows from \eqref{cv0-6} that 
$$
\al_j^{-\frac{1}{p(\mfz_j)}+\frac{d}{2}} \apprle_{(c_p,c_q,\plog,\qlog)} \al_i^{-\frac{1}{p(\mfz_i)}+\frac{d}{2}}.
$$
Thus, from the definition of intrinsic cylinders in Section \ref{Intrinsic cylinders}, there exists a constant $\chi = \chi_{(n,c_p,c_q,\plog,\qlog)} \ge 1$ such that $Q_{\rho_j}^{\al_j}(\mfz_j) \subset Q_{\chi \rho_i}^{\al_i}(\mfz_i)$, which completes the proof.
\end{proof}

\subsection{Stopping-time argument}
We employ in this subsection a {\em stopping-time argument} from \cite{adimurthi2018weight} to derive a covering of the upper-level set of $|\nabla u|^\frac{p(\cdot)(1-\beta)q(\cdot)}{q_{K_{4\rho}(\mfz)}^-}$ with respect to some intrinsic parameter $\al$.

Let us define $\tilde{\al}$ by
\begin{equation}\label{cv2-1}
\tilde{\al}^\frac{1}{\vt_{K_{4\rho}(\mfz)}^+} := \mint{K_{2\rho}(\mfz)}{|\nabla u|^{\frac{p(z)(1-\beta)q(z)}{q_{K_{4\rho}(\mfz)}^-}}}{dz} + \frac{1}{\gamma} \mgh{\gh{\mint{K_{2\rho}(\mfz)}{|\bff|^{\frac{p(z)(1-\beta)q(z)(1+\sigma)}{q_{K_{4\rho}(\mfz)}^-}}}{dz}}^{\frac{1}{1+\sigma}} + 1},
\end{equation}
where the constants $\be$ and $\sigma$ are given in Remark \ref{high_int_remark} and 
\begin{equation}\label{vt_max_def}
\vt_{K_{4\rho}(\mfz)}^+ := \sup_{z \in K_{4\rho}(\mfz)} \vt(z) \overset{\eqref{vt_def}}{=} \frac{1}{-\frac{n}{p_{K_{4\rho}(\mfz)}^-}+\frac{nd}{2}+d}.
\end{equation}
For $\al \ge 1$ and $s \ge 1$, let $E(s,\al)$ denote the upper-level set of $|\nabla u(\cdot)|^\frac{p(\cdot)(1-\beta)q(\cdot)}{q_{K_{4\rho}(\mfz)}^-}$, defined by
\begin{equation}\label{cv2-2}
E(s,\al) := \mgh{z \in K_{s\rho}(\mfz) : |\nabla u(z)|^\frac{p(z)(1-\beta)q(z)}{q_{K_{4\rho}(\mfz)}^-} > \al}.
\end{equation}
Fix any $1 \le s_1 < s_2 \le 2$ and any $\al \ge 1$ satisfying
\begin{equation}\label{cv2-3}
\al > A \tilde{\al}, \quad \text{where} \ \ A := \mgh{\gh{\frac{16}{7}}^n \gh{\frac{120\chi}{s_2-s_1}}^{n+2}}^{\vt_{K_{4\rho}(\mfz)}^+}.
\end{equation}
Here $\chi$ is given in Lemma \ref{cv argument}. Fix any 
\begin{equation}\label{cv2-4}
\tilde{\rho} \in \left( \frac{(s_2-s_1)\rho}{60\chi}, (s_2-s_1)\rho \right].
\end{equation}
We check that for all $\tilde{\mfz} \in K_{s_1r}(\mfz)$,
\begin{equation*}
\begin{array}{rcl}
&&\hspace*{-3cm}\mint{K_{\tilde{\rho}}^{\al_{\tilde{\mfz}}}(\tilde{\mfz})}{|\nabla u|^{\frac{p(z)(1-\beta)q(z)}{q_{K_{4\rho}(\mfz)}^-}}}{dz} + \frac{1}{\gamma} \gh{\mint{K_{\tilde{\rho}}^{\al_{\tilde{\mfz}}}(\tilde{\mfz})}{|\bff|^{\frac{p(z)(1-\beta)q(z)(1+\sigma)}{q_{K_{4\rho}(\mfz)}^-}}}{dz}}^{\frac{1}{1+\sigma}}\\
&\le & \frac{|Q_{2r}|}{|K_{\tilde{\rho}}^{\al_{\tilde{\mfz}}}(\tilde{\mfz})|} \mgh{ \mint{K_{2\rho}(\mfz)}{|\nabla u|^{\frac{p(z)(1-\beta)q(z)}{q_{K_{4\rho}(\mfz)}^-}}}{dz} + \frac{1}{\gamma} \gh{\mint{K_{2\rho}(\mfz)}{|\bff|^{\frac{p(z)(1-\beta)q(z)(1+\sigma)}{q_{K_{4\rho}(\mfz)}^-}}}{dz}}^{\frac{1}{1+\sigma}} }\\
&\overset{\eqref{cv2-1}}{\le}& \frac{|Q_{2r}|}{\al_{\tilde{\mfz}}^{-\frac{n}{p(\tilde{\mfz})} + \frac{nd}{2} -1 +d} |K_{\tilde{\rho}}(\tilde{\mfz})|} \tilde{\al}^{\frac{1}{\vt_{K_{4\rho}(\mfz)}^+}} \overset{\eqref{measure_one}}{\le} \gh{\frac{16}{7}}^n \gh{\frac{2r}{\tilde{\rho}}}^{n+2} \al_{\tilde{\mfz}}^{\frac{n}{p(\mfz)} - \frac{nd}{2} +1 -d} \tilde{\al}^{\frac{1}{\vt_{K_{4\rho}(\mfz)}^+}}\\
&\overset{\eqref{cv2-4}}{\le}& \gh{\frac{16}{7}}^n \gh{\frac{120\chi}{(s_2-s_1)\rho}}^{n+2} \al_{\tilde{\mfz}}^{\frac{n}{p(\tilde{\mfz})} - \frac{nd}{2} +1 -d} \tilde{\al}^{\frac{1}{\vt_{K_{4\rho}(\mfz)}^+}}\\
&\overset{\eqref{cv2-3}}{<} &\al_{\tilde{\mfz}}^{\frac{n}{p(\tilde{\mfz})} - \frac{nd}{2} +1 -d} \al^{\frac{1}{\vt_{K_{4\rho}(\mfz)}^+}} \le \al,
\end{array}
\end{equation*}
where $\al_{\tilde{\mfz}} := \al^{\frac{q_{K_{4\rho}(\mfz)}^-}{q(\tilde{\mfz})}}$. The last inequality has used the fact that $\frac{1}{\vt_{K_{4\rho}(\mfz)}^+} \le -\frac{n}{p(\tilde{\mfz})} + \frac{nd}{2} + d$ and $1 \le \al_{\tilde{\mfz}} \le \al$.

On the other hand, in view of the Lebesgue differentiation theorem, for every Lebesgue point $\tilde{\mfz}$ of $|\nabla u|^{\frac{p(\cdot)(1-\beta)q(\cdot)}{q_{K_{4\rho}(\mfz)}^-}}$ in $E(s_1,\al)$, we have
\begin{equation*}\label{cv2-12}
\lim_{\tilde{\rho}\to 0} \mgh{ \mint{K_{\tilde{\rho}}^{\al_{\tilde{\mfz}}}(\tilde{\mfz})}{|\nabla u|^{\frac{p(z)(1-\beta)q(z)}{q_{K_{4\rho}(\mfz)}^-}}}{dz} + \frac{1}{\gamma} \gh{\mint{K_{\tilde{\rho}}^{\al_{\tilde{\mfz}}}(\tilde{\mfz})}{|\bff|^{\frac{p(z)(1-\beta)q(z)(1+\sigma)}{q_{K_{4\rho}(\mfz)}^-}}}{dz}}^{\frac{1}{1+\sigma}} } > \al.
\end{equation*}
Then for almost every such point, there exists $\rho_{\tilde{\mfz}} \in \left(0, \frac{(s_2-s_1)\rho}{60\chi} \right]$ such that
\begin{gather*}\label{cv2-13}
\mint{K_{\rho_{\tilde{\mfz}}}^{\al_{\tilde{\mfz}}}(\tilde{\mfz})}{|\nabla u|^{\frac{p(z)(1-\beta)q(z)}{q_{K_{4\rho}(\mfz)}^-}}}{dz} + \frac{1}{\gamma} \gh{\mint{K_{\rho_{\tilde{\mfz}}}^{\al_{\tilde{\mfz}}}(\tilde{\mfz})}{|\bff|^{\frac{p(z)(1-\beta)q(z)(1+\sigma)}{q_{K_{4\rho}(\mfz)}^-}}}{dz}}^{\frac{1}{1+\sigma}} = \al,\\
\mint{K_{\tilde{\rho}}^{\al_{\tilde{\mfz}}}(\tilde{\mfz})}{|\nabla u|^{\frac{p(z)(1-\beta)q(z)}{q_{K_{4\rho}(\mfz)}^-}}}{dz} + \frac{1}{\gamma} \gh{\mint{K_{\tilde{\rho}}^{\al_{\tilde{\mfz}}}(\tilde{\mfz})}{|\bff|^{\frac{p(z)(1-\beta)q(z)(1+\sigma)}{q_{K_{4\rho}(\mfz)}^-}}}{dz}}^{\frac{1}{1+\sigma}} < \al \quad \forall \tilde{\rho} \in \left(\rho_{\tilde{\mfz}}, \frac{(s_2-s_1)\rho}{60\chi} \right].
\end{gather*}

Applying Lemma \ref{cv lemma} and Lemma \ref{cv argument} to the collection of intrinsic cylinders $\mgh{Q_{\rho_{\mfz}}^{\al_{\tilde{\mfz}}}(\tilde{\mfz})}$ with $\rho_{\tilde{\mfz}}$ replacing $\tilde{\rho}$ and $\al_{\tilde{\mfz}}$ replacing $\al$, there exist $\mgh{\mfz_i}_{i=1}^\infty \subset E(s_1,\al)$ and 
$\rho_i \in \left(0, \frac{(s_2-s_1)\rho}{60\chi} \right]$, where $\al_i := \al^{\frac{q_{K_{4\rho}(\mfz)}^-}{q(\mfz_i)}}$ for 
$i=1,2,\cdots,$ such that $\mgh{Q_{\rho_i}^{\al_i}(\mfz_i)}_{i=1}^\infty$ is mutually disjoint,
\begin{equation}\label{cv2-6}
E(s_1,\al) \setminus N \subset \bigcup_{i=1}^{\infty} K_{\chi \rho_i}^{\al_i}(\mfz_i) \subset K_{s_2 \rho}(\mfz),
\end{equation}
for some Lebesgue measure zero set $N$, and for each $i$ we have
\begin{equation}\label{cv2-7}
\mint{K_{\rho_i}^{\al_i}(\mfz_i)}{|\nabla u|^{\frac{p(z)(1-\beta)q(z)}{q_{K_{4\rho}(\mfz)}^-}}}{dz} + \frac{1}{\gamma} \gh{\mint{K_{\rho_i}^{\al_i}(\mfz_i)}{|\bff|^{\frac{p(z)(1-\beta)q(z)(1+\sigma)}{q_{K_{4\rho}(\mfz)}^-}}}{dz}}^{\frac{1}{1+\sigma}} = \al,
\end{equation}
and
\begin{equation}\label{cv2-8}
\mint{K_{\tilde{\rho}}^{\al_i}(\mfz_i)}{|\nabla u|^{\frac{p(z)(1-\beta)q(z)}{q_{K_{4\rho}(\mfz)}^-}}}{dz} + \frac{1}{\gamma} \gh{\mint{K_{\tilde{\rho}}^{\al_i}(\mfz_i)}{|\bff|^{\frac{p(z)(1-\beta)q(z)(1+\sigma)}{q_{K_{4\rho}(\mfz)}^-}}}{dz}}^{\frac{1}{1+\sigma}} < \al,
\end{equation}
for any $\tilde{\rho} \in \left(\rho_i, (s_2-s_1)\rho \right]$.
Note that since $\min\mgh{1,\al^{\frac{1}{p(\mfz)}-\frac{d}{2}},\al^{\frac{1-d}{2}}} = 1$, we have $\bigcup_{i=1}^{\infty} K_{\chi \rho_i}^{\al_i}(\mfz_i) \subset K_{s_2 \rho}(\mfz)$.

\subsection{Power decay estimates on unified intrinsic cylinders}
Here we derive the power decay estimate \eqref{cv4-9} on the upper-level set of $|\nabla u|^\frac{p(\cdot)(1-\beta)q(\cdot)}{q_{K_{4\rho}(\mfz)}^-}$, where $\be$ is given in Remark \ref{high_int_remark}. For any $1 \le s_1 < s_2 \le 2$ and any $\al \ge 1$ satisfying \eqref{cv2-3}, we consider $Q_{\rho_i}^{\al_i}(\mfz_i)$, $i=1,2,\cdots,$ selected in the previous subsection, with 
\begin{equation}\label{cv3-1}
\al_i := \al^{\frac{q_{K_{4\rho}(\mfz)}^-}{q(\mfz_i)}} \quad \text{and} \quad 60\chi \rho_i \le (s_2 - s_1)\rho \le \rho,
\end{equation}
where $\chi$ is given in Lemma \ref{cv argument}.

We divide into the two cases: $Q_{4\chi \rho_i}^{\al_i}(\mfz_i) \subset \OO_T$ and $Q_{4\chi \rho_i}^{\al_i}(\mfz_i) \not\subset \OO_T$. We only consider the boundary case $Q_{4\chi \rho_i}^{\al_i}(\mfz_i) \not\subset \OO_T$. The interior case $Q_{4\chi \rho_i}^{\al_i}(\mfz_i) \subset \OO_T$ can be proved in a similar way. 

Since $Q_{4\chi \rho_i}^{\al_i}(\mfz_i) \not\subset \OO_T$, there exists a boundary point $(\tilde{\mfx}_i, \mft_i) \in \gh{\partial\OO \times (-T,T)} \cap Q_{4\chi \rho_i}^{\al_i}(\mfz_i)$. Since $(\pp,\aa,\Om)$ is $(\ga,\bs_0)$-vanishing, there exists a new coordinate system modulo rotation and
translation, which we still denote by $\{x_1, \cdots, x_n, t\}$,
with the origin is $(\tilde{\mfx}_i, \mft_i) + 56\chi \gamma \rho_i e_n$, where $e_n
:= (0, \cdots, 0,1)$ and 
\begin{equation*}
B_{\rho}^+(0) \ \subset \ \OO_{\rho}(0) \ \subset \ B_{\rho}(0) \cap \{(x, t) : x_n > - 112\chi \gamma \rho\} \quad \text{for any} \ \ 0 < \rho < 48\chi \rho_i.
\end{equation*}

Set $\tilde{\mfz}_i := (0,\mft_i)$. Since $|\mfx_i| \le |\mfx_i - \tilde{\mfx}_i| + |\tilde{\mfx}_i| \le (4+56\gamma)\chi \rho_i \le 11\chi \rho_i$, we have from \eqref{cv3-1} and \eqref{cv2-6} that 
\begin{equation*}\label{cv3-2.5}
K_{\chi \rho_i}^{\al_i}(\mfz_i) \subset K_{12\chi \rho_i}^{\al_i}(\tilde{\mfz}_i) \subset K_{48\chi \rho_i}^{\al_i}(\tilde{\mfz}_i) \subset K_{60\chi \rho_i}^{\al_i}(\mfz_i) \subset K_{s_2 \rho}(\mfz) \subset K_{4\rho}(\mfz),
\end{equation*}
and thus
\begin{equation*}\label{cv3-3}
p_{K_{48\chi \rho_i}^{\al_i}(\tilde{\mfz}_i)}^+ - p_{K_{48\chi \rho_i}^{\al_i}(\tilde{\mfz}_i)}^- \le p_{K_{4\rho}(\mfz)}^+ - p_{K_{4\rho}(\mfz)}^- \le \modp(2\rho_0) \quad \text{and} \quad q_{K_{48\chi \rho_i}^{\al_i}(\tilde{\mfz}_i)}^+ - q_{K_{48\chi \rho_i}^{\al_i}(\tilde{\mfz}_i)}^- \le \modq(2\rho_0).
\end{equation*}
We employ \eqref{cv2-7} with taking $c_a =2(48)^{n+2}$ to derive
\begin{equation*}\label{cv3-4}
\al_i \le \al < c_a\mgh{ \mint{K_{48\chi \rho_i}^{\al_i}(\tilde{\mfz}_i)}{|\nabla u|^{\frac{p(z)(1-\beta)q(z)}{q_{K_{4\rho}(\mfz)}^-}}}{dz} + \frac{1}{\gamma} \gh{\mint{K_{48\chi \rho_i}^{\al_i}(\tilde{\mfz}_i)}{|\bff|^{\frac{p(z)(1-\beta)q(z)(1+\sigma)}{q_{K_{4\rho}(\mfz)}^-}}}{dz}}^{\frac{1}{1+\sigma}} },
\end{equation*}
where $\be$ and $\sigma$ are given in Remark \ref{high_int_remark}. Now applying Lemma \ref{cv lemma} with $\al=\al_i$, $\tilde{\rho}=12\chi \rho_i$ and $\tilde{\mfz}=\tilde{\mfz}_i$,  we obtain
\begin{equation*}\label{cv3-5}
\begin{array}{c}
 \al_i^{-\frac{n}{p(\tilde{\mfz}_i)} + \frac{nd}{2} + d} \le \Ga^2 (12\chi \rho_i)^{-(n+2)}, \qquad  p_{K_{48\chi \rho_i}^{\al_i}(\tilde{\mfz}_i)}^+ - p_{K_{48\chi \rho_i}^{\al_i}(\tilde{\mfz}_i)}^- \le \modp(12\Ga \chi \rho_i),\\
 \al_i^{p_{K_{48\chi \rho_i}^{\al_i}(\tilde{\mfz}_i)}^+ - p_{K_{48\chi \rho_i}^{\al_i}(\tilde{\mfz}_i)}^-} \le c_p,
\end{array}
\end{equation*}
and 
\begin{gather}
\label{cv3-6} q_{K_{48\chi \rho_i}^{\al_i}(\tilde{\mfz}_i)}^+ - q_{K_{48\chi \rho_i}^{\al_i}(\tilde{\mfz}_i)}^- \le \modq(12\Ga \chi \rho_i), \quad \al_i^{q_{K_{48\chi \rho_i}^{\al_i}(\tilde{\mfz}_i)}^+ - q_{K_{48\chi \rho_i}^{\al_i}(\tilde{\mfz}_i)}^-} \le c_q,
\end{gather}
where $c_p$ and $c_q$ are given in \eqref{cv0-3} and \eqref{cv0-4}, respectively. We can now  directly compute to get
\begin{equation*}\label{cv3-7}
\begin{array}{rcl}
\mint{K_{48\chi \rho_i}^{\al_i}(\tilde{\mfz}_i)}{|\nabla u|^{p(z)(1-\beta)}}{dz} &\le & \gh{\mint{K_{48\chi \rho_i}^{\al_i}(\tilde{\mfz}_i)}{|\nabla u|^{\frac{p(z)(1-\beta)q(z)}{q_{K_{4\rho}(\mfz)}^-}}}{dz} + 1}^{\frac{q_{K_{4\rho}(\mfz)}^-}{q_{K_{48\chi \rho_i}^{\al_i}(\tilde{\mfz}_i)}^-}}\\ 
& \overset{\eqref{cv2-8}}{\apprle}&  \al^{\frac{q_{K_{4\rho}(\mfz)}^-}{q_{K_{48\chi \rho_i}^{\al_i}(\tilde{\mfz}_i)}^-}} = \al^{\frac{q_{K_{4\rho}(\mfz)}^- \gh{q(\mfz_i) - q_{K_{48\chi \rho_i}^{\al_i}(\tilde{\mfz}_i)}^-}}{q_{K_{48\chi \rho_i}^{\al_i}(\tilde{\mfz}_i)}^- q(\mfz_i)}} \al^{\frac{q_{K_{4\rho}(\mfz)}^-}{q(\mfz_i)}}\\
& \overset{\eqref{cv3-1}}{\le}&  \al_i^{\frac{q_{K_{48\chi \rho_i}^{\al_i}(\tilde{\mfz}_i)}^+ - q_{K_{48\chi \rho_i}^{\al_i}(\tilde{\mfz}_i)}^-}{q^-}} \al_i \overset{\eqref{cv3-6}}{\apprle} \al_i.
\end{array}
\end{equation*}
Proceeding similarly, we also get
\begin{equation*}\label{cv3-7.1}
\begin{aligned}
\mint{K_{48\chi \rho_i}^{\al_i}(\tilde{\mfz}_i)}{|\bff|^{p(z)(1-\beta)}}{dz} \apprle \gamma^{\frac{q^-}{q^+}}\al_i.
\end{aligned}
\end{equation*}
Therefore, applying Theorem \ref{first_diff_thm}, Theorem \ref{second_diff_thm}, and Lemma \ref{existence_ov}, we have the following lemma:
\begin{lemma}\label{cv3_lemma1}
For any $\ve \in (0,1)$, there exists $\gamma = \gamma(n,\La_0,\La_1,\plog,\qlog,\ve) > 0$ satisfying $Q_{4\chi \rho_i}^{\al_i}(\mfz_i) \not\subset \OO_T$ such that
\begin{gather*}
\label{cv3-8} \mint{K_{12\chi \rho_i}^{\al_i}(\tilde{\mfz}_i)}{|\nabla u-\nabla w|^{p(z)(1-\beta)}}{dz} \le \ve \al_i, \qquad \mint{K_{12\chi \rho_i}^{\al_i}(\tilde{\mfz}_i)}{|\nabla w-\nabla\bar{V}|^{p(\mfz)(1-\beta)}}{dz} \le \ve \al_i,\\
\label{cv3-9} \mint{K_{24\chi \rho_i}^{\al_i}(\tilde{\mfz}_i)}{|\nabla w|^{p(z)(1-\beta)}}{dz} \le \ve \al_i, \qquad \text{and} \quad \Norm{\nabla\bar{V}}_{L^\infty(K_{12\chi \rho_i}^{\al_i}(\tilde{\mfz}_i),\RR^n)}^{p(\mfz)(1-\beta)} \apprle \al_i.
\end{gather*}
\end{lemma}

From a similar way in \cite[Corollary 5.6]{BO}, we can also obtain from Lemma \ref{cv3_lemma1} that the following estimates:
\begin{lemma}\label{cv3_lemma2}
Under the assumptions as in Lemma \ref{cv3_lemma1}, we have
\begin{equation}\label{cv3-10}
\mint{K_{12\chi \rho_i}^{\al_i}(\tilde{\mfz}_i)}{|\nabla u-\nabla \bar{V}|^{\frac{p(z)(1-\beta)q(z)}{q_{K_{4\rho}(\mfz)}^-}}}{dz} \le \ve \al \quad \text{and} \quad \Norm{|\nabla\bar{V}|^{\frac{p(\cdot)(1-\beta)q(\cdot)}{q_{K_{4\rho}(\mfz)}^-}}}_{L^\infty(K_{12\chi \rho_i}^{\al_i}(\tilde{\mfz}_i),\RR^n)} \le \al c_2 
\end{equation}
for some constant $c_2 = c_2(n,\La_0,\La_1,\plog,\qlog) \ge 1$.
\end{lemma}

We now estimate the integration of $|\nabla u|^{\frac{p(\cdot)(1-\beta)q(\cdot)}{q_{K_{4\rho}(\mfz)}^-}}$ on the upper-level set $E(s_1,B\al)$, where 
\begin{equation}\label{cv4-1}
B := 2^{\frac{p^+ (1-\be) q^+}{q^-}} c_2 \ge 1,
\end{equation}
and $c_2$ is given in Lemma \ref{cv3_lemma2}. Recalling \eqref{cv2-2}, it follows from \eqref{cv2-6} that
\begin{equation*}\label{cv4-2}
E(s_1, B\al) \setminus N \subset E(s_1, \al) \setminus N \subset \bigcup_{i=1}^{\infty} K_{\chi \rho_i}^{\al_i}(\mfz_i) \subset K_{s_2 \rho}(\mfz),
\end{equation*}
and
\begin{equation*}\label{cv4-3}
\integral{E(s_1, B\al)}{|\nabla u|^{\frac{p(z)(1-\beta)q(z)}{q_{K_{4\rho}(\mfz)}^-}}}{dz} \le \sum_{i=1}^{\infty} \integral{E(s_1, B\al) \cap K_{\chi \rho_i}^{\al_i}(\mfz_i)}{|\nabla u|^{\frac{p(z)(1-\beta)q(z)}{q_{K_{4\rho}(\mfz)}^-}}}{dz}.
\end{equation*}
We discover that for any $z \in E(s_1, B\al) \cap K_{12\chi \rho_i}^{\al_i}(\tilde{\mfz}_i)$,

\begin{equation*}\label{cv4-4}
\begin{array}{rcl}
|\nabla u|^{\frac{p(z)(1-\beta)q(z)}{q_{K_{4\rho}(\mfz)}^-}} &\overset{\eqref{cv3-10}}{\le} &2^{\frac{p^+ (1-\be) q^+}{q^-}-1} \gh{|\nabla u - \nabla\bar{V}|^{\frac{p(z)(1-\beta)q(z)}{q_{K_{4\rho}(\mfz)}^-}} + c_2 \al}\\
&\overset{\eqref{cv2-2},\eqref{cv4-1}}{\le}& 2^{\frac{p^+ (1-\be) q^+}{q^-}-1} |\nabla u - \nabla\bar{V}|^{\frac{p(z)(1-\beta)q(z)}{q_{K_{4\rho}(\mfz)}^-}} + \frac12 |\nabla u|^{\frac{p(z)(1-\beta)q(z)}{q_{K_{4\rho}(\mfz)}^-}}.
\end{array}
\end{equation*}
Then this implies
\begin{equation*}\label{cv4-5}
\begin{array}{rcl}
\integral{E(s_1, B\al) \cap K_{\chi \rho_i}^{\al_i}(\mfz_i)}{|\nabla u|^{\frac{p(z)(1-\beta)q(z)}{q_{K_{4\rho}(\mfz)}^-}}}{dz} &\le& 2^{\frac{p^+ (1-\be) q^+}{q^-}} \integral{E(s_1, B\al) \cap K_{12\chi \rho_i}^{\al_i}(\tilde{\mfz}_i)}{|\nabla u- \nabla\bar{V}|^{\frac{p(z)(1-\beta)q(z)}{q_{K_{4\rho}(\mfz)}^-}}}{dz}\\
&\overset{\eqref{measure_one}, \eqref{cv3-10}}{\apprle} & \ve \al |K_{\chi \rho_i}^{\al_i}(\mfz_i)|,
\end{array}
\end{equation*}
that is,
\begin{equation}\label{cv4-6}
\integral{E(s_1, B\al)}{|\nabla u|^{\frac{p(z)(1-\beta)q(z)}{q_{K_{4\rho}(\mfz)}^-}}}{dz} \apprle \ve \al \sum_{i=1}^{\infty} |K_{\chi \rho_i}^{\al_i}(\mfz_i)|.
\end{equation}

On the other hand, we know from \eqref{cv2-7} that either
\begin{equation*}\label{cv4-7}
\frac{\al}{2} \le \mint{K_{\rho_i}^{\al_i}(\mfz_i)}{|\nabla u|^{\frac{p(z)(1-\beta)q(z)}{q_{K_{4\rho}(\mfz)}^-}}}{dz}  \txt{or}  \frac{\al}{2} \le \frac{1}{\gamma} \gh{\mint{K_{\rho_i}^{\al_i}(\mfz_i)}{|\bff|^{\frac{p(z)(1-\beta)q(z)(1+\sigma)}{q_{K_{4\rho}(\mfz)}^-}}}{dz}}^{\frac{1}{1+\sigma}},
\end{equation*}
and then we calculate 
\begin{equation}\label{cv4-8}
\begin{array}{rcl}
|K_{\tilde{\rho}}^{\al_i}(\mfz_i)| &\le & \frac{4}{\al} \iint_{\mgh{z \in K_{\rho_i}^{\al_i}(\mfz_i) : |\nabla u(z)|^\frac{p(z)(1-\beta)q(z)}{q_{K_{4\rho}(\mfz)}^-} > \frac{\al}{4}}}{|\nabla u|^{\frac{p(z)(1-\beta)q(z)}{q_{K_{4\rho}(\mfz)}^-}}}{dz}\\
&& \quad + \gh{\frac{4}{\gamma\al}}^{1+\sigma} \iint_{\mgh{z \in K_{\rho_i}^{\al_i}(\mfz_i) : |\bff|^\frac{p(z)(1-\beta)q(z)}{q_{K_{4\rho}(\mfz)}^-} > \frac{\ga\al}{4}}}{|\bff|^{\frac{p(z)(1-\beta)q(z)(1+\sigma)}{q_{K_{4\rho}(\mfz)}^-}}}{dz}.
\end{array}
\end{equation}

Plugging \eqref{cv4-8} into \eqref{cv4-6} and using the fact that the family $\mgh{K_{\rho_i}^{\al_i}(\mfz_i)}_{i=1}^\infty \subset K_{s_2 \rho}(\mfz)$ is pairwise disjoint, we conclude
\begin{equation}\label{cv4-9}
\begin{aligned}
\integral{E(s_1, B\al)}{|\nabla u|^{\frac{p(z)(1-\beta)q(z)}{q_{K_{4\rho}(\mfz)}^-}}}{dz} &\apprle \ve \integral{E(s_2,\frac{\al}{4})}{|\nabla u|^{\frac{p(z)(1-\beta)q(z)}{q_{K_{4\rho}(\mfz)}^-}}}{dz}\\
&\quad + \frac{\ve}{\gamma^{1+\sigma} \al^{\sigma}} \integral{\mgh{z \in K_{s_2 \rho}(\mfz) : |\bff|^\frac{p(z)(1-\beta)q(z)}{q_{K_{4\rho}(\mfz)}^-} > \frac{\ga\al}{4}}}{|\bff|^{\frac{p(z)(1-\beta)q(z)(1+\sigma)}{q_{K_{4\rho}(\mfz)}^-}}}{dz}.
\end{aligned}
\end{equation}


\section{Proof of the main results}
\label{nine}

\subsection{Proof of Theorem \ref{main_theorem1}}\label{proof of local estimate}

Fix any $\mfz \in \Om_T$, $\be \in (0,{\be_0})$, and $\rho \in (0,\rho_0]$, where $\be_0$ and $\rho_0$ are given in Section \ref{def_be_0} and Remark \ref{remark_radius}, respectively. Define the constant $\bm$ by
\begin{equation}\label{def_MM}
\bm := \iint_{\Om_T} \lbr[[]|\bff|^{p(z)\max\mgh{(1-\beta)q^-,1}} + 1\rbr[]] \ dz + 1.
\end{equation}
Clearly, we have $\bm \apprge \bm_0 \ge 1$, where $\bm_0$ is given in \eqref{def_M_0}. Putting $\rho_0 =\frac{1}{C_0 \bm}$ for some constant $C_0 = {C_0}_{(\La_0,\La_1,\plog, \qlog, n, \bs_0)}>0$, we can apply all results in Section \ref{eight}.

For $k >0$, we define the truncation of $|\nabla u|^{\frac{p(\cdot)(1-\beta)q(\cdot)}{q_{K_{4\rho}(\mfz)}^-}}$ as
\begin{equation*}
\gh{|\nabla u|^{\frac{p(\cdot)(1-\beta)q(\cdot)}{q_{K_{4\rho}(\mfz)}^-}}}_k (z) := \min \mgh{|\nabla u|^{\frac{p(z)(1-\beta)q(z)}{q_{K_{4\rho}(\mfz)}^-}}, k}.
\end{equation*}
Let $1 \le s_1 < s_2 \le 2$. Lemma \ref{useful int} implies that for sufficiently large $k > 1$,
\begin{equation}\label{ma1-1}
\begin{aligned}
&\hspace*{-2cm}\integral{K_{s_1 \rho}(\mfz)}{\gh{|\nabla u|^{\frac{p(\cdot)(1-\beta)q(\cdot)}{q_{K_{4\rho}(\mfz)}^-}}}_k^{q_{K_{4\rho}(\mfz)}^- -1} |\nabla u|^{\frac{p(z)(1-\beta)q(z)}{q_{K_{4\rho}(\mfz)}^-}}}{dz}\\
&\quad = \gh{q_{K_{4\rho}(\mfz)}^- -1} \int_0^k \al^{q_{K_{4\rho}(\mfz)}^- -2} \integral{E(s_1,\al)}{|\nabla u|^{\frac{p(z)(1-\beta)q(z)}{q_{K_{4\rho}(\mfz)}^-}}}{dz d\al}\\
&\quad = \gh{q_{K_{4\rho}(\mfz)}^- -1} B^{q_{K_{4\rho}(\mfz)}^- -1} \int_0^{\frac{k}{B}} \al^{q_{K_{4\rho}(\mfz)}^- -2} \integral{E(s_1,B\al)}{|\nabla u|^{\frac{p(z)(1-\beta)q(z)}{q_{K_{4\rho}(\mfz)}^-}}}{dz d\al}\\
&\quad = \gh{q_{K_{4\rho}(\mfz)}^- -1} B^{q_{K_{4\rho}(\mfz)}^- -1} \int_0^{A\tilde{\al}} \al^{q_{K_{4\rho}(\mfz)}^- -2} \ d\al \integral{K_{s_1 \rho}(\mfz)}{|\nabla u|^{\frac{p(z)(1-\beta)q(z)}{q_{K_{4\rho}(\mfz)}^-}}}{dzd\al}\\
&\qquad \ + \gh{q_{K_{4\rho}(\mfz)}^- -1} B^{q_{K_{4\rho}(\mfz)}^- -1} \int_{A\tilde{\al}}^{\frac{k}{B}} \al^{q_{K_{4\rho}(\mfz)}^- -2} \integral{E(s_1,B\al)}{|\nabla u|^{\frac{p(z)(1-\beta)q(z)}{q_{K_{4\rho}(\mfz)}^-}}}{dz d\al}\\
&\quad =: I_1 + I_2,
\end{aligned}
\end{equation}
where $\tilde{\al}$, $A$, $B$, and $E(s_1,\al)$ are given in \eqref{cv2-1}, \eqref{cv2-3}, \eqref{cv4-1}, and \eqref{cv2-2}, respectively. For $I_1$, we compute directly that
\begin{equation}\label{ma1-2}
I_1 \le \gh{A B \tilde{\al}}^{q_{K_{4\rho}(\mfz)}^- -1} \integral{K_{2\rho}(\mfz)}{|\nabla u|^{\frac{p(z)(1-\beta)q(z)}{q_{K_{4\rho}(\mfz)}^-}}}{dz} \apprle \frac{ \tilde{\al}^{q_{K_{4\rho}(\mfz)}^- -1}}{(s_2-s_1)^{(n+2)(q^+ -1)\vt_{K_{4\rho}(\mfz)}^+}} \integral{K_{2\rho}(\mfz)}{|\nabla u|^{\frac{p(z)(1-\beta)q(z)}{q_{K_{4\rho}(\mfz)}^-}}}{dz}.
\end{equation}
For $I_2$, it follows from \eqref{cv4-9} and Lemma \ref{useful int} that
\begin{equation}\label{ma1-3}
\begin{aligned}
I_2 \apprle \ve \integral{K_{s_2 \rho}(\mfz)}{\gh{|\nabla u|^{\frac{p(\cdot)(1-\beta)q(\cdot)}{q_{K_{4\rho}(\mfz)}^-}}}_k^{q_{K_{4\rho}(\mfz)}^- -1} |\nabla u|^{\frac{p(z)(1-\beta)q(z)}{q_{K_{4\rho}(\mfz)}^-}}}{dz} + \ve \gamma^{-q_{K_{4\rho}(\mfz)}^-} \integral{K_{s_2 \rho}(\mfz)}{|\bff|^{p(z)(1-\beta)q(z)}}{dz}.
\end{aligned}
\end{equation}
Here we choose $\ve$ small enough which also  determines $\gamma_0$.

Plugging \eqref{ma1-2} and \eqref{ma1-3} into \eqref{ma1-1} and applying Lemma \ref{useful tech}, we deduce
\begin{equation*}\label{ma1-4}
\begin{aligned}
&\integral{K_{\rho}(\mfz)}{\gh{|\nabla u|^{\frac{p(\cdot)(1-\beta)q(\cdot)}{q_{K_{4\rho}(\mfz)}^-}}}_k^{q_{K_{4\rho}(\mfz)}^- -1} |\nabla u|^{\frac{p(z)(1-\beta)q(z)}{q_{K_{4\rho}(\mfz)}^-}}}{dz}\\
&\qquad \apprle \tilde{\al}^{q_{K_{4\rho}(\mfz)}^- -1} \integral{K_{2\rho}(\mfz)}{|\nabla u|^{\frac{p(z)(1-\beta)q(z)}{q_{K_{4\rho}(\mfz)}^-}}}{dz} +  \integral{K_{2\rho}(\mfz)}{|\bff|^{p(z)(1-\beta)q(z)}}{dz}.
\end{aligned}
\end{equation*}
As $k \to \infty$, we have
\begin{equation}\label{ma1-5}
\integral{K_{\rho}(\mfz)}{|\nabla u|^{p(z)(1-\beta)q(z)}}{dz} \apprle \tilde{\al}^{q_{K_{4\rho}(\mfz)}^- -1} \integral{K_{2\rho}(\mfz)}{|\nabla u|^{\frac{p(z)(1-\beta)q(z)}{q_{K_{4\rho}(\mfz)}^-}}}{dz} +  \integral{K_{2\rho}(\mfz)}{|\bff|^{p(z)(1-\beta)q(z)}}{dz}.
\end{equation}

On the other hand, we note that
\begin{equation}\label{ma1-6}
\begin{array}{rcl}
\gh{\mint{K_{2\rho}(\mfz)}{\bgh{|\nabla u|^{p(z)(1-\beta)}+|\bff|^{p(z)(1-\beta)q^-}}}{dz}}^{\modq(8\rho)} &\overset{\eqref{def_MM},\eqref{size_u_f}}{\apprle} &\gh{\frac{\bm}{|K_{2\rho}(\mfz)|}}^{\modq(8\rho)}\\
&\apprle &\gh{\frac{1}{8\rho}}^{(n+3)\modq(8\rho)} \overset{\text{Definition \ref{definition_p_log}}}{\apprle} 1,
\end{array}
\end{equation}
and similarly
\begin{equation}\label{ma1-6.5}
\gh{\mint{K_{2\rho}(\mfz)}{\bgh{|\nabla u|^{p(z)(1-\beta)}+|\bff|^{p(z)(1-\beta)q^-}}}{dz}}^{\modp(8\rho)} \apprle 1.
\end{equation} 
Recalling \eqref{vt_def} and \eqref{vt_max_def}, it follows 
\begin{equation}\label{ma1-7}
\vt_{K_{4\rho}(\mfz)}^+ - \vt(\mfz) \apprle \modp(8\rho).
\end{equation}
Then we see
\begin{equation}\label{ma1-8}
\begin{array}{rcl}
&&\hspace*{-2cm}\mint{K_{2\rho}(\mfz)}{|\nabla u|^{\frac{p(z)(1-\beta)q(z)}{q_{K_{4\rho}(\mfz)}^-}}}{dz} \overset{\eqref{ca1}}{\apprle} \mint{K_{2\rho}(\mfz)}{|\nabla u|^{p(z)(1-\beta) \gh{1+\frac{\modq(8\rho)}{q^-}}}}{dz} + 1\\
& \overset{\eqref{hi1}}{\apprle}& \gh{\mint{K_{4\rho}(\mfz)}{\gh{|\nabla u|+|\bff|}^{p(z)(1-\beta)}}{dz}}^{1+\frac{\modq(8\rho)}{q^-}\tilde{\theta}} + \mint{K_{4\rho}(\mfz)}{|\bff|^{p(z)(1-\beta)\gh{1+\frac{\modq(8\rho)}{q^-}}}}{dz} + 1\\
& \overset{\eqref{ma1-6},\eqref{ca2}}{\apprle} & \mint{K_{4\rho}(\mfz)}{|\nabla u|^{p(z)(1-\beta)}}{dz} + \gh{\mint{K_{4\rho}(\mfz)}{|\bff|^{p(z)(1-\beta)q^-}}{dz}}^{\frac{1}{q^-}} + 1\\
& \overset{\eqref{ma1-6}}{\apprle} & \mint{K_{4\rho}(\mfz)}{|\nabla u|^{p(z)(1-\beta)}}{dz} + \gh{\mint{K_{4\rho}(\mfz)}{|\bff|^{p(z)(1-\beta)q(z)}}{dz}}^{\frac{1}{q(\mfz)}} + 1,
\end{array}
\end{equation}
and 
\begin{equation}\label{ma1-9}
\begin{aligned}	
\tilde{\al}^{q_{K_{4\rho}(\mfz)}^- -1} &\overset{\eqref{cv2-1}}{\apprle} \mgh{\mint{K_{2\rho}(\mfz)}{|\nabla u|^{\frac{p(z)(1-\beta)q(z)}{q_{K_{4\rho}(\mfz)}^-}}}{dz} + \gh{\mint{K_{2\rho}(\mfz)}{|\bff|^{\frac{p(z)(1-\beta)q(z)(1+\sigma)}{q_{K_{4\rho}(\mfz)}^-}}}{dz}}^{\frac{1}{1+\sigma}} + 1}^{\vt_{K_{4\rho}(\mfz)}^+ \gh{q_{K_{4\rho}(\mfz)}^- -1}}\\
& \overset{\eqref{ca2}, \eqref{ma1-6}-\eqref{ma1-8}}{\apprle} \mgh{\mint{K_{4\rho}(\mfz)}{|\nabla u|^{p(z)(1-\beta)}}{dz} + \gh{\mint{K_{4\rho}(\mfz)}{|\bff|^{p(z)(1-\beta)q(z)}}{dz}}^{\frac{1}{q(\mfz)}} + 1}^{\vt(\mfz) \gh{q(\mfz) -1}}.
\end{aligned}
\end{equation}

We finally obtain from \eqref{ma1-5}, \eqref{ma1-8}, and \eqref{ma1-9} that
\begin{equation}\label{ma1-10}
\mint{K_{\rho}(\mfz)}{|\nabla u|^{p(z)(1-\beta)q(z)}}{dz} \apprle \mgh{ \mint{K_{4\rho}(\mfz)}{|\nabla u|^{p(z)(1-\beta)}}{dz} + \gh{\mint{K_{4\rho}(\mfz)}{|\bff|^{p(z)(1-\beta)q(z)}}{dz}}^{\frac{1}{q(\mfz)}} + 1}^{1+\vt(\mfz) \gh{q(\mfz) -1}},
\end{equation}
which completes the proof.

\subsection{Proof of Theorem \ref{main_theorem2}}
We extend the local estimate \eqref{ma1-10} up to the boundary. We first choose $\rho = \frac{1}{C_0 \bm}$, where $C_0$ and $\bm$ are given in Section \ref{proof of local estimate}. From the standard covering argument, we can find finitely many disjoint parabolic cylinders $\mgh{Q_{\frac{\rho}{3}}(\mfz_k)}_{k=1}^m$, $\mfz_k \in \OO_T$, such that $\bar{\OO}_T \subset \bigcup_{k=1}^m Q_{\rho}(\mfz_k)$. Note that for an integrable function $f$, we have $\sum_{k=1}^m \integral{K_{4\rho}(\mfz_k)}{f}{dz} \apprle_{(n)} \integral{\OO_T}{f}{dz}.$ 

Then it follows from \eqref{ma1-10} that
\begin{equation}\label{ge-1}
\begin{aligned}
&\integral{\OO_T}{|\nabla u|^{p(z)(1-\beta)q(z)}}{dz} \le \sum_{k=1}^m \integral{K_{\rho}(\mfz_k)}{|\nabla u|^{p(z)(1-\beta)q(z)}}{dz}\\
&\quad \lesssim \rho^{n+2} \mgh{\rho^{-(n+2)q^+} \gh{\integral{\OO_T}{\bgh{|\nabla u|^{p(z)(1-\beta)}+1}}{dz}}^{q^+} + \rho^{-(n+2)}\integral{\OO_T}{\bgh{|\bff|^{p(z)(1-\beta)q(z)}+1}}{dz}}^{1+\vt^+ \gh{q^+ -1}},
\end{aligned}
\end{equation}
where $\vt^+ :=\sup_{z \in\OO_T} \vt(z)$.

Let $M^+$ and $M^-$ be any two  constants such that additionally we have $1 < M^- \le q^- \leq q(\cdot) \le q^+\leq M^+<\infty$. In the proof of 
Theorem \ref{main_theorem1} in Section \ref{proof of local estimate}, we see that ${\beta_0}$ can be chosen to depend on $M^+$ instead of $q^{\pm}$. This, in particular, implies that we can choose $\beta_0$ independent of $M^-$. 


Let us now define $r(z) :=  \frac{p(z)(1-\beta)}{p(z)} q(z)$ for $\beta \in (0,\beta_0)$ (it is important to note that we cannot take $\beta=0$), then we trivially have
$$r^- \ge \gh{\min_{z\in\OO_T}\frac{p(z)(1-\beta)}{p(z)}} M^- \quad \text{and} \quad r^+ \le \gh{\max_{z\in\OO_T}\frac{p(z)(1-\beta)}{p(z)}} M^+.$$
Note that $r(\cdot)$ is clearly log-H\"{o}lder continuous with the log-H\"older constants equivalent to the ones satisfied by $\qq$.

Since all the estimates above are independent of $M^-$ and $\beta_0$ is is independent of $M^-$, we can choose $M^-$ small such that $\gh{\min_{z\in\OO_T}\frac{p(z)(1-\beta)}{p(z)}} M^- \leq 1$. This in particular allows $r^-=1$. 

For this choice of the exponent $r(\cdot)$, we conclude from \eqref{ge-1}, \eqref{def_MM}, and the definition of $\rho$ that
\begin{equation*}\label{ge-3}
\begin{aligned}
\integral{\OO_T}{|\nabla u|^{p(z)r(z)}}{dz} &\le C \mgh{\gh{\integral{\OO_T}{|\bff|^{p(z)r(z)}}{dz}}^{\gh{1+\vt^+ \gh{q^+ -1}}(n+3)q^+ - (n+2)} + 1}\\
&\le C \mgh{\gh{\integral{\OO_T}{|\bff|^{p(z)r(z)}}{dz}}^{\gh{1+\vt^+ \gh{M^+ -1}}(n+3)M^+ - (n+2)} +1},
\end{aligned}
\end{equation*}
for some constant $C=C_{(\La_0,\La_1,\plog, r^{\pm}_{\log}, M^+, n, \OO_T,\bs_0)}>0$, which completes the proof.


%
\begin{appendices}
\addtocontents{toc}{\def\protect\cftchappresnum{}} 

\section{The method of Lipschitz truncation - first difference estimate}
\label{lipschitz_truncation}
In this appendix, following the techniques developed in \cite{adimurthi2018sharp} which were originally pioneered in \cite{KL},  we will develop a modified version of Lipschitz truncation suited to our needs.
Recall that $u$ is a weak solution of \eqref{basic_pde} and $w$ is a weak solution of \eqref{wapprox_int}. For this section, we only need to assume the following restrictions on the size of the region $K_{4\rho}^{\al}(\mfz)$: In particular, we will take $\tilde{\rho}_3$  small such that \descref{R6}{R6} and \descref{R4}{R4} are applicable.

To simplify the notation, we will define 
\begin{equation}
\label{def_s}
s:= \scalet{\al}{\mfz} (4\rho)^2.
\end{equation}


Let us now collect some well known results that will be needed in the course of the proof. The first lemma is a time localised version of the parabolic Poincar\'e inequality (see \cite[Lemma 4.2]{adimurthi2018weight} for the proof):
\begin{lemma}
 \label{lemma_crucial_1}
 Let $f \in L^{\vt} (-T,T; W^{1,\vt}(\Om))$ with $\vt \in (1,\infty)$  and suppose that $\mathcal{B}_{r} \Subset \Om$ be compactly contained ball of radius $r>0$. Let $I \subset (-T,T)$ be a time interval  and $\rho(x,t) \in L^1(\mathcal{B}_r \times I)$ be any positive function such that $$\|\rho\|_{L^{\infty}(\mathcal{B}_r\times I)} \apprle_{(n)} \frac{|\mathcal{B}_r\times I|}{\|\rho\|_{L^1(\mathcal{B}_r\times I)}} $$ and $\mu(x) \in C_c^{\infty}(\mathcal{B}_r)$ be  such that $\int_{\mathcal{B}_r} \mu(x) \ dx = 1$ with $|\mu| \leq  \frac{C_{(n)}}{r^n}$ and $|\nabla \mu| \leq  \frac{C_{(n)}}{r^{n+1}}$, then there holds:
 \begin{equation*}
 \begin{array}{ll}
  \fiint_{\mathcal{B}_r \times I} \left|\frac{f(z)\lsb{\chi}{J} - \avgs{f\lsb{\chi}{J}}{\rho}}{r}\right|^{\vt} \ dz & \apprle_{(n,s,C^{\mu})} \fiint_{\mathcal{B}_r \times I} |\nabla f|^{\vt}\lsb{\chi}{J} \ dz + \sup_{t_1,t_2 \in I} \left| \frac{\avgs{f\lsb{\chi}{J}}{\mu}(t_2) - \avgs{f\lsb{\chi}{J}}{\mu}(t_1)}{r} \right|^{\vt}
  \end{array}
 \end{equation*}
where $\avgs{f\lsb{\chi}{J}}{\rho}:= \int_{\mathcal{B}_r\times I} f(z)\lsb{\chi}{J} \frac{\rho(z)}{\|\rho\|_{L^1(\mathcal{B}_r\times I)}} \ dz\ $, $\avgs{f\lsb{\chi}{J}}{\mu}(t_i) := \int_{\mathcal{B}_r} f(x,t_i) \mu(x) \lsb{\chi}{J} \ dx$ and $J \Subset (-\infty,\infty)$ is some fixed time-interval. 
\end{lemma}

\begin{lemma}
\label{lemma_crucial_2}
 For any  $h \in (0,2s)$ and let $\phi(x) \in C_c^{\infty}({\Om_{4\rho}^{\al}(\mfx)})$ and $\varphi(t) \in C^{\infty}(\mft-s,\infty)$ with $\varphi(\mft-s) = 0$ be a  non-negative function and $[u]_h,[w]_h$ be the Steklov average as defined in \eqref{stek1}. Then   the following estimate holds for any time interval $(t_1,t_2) \subset [\mft-s,\mft+s]$:
 \begin{equation*}
  \label{lemma_crucial_2_est}
  \begin{array}{rcl}
  |\avgs{{[u-w]_h \varphi}}{\phi} (t_2) - \avgs{{[u-w]_h\varphi}}{\phi}(t_1)| & \leq & \|\nabla \phi\|_{L^{\infty}{({\Om_{4\rho}^{\al}(\mfz)})}} \|\varphi\|_{L^{\infty}(t_1,t_2)} \iint_{{\Om_{4\rho}^{\al}(\mfx)} \times (t_1,t_2)} \abs{\aa(z,\nabla w) - \aa(z,\nabla u)} \ dz \\
   & &\qquad  +\|\nabla \phi\|_{L^{\infty}{({\Om_{4\rho}^{\al}(\mfz)})}} \|\varphi\|_{L^{\infty}(t_1,t_2)} \iint_{{\Om_{4\rho}^{\al}(\mfx)} \times (t_1,t_2)} [|\bff|^{\pp-1}]_h \ dz \\
   & &\qquad + \|\phi\|_{L^{\infty}{({\Om_{4\rho}^{\al}(\mfx)})}} \|\varphi'\|_{L^{\infty}(t_1,t_2)} \iint_{{\Om_{4\rho}^{\al}(\mfx)} \times (t_1,t_2)} |[u-w]_h| \ dz.
  \end{array}
 \end{equation*}
\end{lemma}

\subsection{Construction of test function}

Let us denote the following functions:
\begin{gather*}
v(z) := u(z) - w(z) \txt{and} \vh(z) := [u-w]_h(z), \label{def_vh}
\end{gather*}
where $[u-w]_h(z)$ denotes the usual Steklov average. It is easy to  see that $\vh \xrightarrow{h \searrow 0} v$. We also note that $v(z) = 0$ for $z \in \pa_p K_{4\rho}^{\al}(\mfz)$.  For some fixed $\mathfrak{q}$ such that $1 <\mathfrak{q}< \frac{p^-}{p^+-1}$, with $\mm$ as given in \eqref{par_max}, let us now define
\begin{equation}
\label{def_g_A}
\begin{array}{ll}
g(z) & := \mm\lbr \lbr[[]\frac{|v|}{\scalex{\al}{\mfz}\rho} + |\nabla u| + |\nabla w| + |\bff| + 1\rbr[]]^{\frac{p(z)}{\mathfrak{q}}} \lsb{\chi}{K_{4\rho}^{\al}(\mfz)}\rbr^{\mathfrak{q}(1-\be)}.
\end{array}
\end{equation}

For a fixed $\la \geq 1$, let us define the \emph{good set} by 
\begin{equation}
\label{elambda}
\elam := \{ z \in \RR^{n+1} : g(z) \leq \la^{1-\be}\}.
\end{equation}

For the rest of this section,  we will always assume that the following bound holds:
\begin{lemma}
\label{bound_rho}
With $\rho \leq \tilde{\rho}_3$, there holds
\[
\rho^{\pm \abs{p^+_{K_{4\rho}^{\al}(\mfz)}-p^-_{K_{4\rho}^{\al}(\mfz)}}} \leq C_{(\plog,n)}.
\]
\end{lemma}
\begin{proof}
Since $\pp \in \plog$, we have from Remark \ref{remark_def_p_log}, 
\begin{equation*}
p^+_{K_{4\rho}^{\al}(\mfz)} - p^-_{K_{4\rho}^{\al}(\mfz)} \leq \modp \lbr \max\left\{ 8\scalex{\al}{\mfz}\rho, \sqrt{\scalet{\al}{\mfz}32\rho^2} \right\} \rbr \leq \modp(32\rho).
\end{equation*}

Since $\rho \leq 1$, we only need to bound  $\rho^{-(p^+_{K_{4\rho}^{\al}(\mfz)}-p^-_{K_{4\rho}^{\al}(\mfz)})}$, which we do as follows:
\begin{equation*}
\rho^{p^-_{K_{4\rho}^{\al}(\mfz)}-p^+_{K_{4\rho}^{\al}(\mfz)}}  \leq \rho^{-32\modp(\rho) } = e^{32 \modp(\rho) \log \frac{1}{\rho}}\leq C_{(\plog,n)}.
\end{equation*}
This completes the proof of the lemma. 
\end{proof}

Following the ideas from \cite[Lemma 5.10]{adimurthi2018sharp}, we can obtain a Vitali-type covering lemma.
\begin{lemma}
\label{lemma_vitali}
Let $\la \geq 1$ be such that  \eqref{elambda} is given,   then for every $z \in K_{4\rho}^{\al}(\mfz)\setminus\elam$, consider the parabolic cylinders of the form 
\begin{equation*}
\label{q_rho_z}
Q_{\rho_z}^{\la}(z) := B_{\scalex{\la}{z}\rho_z}(x) \times (t - \scalet{\la}{z}\rho_z^2,t + \scalet{\la}{z}\rho_z^2)
 \end{equation*}
 where 
 $\rho_z := d^{\la}_z(z,\elam) := \inf_{\tz \in \elam} d^{\la}_z(z,\tz).$
 Let $\mathfrak{k} \in (0,1]$ be a given constant and consider the open covering of $K_{4\rho}^{\al}(\mfz) \setminus \elam$ given by 
 \begin{equation*}
 \label{covering_F}
 \mcf := \left\{ Q_{\mathfrak{k} \rho_z}^{\la}(z)\right\}_{z \in K_{4\rho}^{\al}(\mfz) \setminus \elam}.
 \end{equation*}

 Then there exists a universal constant $\mfi = \mfi(\plog,n)\geq 9$ and a countable disjoint subcollection $ \mcg := \{Q_{\rho_i}^{\la}(z_i)\}_{i \in \NN}\subset \mcf$  such that there holds
 \begin{equation*}
 \bigcup_{\mcf}Q_{\mathfrak{k}\rho_z}^{\la}(z) \subset \bigcup_{\mcg} Q_{\mfi \rho_{z_i}}^{\la} (z_i).
 \end{equation*}
\end{lemma}

We now have the following Whitney type covering whose proof is very similar to \cite[Lemma 5.11]{adimurthi2018sharp}. 
  \begin{lemma}
 \label{whitney_covering}
 There exists a universal constant $\de \in (0,1/4)$ such that for $\mcf$, a given  covering of $K_{4\rho}^{\al}(\mfz) \setminus \elam$ given by the cylinders:
$\mcf := \left\{Q_{\frac{\de}{\mfi}\rho_z}^{\la}(z)\right\}_{z \in K_{4\rho}^{\al}(\mfz) \setminus \elam},$
 where $\mfi$ is the constant from Lemma \ref{lemma_vitali}, there exists a countable subcollection $\mcg = \left\{ Q_{\de \rho_{z_i}}^{\la}(z_i)\right\}_{i \in \NN} = \{ Q_{r_i}^{\la}(z_i)\}_{i \in \NN} $ subordinate to the covering $\mcf$ such that 
 the following holds:
 
\begin{description}
\descitem{(W1)}{W1} $K_{4\rho}^{\al}(\mfz) \setminus \elam \subset \bigcup_{i \in \NN}Q_i $.
\descitem{(W2)}{W2} Each point $z \in K_{4\rho}^{\al}(\mfz) \setminus \elam$ belongs to utmost $C_{(n,\plog)}$ cylinders of the form $2Q_i$. 
\descitem{(W3)}{W3} There exists a constant $C=C_{(n,\plog)}$ such that for any two cylinders $Q_i$ and $Q_j$ with $2Q_i \cap 2Q_j \neq \emptyset$, there holds
\begin{equation*}
|B_i| \leq C |B_j| \leq C |B_i| \txt{and} |I_i| \leq C |I_j| \leq C |I_i|.
\end{equation*}
In particular, there holds $|Q_i| \approx_{(\plog,n)} |Q_j|$.
\descitem{(W4)}{W4} There exists a constant $\hat{c} = \hat{c}_{(n,\plog)}\geq 9$ such that for all $i \in \NN$, there holds:
\begin{equation*}
\hat{c} Q_i \subset \RR^{n+1} \setminus \elam \txt{and} 8\hat{c} Q_i \cap \elam \neq \emptyset.
\end{equation*}
\descitem{(W5)}{W5} For the constant $\hat{c}$ from above, there holds $2Q_i \cap 2Q_j \neq \emptyset$ implies $2Q_i \subset \hat{c}Q_j$. 
\end{description}
\end{lemma}

Once we have obtained the Whitney type covering lemma, we can now obtain the following standard partition of unity lemma:
\begin{lemma}
 \label{partition_unity}
 Subordinate to the covering $\mcg$ obtained in Lemma \ref{whitney_covering} , we obtain a partition of unity $\{ \psi\}_{i=1}^{\infty}$ on $\RR^{n+1} \setminus \elam$ that satisfies the following properties:
 \begin{itemize}
 \item $\sum_{i=1}^{\infty} \psi_i(z) = 1$ for all $z \in {K_{4\rho}^{\al}(\mfz)} \setminus\elam$.
 \item $\psi_i \in C_c^{\infty}(2Q_i)$.
 \item $\|\psi_i\|_{\infty} + \scalex{\la}{z_i} r_i \| \nabla \psi_i\|_{\infty} + \scalet{\la}{z_i} r_i^2 \| \pa_t \psi_i\|_{\infty} \leq C_{(\plog,n)}$ where we have used the notation $r_i := \de \rho_{z_i}$ which is the parabolic radius of $Q_i$ with respect to the metric $d_{z_i}^{\la}$ (see Lemma \ref{whitney_covering} for the notation).
 \item $\psi_i \geq C_{(\plog,n)}$ on $Q_i$. 
 \end{itemize}
 \end{lemma}

 Before we end this subsection, let us recall the following useful bound that will be used throughout this section. For a proof, see the proof of \cite[Lemma 5.10, (5.23)]{adimurthi2018sharp}.
 \begin{equation}
 \label{2.2.28-1}
 \la^{p^+_{2Q_i} - p^-_{2Q_i}} \leq C_{(\plog,n)}.
 \end{equation}

 \subsection{Construction of Lipschitz truncation function}

 Let us first clarify some of the notation that will subsequently be used in the rest of this section: for  $\hat{c}$ from \descref{W4}{W4}, we denote
 \begin{equation*}
 \label{def_qihat}
 \htq_i := \hat{c}Q_i = Q_{\hat{r}_i}^{\la}(z_i), \txt{where} \hat{r}_i := \hat{c} r_i.
 \end{equation*}
We shall also use the notation 
\begin{equation*}
\label{mcii}
\mci(i) := \{ j \in \NN : \spt(\psi_i) \cap \spt(\psi_j) \neq \emptyset\} \txt{and} \mci_z := \{ j \in \NN: z \in \spt(\psi_j) \}.
\end{equation*}

We are now ready to construct  the Lipschitz truncation function:
 \begin{equation}
 \label{lipschitz_function}
 \vlh(z) := \vh(z) - \sum_{i} \psi_i(z) \lbr \vh(z) - \vh^i\rbr,
 \end{equation}
  where we have defined
 \begin{equation}
 \label{def_tuh}
 \vh^i := \left\{
 \begin{array}{ll}
 \fiint_{2Q_i} \vh(z)\lsb{\chi}{[\mft-s,\mft+s]} \ dz & \text{if} \ \ 2Q_i \subset \Om_{4\rho}^{\al}(\mfx) \times (\mft-s,\infty), \\
 0 & \text{else}. 
 \end{array}\right.
 \end{equation}

  From construction in \eqref{lipschitz_function} and \eqref{def_tuh}, we see that 
 \begin{equation*}
 \spt(\vlh) \subset \Om_{4\rho}^{\al}(\mfx) \times (\mft-s,\infty).
 \end{equation*}

 We see that $\vlh$ has the right support for the test function and hence the rest of this section will be devoted to proving the Lipschitz regularity of $\vlh$ on $K_{4\rho}^{\al}(\mfx)$ as well as some useful estimates.

 \subsection{Some estimates on the test function}
 
 In this subsection, we will collect some useful estimates on the test function. The proofs of these estimates follow similarly to those in \cite{adimurthi2018sharp} and hence we will only provide an outline of the proofs.
 
 \begin{lemma}
\label{lemma3.6_pre}
Let $ \mfz \in K_{4\rho}^{\al}(\mfz) \setminus \elam$, then from \descref{W1}{W1}, we have that $\mfz \in 2Q_i$ for some $i \in \mci_{\mfz}$. For any $1 \leq \tht \leq \frac{p^-}{\mathfrak{q}}$,  there holds
\begin{gather}
|\vh^i|^{\tht} \leq  \fiint_{2Q_i} |\vh(\tz)|^{\tht}\lsb{\chi}{[\mft-s,\mft+s]} \ d\tz  \apprle_{(\plog,n)} (\scalex{\al}{\mfz}\rho)^{\tht} \la^{\frac{\tht}{p(z_i)}},  \label{lemma3.6_pre_one}\\
\fiint_{2Q_i}|\nabla \vh(\tz)|^{\tht}\lsb{\chi}{[\mft-s,\mft+s]} \ d\tz \apprle_{(\plog,n)}  \la^{\frac{\tht}{p(z_i)}}  \label{lemma3.6_pre_two}.
\end{gather}
\end{lemma}
\begin{proof}
\begin{description}[leftmargin=*]
\item[Proof of \eqref{lemma3.6_pre_one}:] We prove this estimate as follows:
\begin{equation*}
|\vh^i|^{\tht}  
\apprle (\scalex{\al}{\mfz}\rho)^{\tht} \lbr \fiint_{8\hat{c}Q_i} \left[ 1 + \abs{\frac{v(z)}{\scalex{\al}{\mfz}\rho}} \right]^{\frac{\pp}{\mathfrak{q}}}\ d\tz\rbr^{\frac{\tht\mathfrak{\mathfrak{q}}}{p^-_{2Q_i}}} 
 \overset{\eqref{elambda}}{\apprle} (\scalex{\al}{\mfz}\rho)^{\tht} \la^{\frac{\tht}{p^-_{2Q_i}}} 
 \overset{\eqref{2.2.28-1}}{\apprle} (\scalex{\al}{\mfz}\rho)^{\tht} \la^{\frac{\tht}{p(z_i)}}.
\end{equation*}

\item[Proof of \eqref{lemma3.6_pre_two}:] From \eqref{elambda}, we see that 
\begin{equation*}\label{lemma3.6_pre_1}
\fiint_{2Q_i} |\nabla \vh|^{\tht}\lsb{\chi}{[\mft-s,\mft+s]} \ d\tz 
 \apprle   \lbr \fiint_{8\hat{c}Q_i} \left[|\nabla v| +1\right]^{\frac{\pp}{\mathfrak{q}}}\ d\tz  \rbr^{\frac{\tht\mathfrak{\mathfrak{q}}}{p^-_{2Q_i}}}
 \overset{\eqref{elambda}}{\apprle}  \la^{\frac{\tht}{p^-_{2Q_i}}} 
 \overset{\eqref{2.2.28-1}}{\apprle}  \la^{\frac{\tht}{p(z_i)}}.
\end{equation*}
\end{description}
\end{proof}

 \begin{corollary}
 \label{lemma3.6}
For any $z \in K_{4\rho}^{\al}(\mfz) \setminus \elam$, we have $z \in 2Q_i$ for some $i \in \mci_{z}$, then there holds
 \begin{equation*}
 |\vh(z)| \apprle_{(n,\plog,\lamot)}(\scalex{\al}{\mfz}\rho)  \la^{\frac{1}{p(z_i)}},
 \end{equation*}
where $z_i$ is the centre of $Q_i$. 
 \end{corollary}
 

\begin{lemma}
\label{improved_est}
Let $2Q_i$ be a parabolic Whitney type cylinder, then for any $1 \leq \tht \leq \frac{p^-}{\mathfrak{q}}$, there holds
\begin{equation*}
\fiint_{2Q_i} |\vh(z)\lsb{\chi}{[\mft-s,\mft+s]}- \vh^i|^{\tht} \  dz \apprle_{(\plog,\lamot,n)} \min\left\{ \scalex{\al}{\mfz}\rho, \scalex{\la}{z_i}r_i\right\}^{\tht} \la^{\frac{\tht}{p(z_i)}}.
\end{equation*}
\end{lemma}
\begin{proof}
Let us consider the following two cases:
\begin{description}[leftmargin=*]
\item[Case $\scalex{\al}{\mfz}\rho \leq  \scalex{\la}{z_i}r_i$:] In this case, we can use triangle inequality along with \eqref{lemma3.6_pre_one} to get
\begin{equation}
\label{6.18}
\fiint_{2Q_i} |\vh(\tz)\lsb{\chi}{[\mft-s,\mft+s]}- \vh^i|^{\tht} \  d\tz \apprle 2 \fiint_{2Q_i} |\vh(\tz)|^{\tht}\lsb{\chi}{[\mft-s,\mft+s]} \ d\tz \overset{\eqref{lemma3.6_pre_one}}{\apprle} (\scalex{\al}{\mfz}\rho)^{\tht} \la^{\frac{\tht}{p(z_i)}}.
\end{equation}

\item[Case $\scalex{\al}{\mfz}\rho \geq  \scalex{\la}{z_i}r_i$:] Applying Lemma \ref{lemma_crucial_2} with  $\mu \in C_c^{\infty}(2B_i)$ such that $|\mu(x)| \apprle \frac{1}{\lbr \scalex{\la}{z_i} r_i\rbr^n}$ and $|\nabla \mu(x)| \apprle \frac{1}{\lbr \scalex{\la}{z_i} r_i\rbr^{n+1}}$, we get
\begin{equation}
\label{6.19}
\begin{aligned}
\fiint_{2Q_i} |\vh(z)\lsb{\chi}{[\mft-s,\mft+s]}- \vh^i|^{\tht} \  dz &\leq \lbr \scalex{\la}{z_i} r_i\rbr^{\tht} \fiint_{2Q_i} |\nabla \vh|^{\tht}\lsb{\chi}{[\mft-s,\mft+s]} \ d\tz\\
 &\qquad + \sup_{t_1,t_2 \in 2I_i\cap [\mft-s,\mft+s]} |\avgs{\vh}{\mu}(t_2) - \avgs{\vh}{\mu}(t_1)|^{\tht}.\\
\end{aligned}
\end{equation}
The first term on the right of \eqref{6.19} can be estimated using \eqref{lemma3.6_pre_two} to get
\begin{equation}
\label{est_J_1}
\lbr \scalex{\la}{z_i} r_i\rbr^{\tht} \fiint_{2Q_i} |\nabla \vh|^{\tht}\lsb{\chi}{[\mft-s,\mft+s]} \ d\tz \apprle \lbr \scalex{\la}{z_i} r_i\rbr^{\tht}    \la^{\frac{\tht}{p(z_i)}}.
\end{equation}

To estimate the second term on the right of \eqref{6.19}, we make use of Lemma \ref{lemma_crucial_2} with $\phi(x) = \mu(x)$ and $\varphi(t) \equiv 1$, we get
\begin{equation}
\label{est_J_2_one}
\begin{array}{rcl}
|\avgs{\vh}{\mu}(t_2) - \avgs{\vh}{\mu}(t_1)| & \apprle & \frac{|2Q_i|}{\lbr \scalex{\la}{z_i} r_i\rbr^{n+1}} \fiint_{2Q_i} |\aa(\tz,\nabla u) - \aa(\tz,\nabla w)| + |\bff|^{p(\tz)-1} \ d\tz \\
& \apprle & \frac{\scalet{\la}{z_i}r_i^2}{\scalex{\la}{z_i} r_i} \lbr \fiint_{8\hat{c}Q_i} (1 + |\nabla u|+|\nabla w| + |\bff|)^{\frac{p(\tz)}{\mathfrak{q}}} \ d\tz \rbr^{\frac{q(p^+_{2Q_i} -1)}{p^-_{2Q_i}}}\\
& \overset{\eqref{elambda}}{\apprle} & \la^{-1 + \frac{1}{p(z_i)} + \frac{d}{2}} r_i \la^{\frac{p^+_{2Q_i} -1}{p^-_{2Q_i}}}.
\end{array}
\end{equation}
Now making use of \eqref{2.2.28-1} along with the fact that $\la \geq 1$ and $p^-_{2Q_i} \leq p(z_i)$, we get
\begin{equation}
\label{6.22}
\la^{-1+ \frac{1}{p(z_i)} + \frac{p^+_{2Q_i}}{p^-_{2Q_i}} - \frac{1}{p^-_{2Q_i}}} = \la^{\frac{p^+_{2Q_i}-p^-_{2Q_i}}{p^-_{2Q_i}}} \la^{\frac{p^-_{2Q_i}-p(z_i)}{p(z_i)p^-_{2Q_i}}} \leq \la^{\frac{p^+_{2Q_i}-p^-_{2Q_i}}{p^-_{2Q_i}}} \overset{\eqref{2.2.28-1}}{\apprle} C_{(\plog,n)}. 
\end{equation}
Substituting \eqref{6.22} into \eqref{est_J_2_one}, we get
\begin{equation}
\label{est_J_2}
|\avgs{\vh}{\mu}(t_2) - \avgs{\vh}{\mu}(t_1)| \apprle \lbr \scalex{\la}{z_i} r_i \rbr \la^{\frac{1}{p(z_i)}}.
\end{equation}

Thus combining \eqref{est_J_1} and \eqref{est_J_2} into \eqref{6.19}, we get
\[
\fiint_{2Q_i} |\vh(\tz)\lsb{\chi}{[\mft-s,\mft+s]}- \vh^i|^{\tht} \  d\tz \apprle_{(\plog,\lamot,n)} \lbr \scalex{\la}{z_i} r_i\rbr^{\tht}    \la^{\frac{\tht}{p(z_i)}}.
\]
which proves the lemma. 
\end{description}
\end{proof}

\begin{corollary}
\label{corollary3.7}
For any $i \in \NN$ and any $j \in \mci_i$, there holds
\[
|\vh^i - \vh^j| \apprle_{(\plog,\lamot,n)} \min\left\{ \scalex{\al}{\mfz}\rho, \scalex{\la}{z_i}r_i\right\} \la^{\frac{1}{p(z_i)}}.
\]
\end{corollary}

 \subsection{Bounds on \texorpdfstring{$\vlh$ and $\nabla \vlh$}.}
 
 \begin{lemma}
\label{lemma6.7-1}
Let $Q_i$ be a parabolic Whitney type cylinder. Then for any $z \in 2Q_i$, we have the following bound:
\begin{equation}
\label{lemma6.7-1_est}
\lbr\frac{1}{\scalex{\al}{\mfz}\rho} |\vlh(z)| + |\nabla \vlh(z)|\rbr \lsb{\chi}{[\mft-s,\mft+s]} \apprle_{(\plog,\lamot,n)}  \la^{\frac{1}{p(z_i)}}.
\end{equation}
\end{lemma}
%
%

 \begin{corollary}
 \label{corollary6.7-2}
Let $z \in K_{4\rho}^{\al}(\mfz) \setminus \elam$, then $z \in 2Q_i$ for some $i \in \NN$. Then there holds for any $\de \in (0,1]$, the estimates
 \begin{gather}
 \frac{1}{\scalex{\la}{z_i} r_i } |\vlh(z)| \apprle_{{(\plog,\lamot,n)}}\frac{\la^{\frac{1}{p(z_i)}}}{\de} +  \frac{\de}{\lbr \scalex{\la}{z_i}r_i\rbr^2 \la^{\frac{1}{p(z_i)}}} |\vh^i|^2, \label{bound+6.31}\\
  |\nabla \vlh(z)| \apprle_{{(\plog,\lamot,n)}}\frac{\la^{\frac{1}{p(z_i)}}}{\de}. \label{bound+6.31_two}
 \end{gather}

 \end{corollary}

%

 \begin{lemma}
 \label{lemma6.7-3}
 Let $z \in K_{4\rho}^{\al}(\mfz) \setminus \elam$, then $z \in 2Q_i$ for some $i \in \NN$. Then there holds for any $\de \in (0,1]$, the estimates
 \begin{gather}
 |\vlh(z)| \apprle_{(\plog,\lamot,n)} \frac{\scalex{\la}{z_i} r_i\la^{\frac{1}{p(z_i)}}}{\de} + \frac{\de}{\scalex{\la}{z_i} r_i\la^{\frac{1}{p(z_i)}}} \fiint_{\htq_i} |\vh(\tz)|^2 \ d\tz, \label{bound_6.7-3-1} \\
 |\nabla \vlh(z)| \apprle_{(\plog,\lamot,n)} \la^{\frac{1}{p(z_i)}} + \frac{\de}{\lbr \scalex{\la}{z_i} r_i\rbr^2\la^{\frac{1}{p(z_i)}}} \fiint_{\htq_i} |\vh(\tz)|^2 \ d\tz. \nonumber
 \end{gather}
 \end{lemma}

\subsection{Estimates on the time derivative of \texorpdfstring{$\vlh$}.}

 \begin{lemma}
 \label{time_vlh}
 Let $ z \in {K_{4\rho}^{\al}(\mfz)}$, then $z \in 2Q_i$ for some $i \in \NN$. We then have the following estimates for the time derivative of $\vlh$: 
 \begin{equation}
 \label{bound_time_vlh_one}
 |\pa_t \vlh(\tz)| \apprle_{(\plog,\lamot,n)}  \frac{1}{ \scalet{\la}{z_i} r_i^2} \fiint_{Q_i} |\vh(z)| \lsb{\chi}{[-s-s]} \ dz.
 \end{equation}
We also have the improved estimate
\begin{equation}
\label{bound_time_vlh_two}
 |\pa_t \vlh(\tz)| \apprle_{(\plog,\lamot,n)} \frac{1}{\scalet{\la}{z_i} r_i^2} \la^{\frac{1}{p(z_i)}} \min \left\{\scalex{\la}{z_i}r_i,  \scalex{\al}{\mfz}\rho \right\}.
\end{equation}

 \end{lemma}
 \begin{proof}
 Let us prove each of the assertions as follows:
 \begin{description}[leftmargin=*]
 
 \item[Estimate \eqref{bound_time_vlh_one}:] In this case, we proceed as follows
 \begin{equation*}
 \begin{array}{rcl}
 |\pa_t \vlh(z)| & \leq & \sum_{j \in I_i} |\vh^j| |\pa_t\psi_j(z)| \overset{\eqref{lemma3.6_pre_one}}{\apprle}   \frac{1}{\scalet{\la}{z_i} r_i^2} \fiint_{Q_i} |\vh(z)| \lsb{\chi}{[\mft-s,\mft+s]} \ dz.
 \end{array}
 \end{equation*}

 \item[Estimate \eqref{bound_time_vlh_two}:] From the fact that $\sum_{j \in I_i} \psi_j(z) = 1$, we see that $\sum_{j \in I_i} \pa_t \psi_j(z) = 0$ which along with Lemma \ref{partition_unity} gives the following sequence of estimates
 \begin{equation*}
 \begin{array}{rcl}
 |\pa_t\vlh(z)| & = & \abs{\sum_{j \in I_i} \lbr \vh^j - \vh^i \rbr \pa_t \psi_j(z)  } \\
 & \overset{\text{Corollary \ref{corollary3.7}}}{\apprle} & \frac{1}{\scalet{\la}{\mfz} r_i^2} \min\left\{\scalex{\al}{\mfz} \rho, \scalex{\la}{z_i}r_i\right\} \la^{\frac{1}{p(z_i)}}.
 \end{array}
 \end{equation*}

 \end{description}

 \end{proof}

 \subsection{Some important estimates for the test function}
 
 \begin{lemma}
 \label{lemma6.8}
 Let $Q_i$ be a Whitney-type parabolic cylinder for some $ i \in \NN$. Then for any $\vt \in [1,2]$, there holds
 \begin{equation*}
 \label{lemma6.8-one}
 \iint_{{K_{4\rho}^{\al}(\mfz)} \setminus \elam} |\vlh(z)|^{\vt} \ dz \apprle_{(\plog,\lamot,n)} \iint_{{K_{4\rho}^{\al}(\mfz)} \setminus \elam} |\vh(z)|^{\vt} \ dz.
 \end{equation*}
\end{lemma}
%

 \begin{lemma}
 \label{lemma6.8-1}
 Let $Q_i$ be a Whitney-type parabolic cylinder for some $i \in \NN$, then there holds
 \begin{equation*}
 \label{lemma6.8-two}
 \fiint_{Q_i} |\vlh(z) - uh(z)| \ dz \apprle_{(\plog,\lamot,n)} \min\left\{ \scalex{\la}{z_i} r_i, \scalex{\al}{\mfz}\rho \right\} \la^{\frac{1}{p(z_i)}}.
 \end{equation*}

 \end{lemma}

 \begin{lemma}
 \label{lemma6.8-2}
 Let $Q_i$ be a Whitney-type parabolic cylinder for some $i \in \NN$, then there holds
 \begin{equation*}
 \label{lemma6.8-three}
 \begin{array}{rl}
 \iint_{K_{4\rho}^{\al}(\mfz) \setminus \elam } |\pa_t \vlh(z) \lbr \vlh(z) - \vh(z)\rbr|^{\vt} \ dz & \apprle_{(\plog,\lamot,n)} \la^{\vt} |\RR^{n+1} \setminus \elam|.
 \end{array}
 \end{equation*}
 \end{lemma}

\begin{proof}
From \descref{W2}{W2}, we see that $K_{4\rho}^{\al}(\mfz) \setminus \elam \subset \bigcup 2Q_i$, thus for a given $i \in \NN$, let use define the following
\begin{equation*}
\label{lm6.8-3-1}
\begin{array}{rcl}
 J_i:=\iint_{2Q_i } \abs{\pa_t \vlh(z) \lbr \vlh(z) - \vh(z)\rbr}^{\vt} \lsb{\chi}{K_{4\rho}^{\al}(\mfz)}\ dz.
\end{array}
\end{equation*}

Making use of \eqref{bound_time_vlh_two}, we  get
\begin{equation*}
\begin{array}{rcl}
J_i & \apprle & \lbr \frac{\la^{\frac{1}{p(z_i)}}}{\scalet{\la}{z_i} r_i^2} \min\{ \scalex{\al}{\mfz}\rho, \scalex{\la}{z_i} r_i\} \rbr^{\vt} \iint_{2Q_i} \abs{\vlh(z)\lsb{\chi}{K_{4\rho}^{\al}(\mfz)} - \vh(z) \lsb{\chi}{K_{4\rho}^{\al}(\mfz)}}^{\vt} \ dz \\
& \apprle & \lbr \frac{\la^{\frac{1}{p(z_i)}}}{\scalet{\la}{z_i} r_i^2} \min\{ \scalex{\al}{\mfz}\rho, \scalex{\la}{z_i} r_i\} \rbr^{\vt} \sum_{j \in I_i}\iint_{2Q_i} \abs{ \vh(z)\lsb{\chi}{K_{4\rho}^{\al}(\mfz)} - \vh^j }^{\vt} \ dz \\
& \overset{\text{Lemma \ref{improved_est}}}{\apprle} &  \lbr \frac{\la^{\frac{1}{p(z_i)}}}{\scalet{\la}{z_i} r_i^2} \min\{ \scalex{\al}{\mfz}\rho, \scalex{\la}{z_i} r_i\} \scalex{\la}{z_i} r_i \la^{\frac{1}{p(z_i)}}\rbr^{\vt} |\hQ_i| = |\htq_i| \la^{\vt}.
\end{array}
\end{equation*}
Summing over all $i \in \NN$, we get the desired inequality.
\end{proof}

 \subsection{Lipschitz continuity estimates}

 We will now show that the function $\vl$  constructed in \eqref{lipschitz_function} is Lipschitz continuous on $B_{4\rho}^{\al}(\mfx) \times (\mft-s,\mft+s)$ where $s$ is as defined in \eqref{def_s}. To do this, we shall use the integral characterization of Lipschitz continuous functions obtained in \cite[Theorem 3.1]{Prato} which says the following:
 \begin{lemma}[Lipschitz characterization]
 \label{deprato}
 Let $\tz \in B_{4\rho}^{\al}(\mfx) \times (\mft-s,\mft+s)$ and $r >0$ be given. Define the parabolic cylinder $Q_r(\tz) := B_r(\tx) \times (\tlt - r^2, \tlt+r^2)$, i.e., $Q_r(\tz) := \{ z \in \RR^{n+1}: d_p(z,\tz) \leq r\}$ where $d_p$ is as defined in Definition \ref{parabolic_metric}. Furthermore suppose that the  following expression is bounded independent of $\tz \in B_{4\rho}^{\al}(\mfx) \times (\mft-s,\mft+s)$ and $r>0$
 \[
 I_r(\tz) := \frac{1}{\abs{B_{4\rho}^{\al}(\mfx) \times (\mft-s,\mft+s) \cap Q_r(\tz)}} \iint_{B_{4\rho}^{\al}(\mfx) \times (\mft-s,\mft+s) \cap Q_r(\tz)} \abs{\frac{\vlh(z) - \avgs{\vlh}{B_{4\rho}^{\al}(\mfx) \times (\mft-s,\mft+s)\cap Q_r(\tz)}}{r}} \ dz < \infty,
 \]
then $\vl \in C^{0,1}(B_{4\rho}^{\al}(\mfx) \times (\mft-s,\mft+s))$.
 \end{lemma}

 \begin{remark}
 From \eqref{def_d} and the fact that $\al \geq 1$,  for any $\tz_1,\tz_2 \in \RR^{n+1}$ and any $\tz \in \RR^{n+1}$, we get
 \begin{equation}
 \label{equiv_dist}
 \begin{array}{rcl}
 d_p(\tz_1,\tz_2) & \overset{\text{Definition} \ \ref{parabolic_metric}}{:=} & \max\lbr[\{]|x_1-x_2|, \sqrt{|t_1-t_2|} \rbr[\}] \\
 & \leq&  \max\left\{\nscalex{\al}{z}|x_1-x_2|, \sqrt{\nscalet{\al}{z} |t_1-t_2|} \right\}\overset{\text{Definition} \ \ref{loc_parabolic_metric}}{ =:} d_{\tz}(\tz_1,\tz_2)\\
 & \leq &  \al^{{\frac{1}{p^-}-\frac{d}{2}}}\al^{\frac32-\frac{d}{2}}\max\left\{|x_1-x_2|, \sqrt{|t_1-t_2|} \right\} \leq C_{(\al,p^-,d)} d_p(\tz_1,\tz_2).
 \end{array}
 \end{equation}
 
 This shows that for any $\tz \in \RR^{n+1}$, we have $d_p \approx_{(\al,p^-,d)} d_{\tz}$. 
 \end{remark}

 In this subsection, we want to apply Lemma \ref{deprato}, hence we only need to ensure the constants involved are independent of $r>0$ and $\tz$ only. \emph{Only for this subsection, we will use the notation $o(1)$ to denote a constant which  can depend on $\al,\al_0,\plog,\lamot,n,\|uh\|_{L^1}, \|u\|_{L^1}$ but {\bf NOT} on  $r>0$ and the point $\tz$.}

 \begin{lemma}
 \label{lipschitz_continuity}
 Let $\al \geq 1$, then for any $\tz \in K_{4\rho}^{\al}(\mfz)$ and $r>0$, there exists a constant $C>0$ independent of $\tz$ and $r$ such that
\begin{equation*}\label{da_pr}
 I_r(\tz) := \frac{1}{\abs{K_{4\rho}^{\al}(\mfz)  \cap Q_r(\tz)}} \iint_{K_{4\rho}^{\al}(\mfz)  \cap Q_r(\tz)} \abs{\frac{\vlh(z) - \avgs{\vlh}{K_{4\rho}^{\al}(\mfz) \cap Q_r(\tz)}}{r}} \ dz \leq C < \infty.
 \end{equation*}
 In particular, this implies for any $\tz_1, \tz_2 \in B_{4\rho}^{\al}(\mfx) \times (\mft-s,\mft+s)$, there exists a constant $K>0$ such that
 \[
 |\vlh(\tz_1) - \vlh(\tz_2)| \leq K d_p(\tz_1,\tz_2).
 \]
 \end{lemma}
\begin{proof}
Let $r>0$ and $\tz \in K_{4\rho}^{\al}(\mfz) $ and denote the cylinder $Q_r(\tz) = Q$. We will now proceed as follows:

\begin{description}[leftmargin=*]
\item[Case $2Q \subset \elam^c$:] From \eqref{lipschitz_function}, it is easy to see that $\vlh \in C^{\infty}(\elam^c)$. Thus, we can apply the mean value theorem to get
\begin{equation}\label{A.47}\begin{array}{rcl}
I_r(\tz) &\apprle& \frac{1}{r} \fiint_{Q \cap \lbr \RR^n \times {[\mft-s,\mft+s]}\rbr} \fiint_{Q \cap \lbr \RR^n \times {[\mft-s,\mft+s]}\rbr}  \abs{\vlh(z_1) - \vlh(z_2)} \ dz_1 \ dz_2 \\
&\apprle &\sup_{z\in Q \cap \lbr \RR^n \times {[\mft-s,\mft+s]}\rbr} \lbr |\nabla \vlh(z)| + r |\pa_t \vlh(z)| \rbr.
\end{array}
\end{equation}
Since $2Q \subset \elam^c$, we can use \eqref{bound+6.31_two} with $\de =1$ and \eqref{bound_time_vlh_two} to bound \eqref{A.47} as follows:
\begin{equation}\label{A.48}
I_r(\tz) \apprle \sup_{z \in Q \cap \lbr \RR^n \times {[\mft-s,\mft+s]}\rbr} \lbr \la^{\frac{1}{p(z)}}  + r \frac{\la^{\frac{d}{2}r_i}}{\scalet{\la}{z} r_i^2} \rbr.
\end{equation}
Here we recall that $z \in 2Q_i$ for some $i \in \NN$ and $r_i$ is the radius of the cylinder $Q_i$.

Since $Q \in \elam^c$, we also have that $z \in 2Q_i$ for some $i \in \NN$. Let $z_i$ be the centre of $Q_i$, then we have
\begin{equation}\label{A.49}
r \leq d_p(z,\elam) \leq d_p(z,z_i) + d_p(z_i,\elam) \leq r_i + d_{z_i}(z_i,\elam) \overset{\eqref{equiv_dist}}{\leq} r_i + \hat{c} r_i = (1+\hat{c}) r_i.
\end{equation}

Substituting \eqref{A.49} into \eqref{A.48}, we get 
\[
I_r(\tz) \apprle \la^{\frac{1}{p^-}} + (1+ \hat{c}) \la^{1- \frac{d}{2}} = o(1).
\]

\item[Case $2Q \nsubseteq \elam^c$:] In this case, we split the proof into three subcases as follows:
\begin{description}
\item[Subcase $2Q \subset \RR^n \times {(-\infty,s]}$ or $2Q \subset \RR^n \times {[-s,\infty)}$:] 

In this situation, it is easy to see that the following holds:
\begin{equation}
  \label{A.70}
  \abs{Q \cap \lbr \RR^n \times [\mft-s,\mft+s]\rbr} \apprge |Q|. 
 \end{equation}
 We apply   triangle inequality and estimate $I_r(\tz)$ by 
 \begin{equation*}
  \label{3.81}
  \begin{array}{rcl}
   I_r(\tz) 
   & \leq &  2J_1 + J_2,
  \end{array}
 \end{equation*}
 where we have set
 \begin{equation}\label{def_J_1_2}\begin{array}{c}
  J_1:= \fiint_{Q \cap \lbr \RR^n \times {[\mft-s,\mft+s]}\rbr} \left| \frac{\vlh(z) - \vh(z)}{r}\right| \ dz, \\
  J_2 := \fiint_{Q \cap \lbr \RR^n \times {[\mft-s,\mft+s]}\rbr} \left| \frac{\vh(z) - \avgs{\vh}{Q\cap \lbr \RR^n \times {[\mft-s,\mft+s]}\rbr}}{r}\right|\ dz.
 \end{array}\end{equation}
We now estimate each of the terms of \eqref{def_J_1_2} as follows:
\begin{description}
\item[Estimate for $J_1$:]  From \eqref{lipschitz_function}, we get
\begin{equation}
\label{3.82}
\begin{array}{ll}
J_1 &\apprle \sum_{i\in \NN} \frac{1}{|Q\cap \lbr\RR^n \times [\mft-s,\mft+s]\rbr|} \iint_{Q \cap\lbr \RR^n \times {[\mft-s,\mft+s]}\rbr\cap 2Q_i} \left| \frac{\vh(z) - \vh^i}{r}\right| \ dz \\
& = \sum_{i\in \NN} \frac{1}{|Q\cap\lbr \RR^n \times [\mft-s,\mft+s]\rbr|} \iint_{Q \cap\lbr \RR^n \times {[\mft-s,\mft+s]}\rbr\cap 2Q_i} \left| \frac{\vh(z)\lsb{\chi}{[\mft-s,\mft+s]} - \vh^i}{r}\right| \ dz.
\end{array}
\end{equation}
Let us fix an $i \in \NN$ and take two points $z_1 \in Q \cap 2Q_i$ and $z_2 \in \elam \cap 2Q$. Making use of  \descref{W5}{W5} along with the trivial bound $d_p (z_1, z_2) \leq  4r$ and $d_p (z_i, z_1) \leq 2r_i$,  we get
\begin{equation}
\label{3.83.1}
\hat{c}r_i =d_p(z_i,\elam) \leq d_p (z_i, z_1) + d_p(z_1, z_2) \leq 2r_i + 4r  \ \Longrightarrow \ r_i \apprle_{(\hat{c})} r,
\end{equation}
where $z_i$ denotes the centre of $Q_i$ as in \descref{W1}{W2} and $\hat{c}$ is from \descref{W4}{W4}.

Note that \eqref{A.70} holds and thus summing over all $i \in \NN$ such that  $Q \cap\lbr \RR^n \times {[\mft-s,\mft+s]}\rbr\cap 2Q_i \neq \emptyset$ in \eqref{3.82} and making use of  \eqref{3.83.1}, we get
\begin{equation*}
 \label{3.84.1}
 \begin{array}{rcl}
 J_1 &\apprle& \sum_{\substack{i\in\NN \\ Q \cap\lbr \RR^n \times {[\mft-s,\mft+s]}\rbr\cap 2Q_i \neq \emptyset }} \frac{|2Q_i|}{|Q\cap\lbr \RR^n \times {[\mft-s,\mft+s]}\rbr|} \fiint_{2Q_i} \left| \frac{\vh(z)\lsb{\chi}{[\mft-s,\mft+s]} - \vh^i}{r}\right| \ dz \\
 &\overset{\eqref{A.70},\eqref{3.83.1}}{\apprle} & \sum_{i\in \NN}  \fiint_{2Q_i} \left| \frac{\vh(z)\lsb{\chi}{[\mft-s,\mft+s]} - \vh^i}{r_i}\right| \ dz.
 \end{array}
\end{equation*}

Using Lemma \ref{improved_est}, we get
\begin{equation*}
\label{bound_J_1}
J_1 \apprle o(1).
\end{equation*}

\item[Estimate for $J_2$:] To estimate this term, we proceed as follows: Note that $Q \cap \lbr \RR^n \times {[\mft-s,\mft+s]}\rbr$ is another cylinder. If $Q \subset B_{4\rho}^{\al}(\mfx) \times \RR$, then choose a cut-off function $\mu \in C_c^{\infty}(B)$ with $|\nabla \mu| \leq \frac{C_{(n)}}{r^{n+1}}$ to get
\begin{equation*}
\begin{array}{ll}
J_2 
& = \fiint_{Q \cap \lbr \RR^n \times {[\mft-s,\mft+s]}\rbr} \left| \frac{\vh(z)\lsb{\chi}{[\mft-s,\mft+s]} - \avgs{\vh\lsb{\chi}{[\mft-s,\mft+s]}}{Q\cap \lbr \RR^n \times {[\mft-s,\mft+s]}\rbr}}{r}\right|\ dz \\
& \apprle  \fiint_{Q \cap \lbr \RR^n \times {[\mft-s,\mft+s]}\rbr} |\nabla \vh| \lsb{\chi}{[\mft-s,\mft+s]} + \sup_{t_1,t_2 \in [\mft-s,\mft+s] \cap Q} \left| \frac{\avgs{\vh\lsb{\chi}{[\mft-s,\mft+s]}}{\mu}(t_1) - \avgs{\vh\lsb{\chi}{[\mft-s,\mft+s]}}{\mu}(t_1)}{r} \right|.
\end{array}
\end{equation*}
Recall that we are in the case $2Q \cap \elam \neq \emptyset$ and $2Q \cap \elam^c \neq \emptyset$. Further applying Lemma \ref{lemma_crucial_2} and proceeding similarly to \eqref{est_J_2_one}, we see that 
\begin{equation*}
\label{J_2_bound_first_case}
J_2 \apprle o(1).
\end{equation*}

On the other hand, if $Q \nsubseteq B_{4\rho}^{\al}(\mfx) \times \RR$, then we can apply Poincar\'e's inequality directly to get
\begin{equation*}
\begin{array}{ll}
J_2 
& \apprle \fiint_{Q \cap \lbr \RR^n \times {[\mft-s,\mft+s]}\rbr} \left| \frac{\vh(z)\lsb{\chi}{[\mft-s,\mft+s]}}{r}\right|\ dz 
 \apprle  \fiint_{Q \cap \lbr \RR^n \times {[\mft-s,\mft+s]}\rbr} \left| \nabla \vh(z)\lsb{\chi}{[\mft-s,\mft+s]}\right|\ dz.
\end{array}
\end{equation*}
Recall that we are in the case $2Q \cap \elam \neq \emptyset$ and $2Q \cap \elam^c \neq \emptyset$. Using \eqref{A.70}, we thus get
\begin{equation}
\label{bound_J_2_second_case}
J_2 \apprle o(1).
\end{equation}

\end{description}
\item[Subcase $2Q \cap \RR^n \times (-\infty, s) \neq \emptyset$ and $2Q \cap \RR^n \times (-s,\infty)\neq \emptyset$ AND $r^2 \leq s$:] In this case, we see that $$|Q \cap \lbr \RR^n \times [\mft-s,\mft+s]\rbr| = |B| \times s.$$
 We apply   triangle inequality and estimate $I_r(z)$ by 
 \begin{equation*}
  \label{3.81_n}
  \begin{array}{ll}
   I_r(z) 
   & \leq 2J_1 + J_2,
  \end{array}
 \end{equation*}
 where we have set
 \begin{equation*}\label{def_J_1_2_n}\begin{array}{c}
  J_1:= \fiint_{Q \cap \lbr \RR^n \times [\mft-s,\mft+s]\rbr} \left| \frac{\vlh(z) - \vh(z)}{r}\right| \ dz, \\
  J_2 := \fiint_{Q \cap \lbr \RR^n \times [\mft-s,\mft+s]\rbr} \left| \frac{\vh(z) - \avgs{\vh}{Q\cap \lbr \RR^n \times [\mft-s,\mft+s]\rbr}}{r}\right|\ dz.
 \end{array}\end{equation*}
 
 Proceeding as before, we get
 \begin{equation*}
 \label{3.82.n}
 \begin{array}{rcl}
 J_1 & \apprle & \sum_{i \in \NN} \frac{|2Q_i|}{|Q \cap \lbr \RR^n \times [\mft-s,\mft+s]\rbr|} \fiint_{2Q_i} \left| \frac{\vh(z)\lsb{\chi}{[\mft-s,\mft+s]} - \vh^i}{r}\right| \ dz \\
 & \overset{\eqref{3.83.1}}{\apprle}& \frac{r_i^{n+2} \scalet{\la}{z_i}\scalexn{\la}{z_i}}{r^{n} s} \sum_{i\in \NN}  \fiint_{2Q_i} \left| \frac{\vh(z)\lsb{\chi}{[\mft-s,\mft+s]} - \vh^i}{r_i}\right| \ dz\\
 & \overset{\eqref{3.83.1}}{\apprle}& \frac{r^{n+2} \scalet{\la}{z_i}\scalexn{\la}{z_i} }{r^{n} s} \sum_{i\in \NN}  \fiint_{2Q_i} \left| \frac{\vh(z)\lsb{\chi}{[\mft-s,\mft+s]} - \vh^i}{r_i}\right| \ dz\\
 & \overset{\text{Lemma \ref{improved_est}}}{\apprle} & o(1).
 \end{array}
 \end{equation*}
To obtain the last inequality, we made use of the bound $r^2 \leq s$.

 The estimate for $J_2$ is exactly as in \eqref{bound_J_2_second_case} to get 
 \begin{equation*}
 \label{bound_J_2_third_case}
 J_2 \apprle o(1). 
 \end{equation*}
\item[Subcase $2Q \cap \RR^n \times (-\infty, s) \neq \emptyset$ and $2Q \cap \RR^n \times (-s,\infty)\neq \emptyset$ AND $r^2 \geq s$:] In this case, we proceed as follows. Using triangle inequality and the bound $|Q \cap \lbr \RR^n \times [\mft-s,\mft+s]\rbr| = |B| \times s$ where $s$ is from \eqref{def_s}, we get
 \begin{equation*}
  \label{3.91}
  \begin{array}{l}
  \fiint_{Q \cap \lbr \RR^n \times [\mft-s,\mft+s]\rbr}  \left| \frac{\vlh(z) - \avgs{\vlh}{Q \cap \lbr \RR^n \times [\mft-s,\mft+s]\rbr}}{r} \right|\ dz \\ 
  \hspace*{4cm} \apprle \frac{1}{|Q \cap \lbr \RR^n \times [\mft-s,\mft+s]\rbr|} \iint_{Q \cap \lbr \RR^n \times [\mft-s,\mft+s]\rbr \cap \elam} |\vlh(z)| \ dz \\
  \hspace*{4cm} \qquad + \frac{1}{|Q \cap \lbr \RR^n \times [\mft-s,\mft+s]\rbr|} \iint_{Q \cap \lbr \RR^n \times [\mft-s,\mft+s]\rbr \setminus \elam} |\vlh(z)| \ dz.
  \end{array}
 \end{equation*}

 By construction of $\vlh$ in \eqref{lipschitz_function}, we have $\vlh = \vh$ on $\elam$.  On $\lbr \RR^n \times [\mft-s,\mft+s]\rbr \setminus \elam$, we can apply Corollary  \ref{lemma3.6} to obtain the following bound:
 \begin{equation*}
  \begin{array}{ll}
   \fiint_{Q \cap \lbr \RR^n \times [\mft-s,\mft+s]\rbr}  \left| \frac{\vlh(z) - \avgsnoleft{Q \cap \lbr \RR^n \times [\mft-s,\mft+s]\rbr}{\vlh}}{r} \right|\ dz & \apprle  \frac{1}{r^n s} \iint_{\lbr \RR^n \times [\mft-s,\mft+s]\rbr} |\vh(z)| \ dz +  o(1)
    \apprle o(1).
  \end{array}
 \end{equation*}

\end{description}

\end{description}

This completes the proof of the Lipschitz continuity. 
\end{proof}
 
\subsection{Crucial estimates for the test function}
In this subsection, we shall prove three crucial estimates that will be needed. 

\begin{lemma}
\label{cruc_1}
Let $\la \geq 1$, then for any $i \in \NN$, $\de \in (0,1]$ and a.e. $t \in (\mft-s,\mft+s)$, there exists a constant $C = C_{(\plog,\lamot,n)}$ such that there holds
\begin{equation}
\label{cruc_est_1}
\abs{\int_{\Om_{4\rho}^{\al}(\mfx)} \lbr v(x,t) - v^i \rbr \vlh(x,t) \psi_i(x,t)\ dx} \leq C \lbr \frac{\la}{\de} |Q_i| + \de  |\htb_i|\fiint_{\htq_i} |v(z)|^2\lsb{\chi}{[\mft-s,\mft+s]}\ dz\rbr.
\end{equation}
\end{lemma}


\begin{proof}
Let us fix any $t \in (-s,s]$,  $i \in \NN$  and take $ \om_i(y,\tau) \vlh(y,\tau)$ as a test function in \eqref{basic_pde} and \eqref{wapprox_int}. Further integrating the resulting expression over $ \left(t_i - \scalet{\la}{z_i}4r_i^2 , t\right)$ along with making use of  the fact that $\psi_i(y,t_i - \scalet{\la}{z_i} 4r_i^2) = 0$, we get for  any $a\in \RR$, the equality
%
\begin{equation}
 \label{3.323}
 \begin{array}{ll}
  \int_{\Om_{4\rho}^{\al}(\mfx)} \lbr[(]  (\vh - a)  \psi_i \vlh \rbr (y,t) \ dy & = \int_{t_i - \max\{\scalet{\la}{z_i} 4r_i^2,-s\}}^t \int_{\Om_{4\rho}^{\al}(\mfx)} \pa_t \left(  (\vh - a) \psi_i \vlh \right) (y,\tau) \ dy \ d\tau \\
  & = \int_{t_i - \max\{\scalet{\la}{z_i} 4r_i^2,-s\}}^t \int_{\Om_{4\rho}^{\al}(\mfx)} \pa_t \left(  [u-w]_h\psi_i \vlh - a \psi_i \vlh \right) (y,\tau) \ dy \ d\tau \\
  & = \int_{t_i - \max\{\scalet{\la}{z_i} 4r_i^2,-s\}	}^t \int_{\Om} \iprod{[\aa(y,\tau,\nabla w)]_h - [\aa(y,\tau,\nabla u)]_h}{\nabla ( \psi_i \vlh)} \ dy \ d\tau  \\
  & \qquad +\  \int_{t_i - \max\{\scalet{\la}{z_i} 4r_i^2,-s\}}^t \int_{\Om}   [|\bff|^{\pp-2} \bff]_h  \ \nabla\lbr  \psi_i \vlh \right) (y,\tau) \ dy \ d\tau  \\
  & \qquad -\  \int_{t_i - \max\{\scalet{\la}{z_i} 4r_i^2,-s\}}^t \int_{\Om} a \pa_t \lbr \psi_i \vlh\rbr \ dy \ d\tau.
 \end{array}
\end{equation}

We can estimate $|\nabla ( \psi_i \vlh)|$ using the chain rule and Lemma \ref{partition_unity}, to get
\begin{equation}
 \label{3.326}
 \begin{array}{ll}
 |\nabla (\psi_i \vlh)| 
 & \apprle \frac{1}{\scalex{\la}{z_i}r_i} |\vlh| + |\nabla \vl|.
 \end{array}
\end{equation}
Similarly, we can estimate $\left|\pa_t\lbr \psi_i \vl \right)\right|$ using the chain rule, to get
 \begin{align*}
 \left| \pa_t \lbr \psi_i \vlh\rbr\right| & \apprle  \frac{1}{\scalet{\la}{z_i} r_i^2} |\vl| + |\pa_t \vl|.\label{3.328}
\end{align*}
%
%
Let us now prove each of the assertions of the lemma.
\begin{description}[leftmargin=*]
 \item[Proof of \eqref{cruc_est_1}:] Let us take  $a=\vh^i$ in the \eqref{3.323} followed by letting $h \searrow 0$ and making use of \eqref{3.326}, \eqref{abounded} and \eqref{bbounded},  we get
 \begin{equation*}
  \label{first_1}
  \begin{array}{ll}
   \left| \int_{\Om_{4\rho}^{\al}(\mfx)} \lbr[(] (v- v^i) \om_i \vl \rbr (y,t) \ dy \right| & \apprle J_1 + J_2 + J_3,
  \end{array}
 \end{equation*}
 where we have set 
 \begin{align}
  J_1& := \frac{1}{\scalex{\la}{z_i}r_i} \iint_{K_{4\rho}^{\al}(\mfz)} \lbr |\nabla u|+|\nabla w|+|\bff|+1\rbr^{p(z)-1}  |\vl|   \lsb{\chi}{2Q_i\cap K_{4\rho}^{\al}(\mfz)} \ dz, \nonumber \\
  J_2& :=   \iint_{K_{4\rho}^{\al}(\mfz)} \lbr |\nabla u|+|\nabla w|+|\bff|+1\rbr^{p(z)-1}  |\nabla \vl|   \lsb{\chi}{2Q_i\cap K_{4\rho}^{\al}(\mfz)} \ dz, \nonumber \\
  J_3&:= \iint_{K_{4\rho}^{\al}(\mfz)} |v-v^i| | \pa_t (\psi_i \vl)|  \lsb{\chi}{2Q_i\cap K_{4\rho}^{\al}(\mfz)} \ dz. \nonumber
 \end{align}

 Let us now estimate each of the terms as follows: 
 \begin{description}[leftmargin=*]
  \item[Bound for $J_1$:] We split the estimate into two cases, the first is when $\scalex{\al}{\mfz}\rho \leq \scalex{\la}{z_i} r_i$. In this case, we make use of \eqref{lemma6.7-1_est} along with \eqref{elambda} to get
  \[
  \begin{array}{rcl}
  J_1 &\apprle& \frac{\scalex{\al}{\mfz}\rho \la^{\frac{1}{p(z_i)}}}{\scalex{\la}{z_i} r_i} |Q_i| \fiint_{2Q_i} \lbr |\nabla u|+|\nabla w|+|\bff|+1\rbr^{p(z)-1}  \ dz  \\
  & \apprle & \frac{\scalex{\al}{\mfz}\rho \la^{\frac{1}{p(z_i)}}}{\scalex{\la}{z_i} r_i} |Q_i| \lbr \fiint_{2Q_i} \lbr |\nabla u|+|\nabla w|+|\bff|+1\rbr^{\frac{\pp}{\mathfrak{q}}}  \ dz\rbr^{\frac{p^+_{2Q_i} -1}{p^-_{2Q_i}}}\\
  & \apprle & \frac{\scalex{\al}{\mfz}\rho \la^{\frac{1}{p(z_i)}}}{\scalex{\al}{\mfz}\rho} |Q_i| \la^{\frac{p^+_{2Q_i} -1}{p^-_{2Q_i}}}\\
  & \apprle & |Q_i| \la \leq \frac{\la}{\de} |2Q_i|.
  \end{array}
  \]
  To obtain the last inequality, we have used $\la^{\frac{1}{p(z_i)} + \frac{p^+_{2Q_i}}{p^-_{2Q_i}}-\frac{1}{p^-_{2Q_i}}-1} \leq C_{(\plog,n)}$.

  In the case  $ \scalex{\al}{\mfz} \rho \geq \scalex{\la}{z_i} r_i$, we get for any $\de \in(0,1]$ using \eqref{bound_6.7-3-1}
  \[
  \begin{array}{rcl}
  J_1 & \apprle & \lbr \frac{\la^{\frac{1}{p(z_i)}}}{\de}  + \frac{\de}{\lbr \scalex{\la}{z_i} r_i \rbr^2 \la^{\frac{1}{p(z_i)}} } \fiint_{\htq_i} |v(z)|^2\lsb{\chi}{[\mft-s,\mft+s]} \ dz \rbr |Q_i| \fiint_{2Q_i} \lbr |\nabla u|+|\nabla w|+|\bff|+1\rbr^{p(z)-1}  \ dz  \\
  & \apprle & |Q_i| \lbr \frac{\la^{\frac{1}{p(z_i)}}}{\de}  + \frac{\de}{\lbr \scalex{\la}{z_i} r_i \rbr^2 \la^{\frac{1}{p(z_i)}} } \fiint_{\htq_i} |v(z)|^2\lsb{\chi}{[\mft-s,\mft+s]} \ dz \rbr  \la^{\frac{p^+_{2Q_i} -1}{p^-_{2Q_i}}}\\
  & \apprle & |Q_i| \frac{\la \la^{\frac{1}{p(z_i)} -1 + \frac{p^+_{2Q_i} -1}{p^-_{2Q_i}}}}{\de} + \de |B_i|  \frac{\scalet{\la}{z_i} r_i^2\la^{\frac{p^+_{2Q_i} -1}{p^-_{2Q_i}}} }{\lbr \scalex{\la}{z_i} r_i \rbr^2 \la^{\frac{1}{p(z_i)}}}\fiint_{\htq_i} |v(z)|^2\lsb{\chi}{[\mft-s,\mft+s]} \ dz\\ 
  & \apprle & \frac{\la}{\de} |Q_i| + \de |\hat{c}B_i|  \fiint_{\htq_i} |v(z)|^2\lsb{\chi}{[\mft-s,\mft+s]} \ dz.
  \end{array}
  \]
To obtain the last inequality, we again made use of $\la^{\frac{1}{p(z_i)} + \frac{p^+_{2Q_i}}{p^-_{2Q_i}}-\frac{1}{p^-_{2Q_i}}-1} \leq C_{(\plog,n)}$.

\begin{equation*}
 \label{bound_I_1}
 J_1 \apprle  \frac{\la}{\de} |\htq_i| +  \de |\htb_i| \fiint_{\htq_i} |v(z)|^2\lsb{\chi}{[\mft-s,\mft+s]} \ dz.
\end{equation*}

  \item[Bound for $J_2$:] In this case, we can directly use \eqref{bound+6.31_two} to get for any $\de \in (0,1]$, the bound
  \begin{equation*}
   \label{bound_I_22}
\begin{array}{rcl}
    J_2 & \apprle & \frac{\la^{\frac{1}{p(z_i)}}}{\de} |Q_i| \fiint_{2Q_i}\lbr |\nabla u|+|\nabla w|+|\bff|+1\rbr^{p(z)-1}  \ dz  \\
    & \apprle & |Q_i| \frac{\la^{\frac{1}{p(z_i)}}}{\de}\la^{\frac{p^+_{2Q_i} -1}{p^-_{2Q_i}}}  \apprle |Q_i| \frac{\la}{\de}.
   \end{array}
  \end{equation*}
  To obtain the last inequality, we again made use of $\la^{\frac{1}{p(z_i)} + \frac{p^+_{2Q_i}}{p^-_{2Q_i}}-\frac{1}{p^-_{2Q_i}}-1} \leq C_{(\plog,n)}$.
  \item[Bound for $J_3$:] Recall that $\htr_i = \hat{c} r_i$ where $\hat{c}$ is from \descref{W4}{W4}. In this case, we make use of \eqref{bound+6.31} and \eqref{bound_time_vlh_two} to get
  \begin{equation}
  \label{6.75}
  \begin{array}{rcl}
     \left| \pa_t \lbr \psi_i \vl\rbr\right| & \apprle & \frac{1}{\scalet{\la}{z_i} r_i^2} \lbr  \frac{\scalex{\la}{z_i}r_i \la^{\frac{1}{p(z_i)}}}{\de} + \frac{\de}{\scalex{\la}{z_i}r_i \la^{\frac{1}{p(z_i)}}}\fiint_{\htq_i} |v|^2 \lsb{\chi}{[\mft-s,\mft+s]} \ dz \rbr+ \\
     & & \quad + \frac{\scalex{\la}{z_i} r_i}{\scalet{\la}{z_i} r_i^2} \la^{\frac{1}{p(z_i)}}\\
  \end{array}
  \end{equation}
  
  Now making use of Lemma \ref{improved_est}, we see that 
  \begin{equation}
  \label{6.76}
  \begin{array}{rcl}
  \iint_{K_{4\rho}^{\al}(\mfz)} |v-v^i|  \lsb{\chi}{2Q_i\cap K_{4\rho}^{\al}(\mfz)} \ dz & \apprle & |Q_i| \fiint_{\htq_i} \abs{v\lsb{\chi}{[\mft-s,\mft+s]} - \avgs{v\lsb{\chi}{[\mft-s,\mft+s]}}{\htq_i}} \ dz\\
  & \apprle & |Q_i| \lbr \scalex{\la}{z_i} r_i \rbr \la^{\frac{1}{p(z_i)}}.
  \end{array}
  \end{equation}
  
  Combining \eqref{6.75} and \eqref{6.76}, we get
  \begin{equation*}
  \begin{array}{rcl}
  J_3 &\apprle & \frac{\la}{\de} |Q_i| + \frac{\de |Q_i|}{\scalet{\la}{z_i}r_i^2} \fiint_{\htq_i} |v|^2 \lsb{\chi}{[\mft-s,\mft+s]} \ dz \\
  & \apprle & \frac{\la}{\de} |Q_i| + \de |\htb_i| \fiint_{\htq_i} |v|^2 \lsb{\chi}{[\mft-s,\mft+s]} \ dz.
  \end{array}
  \end{equation*}
 \end{description}
\end{description}
This completes the proof of the lemma. 
\end{proof}

\begin{lemma}
\label{cruc_3}
Let $\la \geq 1$, then for a.e. $t \in [\mft-s,\mft+s]$, there exists a constant $C = C_{(\plog,\lamot,n)}$ such that there holds
\begin{equation}
\label{cruc_est_3}
\int_{\Om_{4\rho}^{\al}(\mfx)\setminus \elam(t)} \lbr |v|^2 - |v- \vl|^2 \rbr \ dx \geq -C \la |\RR^{n+1} \setminus \elam|.
\end{equation}
\end{lemma}
\begin{proof}
 Let us fix any $t\in [\mft-s,\mft+s]$ and any point $x \in \Om_{4\rho}^{\al}(\mfx) \setminus \elam(t)$.  Now define
 \begin{equation*}
  \tTh := \left\{  i \in \Th: \spt(\psi_i) \cap \Om_{4\rho}^{\al}(\mfx) \times \{t\} \neq \emptyset, \  |v| + |\vl| \neq 0 \  \text{on}\ \spt(\psi_i)\cap (\Om_{4\rho}^{\al}(\mfx) \times \{t\}) \right\}.
 \end{equation*}

 If $i \neq \tTh$, then $v= \vl = 0$ on $\spt(\psi_i) \cap \Om_{4\rho}^{\al}(\mfx) \times \{t\}$, which  implies
 \[
  \int_{\spt(\psi_i) \cap {\Om_{4\rho}^{\al}(\mfx)} \times\{t\}} |u|^2 - |u - \vl|^2 \ dx = 0.
 \]
 Hence we only need to consider $i \in \tTh$. Noting that $\sum_{i \in \tTh} \psi_i(\cdot,t)  \equiv 1$ on $\RR^{n} \cap \elam(t)$, we can rewrite the left-hand side of \eqref{cruc_est_3} as 
 \begin{equation}
  \label{3324}
  \begin{array}{rcl}
  \int_{{\Om_{4\rho}^{\al}(\mfx)} \setminus \elam(t)} (|u|^2 - |v- \vl|^2)(x,t) \ dx &=&  \sum_{i \in \tTh} \int_{\Om_{4\rho}^{\al}(\mfx)}\psi_i \lbr |v|^2 - |v- \vl|^2 \rbr \ dx = J_1 - J_2.\\
  \end{array}
 \end{equation}
 where we have set
\begin{gather*}
 J_1:= \sum_{i \in \tTh} \int_{\Om_{4\rho}^{\al}(\mfx)}\psi_i \lbr |v^i|^2 + 2 \vl (v- v^i) \rbr \ dx, \qquad 
 J_2:= \sum_{i \in \tTh} \int_{\Om_{4\rho}^{\al}(\mfx)}\psi_i  |\vl - v^i|^2  \ dx. 
\end{gather*}
We shall now estimate each of the terms as follows:
\begin{description}[leftmargin=*]
 \item[Estimate of $J_1$:] Using \eqref{cruc_est_1}, we get
  \begin{equation}
   \label{I_1_1}
   J_1 \apprge \sum_{i \in \tTh} \int_{{\Om_{4\rho}^{\al}(\mfx)}} \om_i(z)  |v^i|^2 \ dz - \de \sum_{i \in \tTh} |\htb_i| |v^i|^2 - \sum_{i \in \tTh} \frac{\la}{\de} |\htq_i|.
  \end{equation}
  From \eqref{def_tuh}, we have $v^i = 0$ whenever $\spt(\psi_i) \nsubseteq  \Om_{4\rho}^{\al}(\mfx) \times [-s,\infty)$. Hence we only have to sum over all those $i \in \tTh_1$ for which $\spt(\psi_i) \subset  \Om_{4\rho}^{\al}(\mfx) \times [-s,\infty)$.  In this case, we make use of a suitable choice for $\de \in (0,1]$, and use \descref{W4}{W4}  to estimate \eqref{I_1_1} from below. We get
  \begin{equation}
 \label{bound_I1}
 J_1 \apprge  -\la |\RR^{n+1} \setminus \elam|. 
\end{equation}

\item[Estimate of $J_2$:] For any $x \in K_{4\rho}^{\al}(\mfz) \setminus \elam(t)$, we have from Lemma \ref{partition_unity} that $\sum_{j} \psi_j(x,t) = 1$, which gives
  \begin{equation}
   \label{I_2_1}
   \begin{array}{ll}
    \psi_i(z) |\vlh(z) - v^i|^2  
    & \apprle   \sum_{j \in I_i} |\psi_j(z)|^2  \lbr v^j - v^i\rbr^2 \apprle \min\{ \rho, \scalex{\la}{z_i} r_i\}^2 \la^{\frac{2}{p(z_i)}}. 
   \end{array}
  \end{equation}
To obtain \redref{3.104.a}{a} above, we made use of Corollary \ref{corollary3.7}  along with \descref{W3}{W3}.  Substituting \eqref{I_2_1} into the expression for $J_2$, we get
\begin{equation}
 \label{bound_I_2}
 \begin{array}{rcl}
J_2  &\apprle&  \sum_{i \in \tTh} |{\Om_{4\rho}^{\al}(\mfx)} \cap 2B_i| {\lbr \scalex{\la}{z_i} r_i\rbr^2} \la^{\frac{2}{p(z_i)}}  \\
& \apprle& \sum_{i \in \tTh} \frac{|Q_i|}{\scalet{\la}{z_i} r_i^2} {\lbr \scalex{\la}{z_i} r_i\rbr^2} \la^{\frac{2}{p(z_i)}}  \\
&\apprle & \la |\RR^{n+1} \setminus \elam|.
\end{array}
\end{equation}
\end{description}

Substituting \eqref{bound_I1} and \eqref{bound_I_2} into \eqref{3324}, the proof of the lemma follows.
\end{proof}



\section{The method of Lipschitz truncation - second difference estimate}
\label{lipschitz_truncation_B}

In Appendix \ref{lipschitz_truncation}, we constructed a suitable test function which was used to obtain a difference estimate between the weak solutions of \eqref{basic_pde} and \eqref{wapprox_int}. In this appendix, we will obtain an analogous Lipschitz truncation method that will be used as a test function to obtain difference estimate between the weak solutions of \eqref{wapprox_int} and \eqref{vapprox_bnd}. Most of the estimates follow exactly as in Appendix \ref{lipschitz_truncation} and hence we will only highlight the modifications needed.

Let us first note that the Lipschitz truncation is now constructed over the constant exponent $p(\mfz)$ which actually simplifies a lot of the estimates from Appendix \ref{lipschitz_truncation}. 
Let us denote 
\begin{equation*}
\label{def_s_B}
s:= \scalet{\al}{\mfz} (3\rho)^2.
\end{equation*}

Firstly, let us recall the modified Lemma \ref{lemma_crucial_2}:
\begin{lemma}
\label{lemma_crucial_2_B}
 For any  $h \in (0,2s)$ and let $\phi(x) \in C_c^{\infty}({\Om_{3\rho}^{\al}(\mfx)})$ and $\varphi(t) \in C^{\infty}(\mft-s,\infty)$ with $\varphi(\mft-s) = 0$ be a  non-negative function and $[w]_h,[v]_h$ be the Steklov average as defined in \eqref{stek1}. Then   the following estimate holds for any time interval $(t_1,t_2) \subset [\mft-s,\mft+s]$:
 \begin{equation}
  \label{lemma_crucial_2_est_B}
  \begin{array}{rcl}
  |\avgs{{[w-v]_h \varphi}}{\phi} (t_2) - \avgs{{[w-v]_h\varphi}}{\phi}(t_1)| & \leq & \|\nabla \phi\|_{L^{\infty}{({\Om_{3\rho}^{\al}(\mfz)})}} \|\varphi\|_{L^{\infty}(t_1,t_2)} \iint_{{\Om_{3\rho}^{\al}(\mfx)} \times (t_1,t_2)} \abs{\bbb(t,\nabla v) - \aa(z,\nabla w)} \ dz \\
   & &\qquad + \|\phi\|_{L^{\infty}{({\Om_{3\rho}^{\al}(\mfx)})}} \|\varphi'\|_{L^{\infty}(t_1,t_2)} \iint_{{\Om_{3\rho}^{\al}(\mfx)} \times (t_1,t_2)} |[w-v]_h| \ dz.
  \end{array}
 \end{equation}
\end{lemma}

\subsection{Construction of test function}

Let us denote the following functions:
\begin{gather*}
v(z) := w(z) - v(z) \txt{and} \vh(z) := [w-v]_h(z) \label{def_vh_B}.
\end{gather*}
where $[w-v]_h(z)$ denotes the usual Steklov average. It is easy to  see that $\vh \xrightarrow{h \searrow 0} v$. We also note that $v(z) = 0$ for $z \in \pa_p K_{3\rho}^{\al}(\mfz)$.  For some fixed $\mathfrak{q}$ such that $1 <\mathfrak{q}< \frac{p^-}{p^+-1}$, with $\mm$ as defined in \eqref{par_max}, let us now define
\begin{equation}
\label{def_g_B}
\begin{array}{ll}
g(z) & := \mm\lbr \lbr[[]\frac{|v|}{\scalex{\al}{\mfz}\rho} + |\nabla w| + |\nabla v|  + 1\rbr[]]^{\frac{p(\mfz)}{\mathfrak{q}}} \lsb{\chi}{K_{3\rho}^{\al}(\mfz)}\rbr^{\mathfrak{q}(1-\be)}.
\end{array}
\end{equation}

For a fixed $\la \geq 1$, let us define the \emph{good set} by 
\begin{equation}
\label{elambda_B}
\elam := \{ z \in \RR^{n+1} : g(z) \leq \la^{1-\be}\}.
\end{equation}

Since we are dealing with constant exponent $p(\mfz)$, we have the following Whitney-type covering lemma (see \cite[Chapter 3]{bogelein2013regularity} or \cite[Lemma 3.1]{diening2010existence} for the proof):
  \begin{lemma}
 \label{whitney_covering_B}
 There exists a Whitney covering $\{Q_i(z_i)\}$ of $\elam^c$ in the following sense:
 \begin{description}
  \descitem{(W6)}{W6} $Q_j(z_j) = B_j(x_j) \times I_j(t_j)$ where $B_j(x_j) = B_{\scalex{\la}{\mfz}r_j}(x_j)$ and $I_j(t_j) = (t_j - \scalet{\la}{\mfz} r_j^2, t_j + \scalet{\la}{\mfz} r_j^2)$. 
  \descitem{(W7)}{W7} $\bigcup_j  Q_j(z_j) = \elam^c$.
  \descitem{(W8)}{W8} for all $j \in \NN$, we have $8Q_j \subset \elam^c$ and $16Q_j \cap \elam \neq \emptyset$.
  \descitem{(W9)}{W9} if $Q_j \cap Q_k \neq \emptyset$, then $\frac1{c} r_k \leq r_j \leq c r_k$.
  \descitem{(W10)}{W10} $\sum_j \lsb{\chi}{8Q_j}(z) \leq c(n)$ for all $z \in \elam^c$.
  \end{description}
  Subordinate to this Whitney covering, we have an associated partition of unity denoted by $\{ \psi_j\} \in C_c^{\infty}(\RR^{n+1})$ such that the following holds:
  \begin{description}
  \descitem{(W11)}{W11} $\lsb{\chi}{Q_j} \leq \psi_j \leq \lsb{\chi}{2Q_j}$.
  \descitem{(W12)}{W12} $\|\psi_j\|_{\infty} + \scalex{\la}{\mfz}r_j \| \nabla \psi_j\|_{\infty} + \scalet{\la}{\mfz}r_j^2 \| \pa_t \psi_j\|_{\infty} \leq C$.
  \end{description}
  For a fixed $k \in \NN$, let us define 
  \begin{equation*}\label{Ak}A_k := \left\{ j \in \NN: \frac34Q_k \cap \frac34Q_j \neq \emptyset\right\},\end{equation*} then we have
  \begin{description}
  \descitem{(W13)}{W13} Let $i \in \NN$ be given, then $\sum_{j \in A_i} \psi_j(z) = 1$  for all $z \in 2Q_i$.
  \descitem{(W14)}{W14} Let $i \in \NN$ be given and let  $j \in A_i$, then $\max \{ |Q_j|, |Q_i|\} \leq C_{(n)} |Q_j \cap Q_i|.$
  \descitem{(W15)}{W15}  Let $i \in \NN$ be given and let  $j \in A_i$, then $ \max \{ |Q_j|, |Q_i|\} \leq \left|2Q_j \cap 2Q_i\right|.$
  \descitem{(W16)}{W16} For any $i \in \NN$, we have $\# A_i \leq c(n)$.
  \descitem{(W17)}{W17} Let $i \in \NN$ be given, then for any $j \in A_i$, we have $2Q_j \subset 8Q_i$.
 \end{description}
\end{lemma}

 \subsection{Construction of Lipschitz truncation function}

We shall also use the notation 
\begin{equation*}
\label{mcii_B}
\mci(i) := \{ j \in \NN : \spt(\psi_i) \cap \spt(\psi_j) \neq \emptyset\} \txt{and} \mci_z := \{ j \in \NN: z \in \spt(\psi_j) \}.
\end{equation*}

We are now ready to construct  the Lipschitz truncation function:
 \begin{equation}
 \label{lipschitz_function_B}
 \vlh(z) := \vh(z) - \sum_{i} \psi_i(z) \lbr \vh(z) - \vh^i\rbr,
 \end{equation}
  where we have defined
 \begin{equation*}
 \label{def_tuh_B}
 \vh^i := \left\{
 \begin{array}{ll}
 \fiint_{2Q_i} \vh(z)\lsb{\chi}{[\mft-s,\mft+s]} \ dz & \text{if} \ \ 2Q_i \subset \Om_{3\rho}^{\al}(\mfx) \times (\mft-s,\infty), \\
 0 & \text{else}. 
 \end{array}\right.
 \end{equation*}

  From construction in \eqref{lipschitz_function} and \eqref{def_tuh}, we see that 
 \begin{equation*}
 \spt(\vlh) \subset \Om_{3\rho}^{\al}(\mfx) \times (\mft-s,\infty).
 \end{equation*}

 We see that $\vlh$ has the right support for the test function and hence the rest of this section will be devoted to proving the Lipschitz regularity of $\vlh$ on $K_{3\rho}^{\al}(\mfx)$ as well as some useful estimates.

 \subsection{Some estimates on the test function}
 
 In this subsection, we will collect some useful estimates on the test function. The proofs of these estimates are very similar to the corresponding ones from Appendix \ref{lipschitz_truncation} (in fact simpler because we are dealing with the constant exponent $p(\mfz)$) and will be omitted. Let us first derive a useful estimate:
 \begin{equation}
 \label{aa_bb}
 \begin{array}{rcl}
 |\bbb(t,\nabla v) - \aa(z,\nabla w)| & = &  |\bbb(t,\nabla v) - \bb(z,\nabla w)| + |\bb(z,\nabla w) - \aa(z,\nabla w)|\\
 & \overset{\eqref{bbounded}}{\apprle} &  \lbr \mu^2 + |\nabla v|^2 \rbr^{\frac{p(\mfz) -1}{2}} + \lbr \mu^2 + |\nabla w|^2 \rbr^{\frac{p(\mfz) -1}{2}} + |\bb(z,\nabla w) - \aa(z,\nabla w)|\\
 & \overset{\eqref{def_bb}}{\apprle} &  \lbr \mu^2 + |\nabla v|^2 \rbr^{\frac{p(\mfz) -1}{2}} + \lbr \mu^2 + |\nabla w|^2 \rbr^{\frac{p(\mfz) -1}{2}} +  |\aa(z,\nabla w)| \lbr \mu^2 + |\nabla w|^2 \rbr^{\frac{p(\mfz) -p(z)}{2}}\\
 & \overset{\eqref{abounded}}{\apprle} &  \lbr \mu^2 + |\nabla v|^2 \rbr^{\frac{p(\mfz) -1}{2}} + \lbr \mu^2 + |\nabla w|^2 \rbr^{\frac{p(\mfz) -1}{2}} +   \lbr \mu^2 + |\nabla w|^2 \rbr^{\frac{p(\mfz) -p(z)}{2}+ \frac{p(z) -1}{2}}\\
 & \apprle & \lbr \mu^2 + |\nabla v|^2 \rbr^{\frac{p(\mfz) -1}{2}} + \lbr \mu^2 + |\nabla w|^2 \rbr^{\frac{p(\mfz) -1}{2}}.
 \end{array}
 \end{equation}
 The primary use of \eqref{aa_bb} would be needed to estimate the first term on the right hand side of \eqref{lemma_crucial_2_est_B}.

 \begin{lemma}
\label{lemma3.6_pre_B}
Let $ \mfz \in K_{3\rho}^{\al}(\mfz) \setminus \elam$, then from \descref{W1}{W1}, we have that $\mfz \in 2Q_i$ for some $i \in \mci_{\mfz}$. For any $1 \leq \tht \leq \frac{p^-}{\mathfrak{q}}$,  there holds
\begin{equation} \label{lemma3.6_pre_two_B}\begin{array}{l}
|\vh^i|^{\tht} \leq  \fiint_{2Q_i} |\vh(z)|^{\tht}\lsb{\chi}{[\mft-s,\mft+s]} \ dz  \apprle_{(\plog,n)} (\scalex{\al}{\mfz}\rho)^{\tht} \la^{\frac{\tht}{p(\mfz)}},  \\
\fiint_{2Q_i}|\nabla \vh(z)|^{\tht}\lsb{\chi}{[\mft-s,\mft+s]} \ dz \apprle_{(\plog,n)}  \la^{\frac{\tht}{p(\mfz)}}.
\end{array}\end{equation}
\end{lemma}

 \begin{corollary}
 \label{lemma3.6_B}
For any $z \in K_{3\rho}^{\al}(\mfz) \setminus \elam$, we have $z \in 2Q_i$ for some $i \in \mci_{z}$, then there holds
 \begin{equation*}
 |\vh(z)| \apprle_{(n,\plog,\lamot)}(\scalex{\al}{\mfz}\rho)  \la^{\frac{1}{p(\mfz)}}.
 \end{equation*}
 \end{corollary}
 

\begin{lemma}
\label{improved_est_B}
Let $2Q_i$ be a parabolic Whitney type cylinder, then for any $1 \leq \tht \leq \frac{p^-}{\mathfrak{q}}$, there holds
\begin{equation*}
\fiint_{2Q_i} |\vh(z)\lsb{\chi}{[\mft-s,\mft+s]}- \vh^i|^{\tht} \  dz \apprle_{(\plog,\lamot,n)} \min\left\{ \scalex{\al}{\mfz}\rho, \scalex{\la}{\mfz}r_i\right\}^{\tht} \la^{\frac{\tht}{p(\mfz)}}.
\end{equation*}
\end{lemma}
\begin{proof}
Let us consider the following two cases:
\begin{description}[leftmargin=*]
\item[Case $\scalex{\al}{\mfz}\rho \leq  \scalex{\la}{\mfz}r_i$:] This is very similar to \eqref{6.18}. 

\item[Case $\scalex{\al}{\mfz}\rho \geq  \scalex{\la}{\mfz}r_i$:] Applying Lemma \ref{lemma_crucial_2} with  $\mu \in C_c^{\infty}(2B_i)$ such that $|\mu(x)| \apprle \frac{1}{\lbr \scalex{\la}{\mfz} r_i\rbr^n}$ and $|\nabla \mu(x)| \apprle \frac{1}{\lbr \scalex{\la}{\mfz} r_i\rbr^{n+1}}$, we get
\begin{equation}
\label{6.19_B}
\begin{aligned}
\fiint_{2Q_i} |\vh(z)\lsb{\chi}{[\mft-s,\mft+s]}- \vh^i|^{\tht} \  dz &\leq \lbr \scalex{\la}{\mfz} r_i\rbr^{\tht} \fiint_{2Q_i} |\nabla \vh|^{\tht}\lsb{\chi}{[\mft-s,\mft+s]} \ d\tz \\
&\qquad + \sup_{t_1,t_2 \in 2I_i\cap [\mft-s,\mft+s]} |\avgs{\vh}{\mu}(t_2) - \avgs{\vh}{\mu}(t_1)|^{\tht}.\\
\end{aligned}
\end{equation}
The first term on the right of \eqref{6.19_B} can be estimated using \eqref{lemma3.6_pre_two_B} to get
\begin{equation*}
\label{est_J_1_B}
\lbr \scalex{\la}{\mfz} r_i\rbr^{\tht} \fiint_{2Q_i} |\nabla \vh|^{\tht}\lsb{\chi}{[\mft-s,\mft+s]} \ d\tz \apprle \lbr \scalex{\la}{\mfz} r_i\rbr^{\tht}    \la^{\frac{\tht}{p(\mfz)}}.
\end{equation*}

To estimate the second term on the right of \eqref{6.19_B}, we make use of Lemma \ref{lemma_crucial_2_B} with $\phi(x) = \mu(x)$ and $\varphi(t) \equiv 1$, we get
\begin{equation*}
\label{est_J_2_one_B}
\begin{array}{rcl}
|\avgs{\vh}{\mu}(t_2) - \avgs{\vh}{\mu}(t_1)| & \apprle & \frac{|2Q_i|}{\lbr \scalex{\la}{\mfz} r_i\rbr^{n+1}} \fiint_{2Q_i} |\bbb(t,\nabla v) - \aa(z,\nabla w)| \ dz \\
& \overset{\eqref{aa_bb}}{\apprle} & \frac{\scalet{\la}{\mfz}r_i^2}{\scalex{\la}{\mfz} r_i} \lbr \fiint_{16Q_i} (1 + |\nabla w|+|\nabla v|)^{\frac{p(\mfz)}{\mathfrak{q}}} \ d\tz \rbr^{\frac{q(p(\mfz) -1)}{p(\mfz)}}\\
& \overset{\eqref{elambda_B}}{\apprle} & \la^{-1 + \frac{1}{p(\mfz)} + \frac{d}{2}} r_i \la^{\frac{p(\mfz) -1}{p(\mfz)}}= \lbr \scalex{\la}{\mfz} r_i \rbr \la^{\frac{1}{p(\mfz)}}.
\end{array}
\end{equation*}

Thus combining \eqref{est_J_1} and \eqref{est_J_2} into \eqref{6.19}, we get
\[
\fiint_{2Q_i} |\vh(\tz)\lsb{\chi}{[\mft-s,\mft+s]}- \vh^i|^{\tht} \  d\tz \apprle_{(\plog,\lamot,n)} \lbr \scalex{\la}{\mfz} r_i\rbr^{\tht}    \la^{\frac{\tht}{p(\mfz)}}.
\]
\end{description}
This  proves the lemma. 
\end{proof}

\begin{corollary}
\label{corollary3.7_B}
For any $i \in \NN$ and any $j \in \mci_i$, there holds
\[
|\vh^i - \vh^j| \apprle_{(\plog,\lamot,n)} \min\left\{ \scalex{\al}{\mfz}\rho, \scalex{\la}{\mfz}r_i\right\} \la^{\frac{1}{p(\mfz)}}.
\]
\end{corollary}

 \subsection{Bounds on \texorpdfstring{$\vlh$ and $\nabla \vlh$}.}
 
 \begin{lemma}
\label{lemma6.7-1_B}
Let $Q_i$ be a parabolic Whitney type cylinder. Then for any $z \in 2Q_i$, we have the following bound:
\begin{equation*}
\label{lemma6.7-1_est_B}
\lbr\frac{1}{\scalex{\al}{\mfz}\rho} |\vlh(z)| + |\nabla \vlh(z)|\rbr \lsb{\chi}{[\mft-s,\mft+s]} \apprle_{(\plog,\lamot,n)}  \la^{\frac{1}{p(\mfz)}}.
\end{equation*}
\end{lemma}
%
%

 \begin{corollary}
 \label{corollary6.7-2_B}
Let $z \in K_{3\rho}^{\al}(\mfz) \setminus \elam$, then $z \in 2Q_i$ for some $i \in \NN$. Then there holds for any $\de \in (0,1]$, the estimates
 \begin{gather*}
 \frac{1}{\scalex{\la}{\mfz} r_i } |\vlh(z)| \apprle_{{(\plog,\lamot,n)}}\frac{\la^{\frac{1}{p(\mfz)}}}{\de} +  \frac{\de}{\lbr \scalex{\la}{\mfz}r_i\rbr^2 \la^{\frac{1}{p(\mfz)}}} |\vh^i|^2, \label{bound+6.31_B}\\
  |\nabla \vlh(z)| \apprle_{{(\plog,\lamot,n)}}\frac{\la^{\frac{1}{p(\mfz)}}}{\de}. \label{bound+6.31_two_B}
 \end{gather*}

 \end{corollary}

%

 \begin{lemma}
 \label{lemma6.7-3_B}
 Let $z \in K_{3\rho}^{\al}(\mfz) \setminus \elam$, then $z \in 2Q_i$ for some $i \in \NN$. Then there holds for any $\de \in (0,1]$, the estimates
 \begin{gather*}
 |\vlh(z)| \apprle_{(\plog,\lamot,n)} \frac{\scalex{\la}{\mfz} r_i\la^{\frac{1}{p(\mfz)}}}{\de} + \frac{\de}{\scalex{\la}{\mfz} r_i\la^{\frac{1}{p(\mfz)}}} \fiint_{\htq_i} |\vh(\tz)|^2 \ d\tz, \label{bound_6.7-3-1_B} \\
 |\nabla \vlh(z)| \apprle_{(\plog,\lamot,n)} \la^{\frac{1}{p(\mfz)}} + \frac{\de}{\lbr \scalex{\la}{\mfz} r_i\rbr^2\la^{\frac{1}{p(\mfz)}}} \fiint_{\htq_i} |uh(\tz)|^2 \ d\tz. \label{bound_6.7-3-2_B}
 \end{gather*}
 \end{lemma}

\subsection{Estimates on the time derivative of \texorpdfstring{$\vlh$}.}

 \begin{lemma}
 \label{time_vlh_B}
 Let $ z \in {K_{3\rho}^{\al}(\mfz)}$, then $z \in 2Q_i$ for some $i \in \NN$. We then have the following estimates for the time derivative of $\vlh$: 
 \begin{equation*}
 \label{bound_time_vlh_one_B}
 |\pa_t \vlh(\tz)| \apprle_{(\plog,\lamot,n)}  \frac{1}{ \scalet{\la}{\mfz} r_i^2} \fiint_{\tQ_i} |\vh(z)| \lsb{\chi}{[-s-s]} \ dz.
 \end{equation*}
We also have the improved estimate
\begin{equation*}
\label{bound_time_vlh_two_B}
 |\pa_t \vlh(\tz)| \apprle_{(\plog,\lamot,n)} \frac{1}{\scalet{\la}{\mfz} r_i^2} \la^{\frac{1}{p(\mfz)}} \min \left\{\scalex{\la}{\mfz}r_i,  \scalex{\al}{\mfz}\rho \right\}.
\end{equation*}

 \end{lemma}
%
%
%
%
%
%

 \subsection{Some important estimates for the test function}
 
 \begin{lemma}
 \label{lemma6.8_B}
 Let $Q_i$ be a Whitney-type parabolic cylinder for some $ i \in \NN$. Then for any $\vt \in [1,2]$, there holds
 \begin{equation*}
 \label{lemma6.8-one_B}
 \iint_{{K_{3\rho}^{\al}(\mfz)} \setminus \elam} |\vlh(z)|^{\vt} \ dz \apprle_{(\plog,\lamot,n)} \iint_{{K_{3\rho}^{\al}(\mfz)} \setminus \elam} |\vh(z)|^{\vt} \ dz.
 \end{equation*}
\end{lemma}
%

 \begin{lemma}
 \label{lemma6.8-1_B}
 Let $Q_i$ be a Whitney-type parabolic cylinder for some $i \in \NN$, then there holds
 \begin{equation*}
 \label{lemma6.8-two_B}
 \fiint_{2Q_i} |\vlh(z) - \tvh(z)| \ dz \apprle_{(\plog,\lamot,n)} \min\left\{ \scalex{\la}{\mfz} r_i, \scalex{\al}{\mfz}\rho \right\} \la^{\frac{1}{p(\mfz)}}.
 \end{equation*}

 \end{lemma}

 \begin{lemma}
 \label{lemma6.8-2_B}
 Let $Q_i$ be a Whitney-type parabolic cylinder for some $i \in \NN$, then there holds
 \begin{equation*}
 \label{lemma6.8-three_B}
 \begin{array}{rl}
 \iint_{K_{3\rho}^{\al}(\mfz) \setminus \elam } |\pa_t \vlh(z) \lbr \vlh(z) - \vh(z)\rbr|^{\vt} \ dz & \apprle_{(\plog,\lamot,n)} \la^{\vt} |\RR^{n+1} \setminus \elam|.
 \end{array}
 \end{equation*}
 \end{lemma}

\subsection{Lipschitz continuity}

\begin{lemma}
 \label{lipschitz_continuity_B}
 Let $\la \geq 1$, then for any $\tz \in \Om_{3\rho}^{\al}(\mfx) \times [\mft -s,\mft+s]$ and $r>$, there exists a constant $C>0$ independent of $\tz$ and $r$ such that
\begin{equation*}\label{da_pr_B}
 I_r(\tz) := \frac{1}{\abs{\Om_{3\rho}^{\al}(\mfx) \times [\mft -s,\mft+s]}} \iint_{\Om_{3\rho}^{\al}(\mfx) \times [\mft -s,\mft+s]} \abs{\frac{\vl(z) - \avgs{\vl}{\Om_{3\rho}^{\al}(\mfx) \times [\mft -s,\mft+s]}}{r}} \ dz \leq C < \infty.
 \end{equation*}
 In particular, this implies for any $z_1, z_2 \in \Om_{3\rho}^{\al}(\mfx) \times [\mft -s,\mft+s]$, there exists a constant $K>0$ such that
 \[
 |\vl(z_1) - \vl(z_2)| \leq K d_p(z_1,z_2).
 \]
 \end{lemma}

\subsection{Crucial estimates for the test function}
In this subsection, we shall prove three crucial estimates that will be needed. Note that by the time these estimates are applied, we would have taken $h \searrow 0$ in the Steklov average. 

\begin{lemma}
\label{cruc_1_B}
Let $\la \geq 1$, then for any $i \in \NN$, $\de \in (0,1]$ and a.e. $t \in (\mft-s,\mft+s)$, there exists a constant $C = C_{(\plog,\lamot,n)}$ such that there holds
\begin{equation*}
\label{cruc_est_1_B}
\abs{\int_{\Om_{3\rho}^{\al}(\mfx)} \lbr v(x,t) - v^i \rbr \vl(x,t) \psi_i(x,t)\ dx} \leq C \lbr \frac{\la}{\de} |Q_i| + \de  |B_i|\fiint_{2Q_i} |v(z)|^2\lsb{\chi}{[\mft-s,\mft+s]}\ dz\rbr.
\end{equation*}
\end{lemma}



\begin{lemma}
\label{cruc_3_B}
Let $\la \geq 1$, then for a.e. $t \in [\mft-s,\mft+s]$, there exists a constant $C = C_{(\plog,\lamot,n)}$ such that there holds
\begin{equation*}
\label{cruc_est_3_B}
\int_{\Om_{3\rho}^{\al}(\mfx)\setminus \elam(t)} \lbr |\tv|^2 - |v- \vl|^2 \rbr \ dx \geq -C \la |\RR^{n+1} \setminus \elam|.
\end{equation*}
\end{lemma}
\end{appendices}

\section*{References}

\end{document}